\documentclass[11pt]{article}

\usepackage[a4paper,margin=1in]{geometry}
\usepackage{graphicx}
\usepackage{amsthm}
\usepackage{amsmath}
\usepackage{amssymb}
\usepackage{algorithmic}
\usepackage[ruled,vlined]{algorithm2e}
\usepackage{subfig}
\usepackage{array}
\usepackage{multirow}
\usepackage{booktabs}
\usepackage{makecell}
\usepackage{xcolor}
\usepackage{comment}
\usepackage{placeins}
\usepackage{float}
\definecolor{myblue}{RGB}{0,80,160}

\theoremstyle{plain}
\newtheorem{theorem}{Theorem}
\newtheorem{proposition}{Proposition}
\newtheorem{lemma}{Lemma}
\newtheorem{corollary}{Corollary}
\theoremstyle{definition}
\newtheorem{definition}{Definition}
\newtheorem{assumption}{Assumption}
\theoremstyle{remark}
\newtheorem{remark}{Remark}

\begin{document}
	
\title{Restarted Reflected Halpern Acceleration\\
for Augmented Primal--Dual Methods}
\author{Benqi Liu\({}^{1}\) \quad Ju Cao\({}^{2}\) \quad
Wotao Yin\({}^{3}\) \quad
Zaiwen Wen\({}^{1}\)\thanks{Corresponding author:
\texttt{wenzw@pku.edu.cn}.}\\[0.5em]
\small \({}^{1}\)Beijing International Center for Mathematical Research,
Peking University, Beijing, China\\
\small \({}^{2}\)School of Mathematical Sciences, Peking University,
Beijing, China\\
\small \({}^{3}\)Alibaba US, DAMO Academy\\[0.3em]
\small \texttt{bqliu@pku.edu.cn}, \texttt{2300010612@stu.pku.edu.cn}\\
\small \texttt{wotao.yin@alibaba-inc.com}, \texttt{wenzw@pku.edu.cn}}
\date{}

\maketitle
	
\begin{abstract}
We study linearly constrained composite convex optimization with a smooth term
and a proximable nonsmooth term. We develop a unified augmented primal--dual
framework with primal--dual hybrid gradient-type and augmented
Chambolle--Pock-type metric choices, including a fully augmented
Chambolle--Pock-type family that retains the augmented quadratic term. The
exact scheme admits a degenerate proximal-point form; the linearized scheme
admits a preconditioned forward--backward form. These representations allow
reflected Halpern acceleration to be analyzed directly in primal--dual
variables. For the shadow iterates, we prove convergence to
Karush--Kuhn--Tucker (KKT) points and nonergodic \(O(1/k)\) bounds for the KKT
residual and objective gap, with a scalar worst-case example. We show that
finite identification belongs to the shadow sequence rather than to the
anchored Halpern state. After identification, an affine-face model yields an
exact reduced residual identity and a local-sharpness criterion. Finally, we
prove linear convergence of restart anchors under fixed-point sharpness on the
visited restart set, with local or tail convergence when sharpness follows from
local error bounds. Experiments on linear and convex quadratic programs
illustrate augmentation and linearization.
\end{abstract}


	
\section{Introduction}
\label{intro}
We study the linearly constrained composite convex optimization problem
\begin{equation}\label{P}
    \min_{x \in \mathbb{R}^n} \ \Phi(x) := f(x) + g(x)
    \quad \text{s.t.} \quad Ax = b,
\end{equation}
where \(A\in\mathbb{R}^{m\times n}\), \(b\in\mathbb{R}^m\),
\(f:\mathbb{R}^n\to\mathbb{R}\) is convex and continuously differentiable with
Lipschitz continuous gradient, and
\(g:\mathbb{R}^n\to(-\infty,+\infty]\) is proper, closed, convex, and
proximable. We assume throughout that the feasible set is nonempty and that
\eqref{P} has at least one optimal solution. All convergence statements to
Karush--Kuhn--Tucker (KKT) points are made under the additional assumption that
the KKT set is nonempty, or equivalently that the monotone inclusion associated
with the KKT mapping has at least one zero. This condition is automatic under standard constraint
qualifications; for example, it holds when a suitable relative-interior
condition for the affine constraint and the domain of \(g\) is satisfied. We
keep this requirement explicit because primal solvability alone need not
guarantee the existence of a Lagrange multiplier for the nonsmooth composite
model. This formulation includes, for
example, equality-form linear programs, regularized least-squares models with
linear side constraints, and box-constrained convex quadratic programs written
with explicit equality constraints.

The paper is concerned with first-order primal--dual methods built from the
augmented Lagrangian
\[
    L_{\sigma}(x,y)
    = f(x)+g(x)-\langle y,Ax-b\rangle
      +\frac{\sigma}{2}\|Ax-b\|^2,\qquad \sigma\ge0.
\]
Augmentation is attractive both analytically and computationally: it injects
curvature in the constraint directions and can enlarge the admissible
primal--dual parameter region. At the same time, it complicates the usual
complexity analysis. In particular, for equality-constrained augmented
Lagrangian formulations the dual domain is unbounded; consequently, saddle-point gap
estimates are not a satisfactory substitute for convergence guarantees stated
directly in KKT residuals and objective-value gaps. The central question of
this paper is therefore how to obtain global nonergodic KKT rates and restart
guarantees conditional on fixed-point sharpness on the visited restart set for a
broad augmented primal--dual family.

Two lines of work motivate our approach. The first is the augmented-Lagrangian
and proximal-point tradition. Rockafellar interpreted augmented Lagrangian
methods as proximal point algorithms
~\cite{rockafellar1976augmented,rockafellar1976ppa}; Eckstein and Bertsekas
then connected ADMM with Douglas--Rachford splitting~\cite{eckstein1992dr}.
See also the survey~\cite{boyd2011admm}. Semi-proximal and
generalized ADMM-type methods have subsequently become standard tools for
structured convex optimization and conic programming
~\cite{shapiro2004some,fazel2013hankel,xiao2018generalized}. The second line is
matrix-free primal--dual splitting. It includes the Chambolle--Pock/primal--dual
hybrid gradient (PDHG) method and related schemes for composite saddle-point
models~\cite{chambolle2011first,PDHG,esser2010framework,pock2011diagonal}. Further
developments treat more general primal--dual structures
~\cite{condat2013primaldual,combettes2014forwardbackward}. These splitting
methods have inexpensive iterations and are well suited for large-scale implementations, but their
classical convergence theory is usually formulated for saddle-point gaps or
averaged quantities.

Recent developments have begun to bridge these viewpoints. Zhu et
al.~\cite{zhu2023unified} proposed a unified augmented primal--dual framework
that includes updates of the PDHG, Chambolle--Pock, and OGDA types and
established ergodic \(O(1/N)\) guarantees without assuming bounded optimal
multipliers.
Degenerate preconditioned proximal point algorithms were studied systematically
in~\cite{DPPPA}. Building on this operator-theoretic perspective, Sun et
al.~\cite{sun2025accelerating} reformulated a class of preconditioned ADMM
schemes~\cite{xiao2018generalized} as degenerate proximal point mappings and
applied Halpern-type acceleration~\cite{halpern1967fixed,lieder2021convergence}
to obtain nonergodic \(O(1/k)\) KKT residual bounds. The related solver
HPR-LP~\cite{hpr} shows that this viewpoint can also be computationally useful.
These results show the power of the operator viewpoint, but they do not yet
provide a single treatment of exact and linearized augmented primal--dual
schemes, nor do they explain how reflected Halpern acceleration should be
restarted in the local regime.

The local high-accuracy regime is central for first-order methods. Modern
large-scale solvers for LP and QP, including PDLP and its
extensions~\cite{applegate2021practical,applegate2024infeasibility},
cuPDLP.jl and cuPDLPx~\cite{cupdlpjl2025,cupdlpx}, HPR-LP~\cite{hpr},
PDQP~\cite{pdqp}, HPR-QP~\cite{hprqp}, and PDHCG-II~\cite{pdhcgii}, rely on
restart, scaling, and local residual reduction to reach useful accuracy.
Theoretical progress in this direction is particularly mature for LP, where
sharpness-based restart analyses and refined PDHG geometry were established in
~\cite{applegate2023faster,lu2024rhpdhg,lu2025geometry}. Local identification
and error-bound theory provide a natural language for extending such ideas.
Classical results of Hoffman and Robinson
~\cite{hoffman1952approximate,robinson1981continuity}, together with the
active-manifold and partial-smoothness literature
~\cite{lewis2002active,hare2004identifying}, describe how residuals behave
after the active structure has stabilized. Recent work of D{\'\i}az et
al.~\cite{diaz2026active} further clarifies active-set identification and rapid
local convergence for degenerate primal--dual trajectories. We connect these
local geometric ideas with augmented primal--dual reflected
Halpern schemes for the composite model~\eqref{P}. The analysis proceeds
through degenerate preconditioned proximal-point (dPPM) and preconditioned
forward--backward splitting (PFBS) representations, reflected Halpern KKT
rates, shadow identification, and restart under fixed-point sharpness on the
visited restart set.

The main contributions are as follows.
\begin{itemize}
\item We introduce a unified augmented primal--dual metric template for
\eqref{P}. The template contains PDHG-type and CP-type metric choices; their
usual explicit variants arise after the linearization in Section~\ref{sec3},
when applicable. It also includes a fully augmented CP-type subfamily, denoted
FA-CP, whose primal subproblem retains the augmented quadratic term.
\item We prove exact operator representations for the resulting maps. The exact
augmented scheme is a degenerate preconditioned proximal-point map, while the
linearized scheme is a preconditioned forward--backward map with an explicit
reflected-relaxation range. These representations fix the metric,
residual map, and shadow point used throughout the analysis.
\item We convert reflected Halpern fixed-point estimates into optimization
guarantees for \eqref{P}. For the shadow iterates we prove convergence to KKT
points and nonergodic \(O(1/k)\) bounds for both the KKT residual and the
objective gap, and we give a scalar example showing that the global residual
rate is worst-case tight.
\item We identify the sequence to which local geometry applies in reflected
trajectories. Finite identification is proved for the shadow sequence, not for
the anchored Halpern state. Under an additional affine-face hypothesis, the
identified dynamics admit an exact reduced residual identity and a reduced
sharpness criterion.
\item We prove linear convergence of restart anchors for the reflected
augmented maps under fixed-point sharpness on the visited restart set. Thus the
conclusion is global when this fixed-point sharpness holds globally, and local
or tail-only when sharpness follows from local metric subregularity or
Hoffman--Robinson error bounds.
\end{itemize}

Throughout the theoretical sections, the exact augmented maps are treated as
exact subproblem/resolvent maps. The linearized variants are the cases in which
the primal step becomes explicit or reduces to a simple proximal step for
common metric choices. In the numerical QP experiments, the subproblem-based
variants are implemented with the same inexact inner solver as the baseline and
are therefore treated as practical inexact realizations of the exact maps.

Several ingredients used in the analysis are standard, including abstract
Halpern residual estimates, averagedness of forward--backward maps, partial
smoothness, and sharpness-based restart. The contribution is the way these
tools are connected to the augmented primal--dual maps for \eqref{P}. Relative
to the unified augmented primal--dual framework of Zhu
et~al.~\cite{zhu2023unified}, we give exact dPPM and linearized PFBS
representations for the metric template studied here, including the
positive-\(\sigma\) FA-CP subfamily, and use these representations to define the
residual map and shadow point used in the analysis. Relative to the accelerated
dPPM analysis of Sun et~al.~\cite{sun2025accelerating}, we translate reflected
Halpern estimates into nonergodic KKT residual and objective-gap bounds for the
shadow iterates of \eqref{P}, and identify the shadow sequence as the
finite-identification object for reflected trajectories. Relative to LP
sharpness-restart analyses~\cite{applegate2023faster,lu2024rhpdhg,lu2025geometry},
we prove a conditional restart theory for reflected augmented primal--dual maps
and show how local metric subregularity or Hoffman--Robinson error bounds
transfer to fixed-point sharpness on the visited restart set.

The assumptions are layered according to the role of each result. Algebraic
identities, global convergence, local identification, and restart require
different metric or locality conditions, which are stated in the corresponding
results. In particular, restart is global only under global fixed-point
sharpness on the visited restart set; when sharpness is supplied by local
metric subregularity or Hoffman--Robinson error bounds, the conclusion is local
or tail linear convergence of restart anchors. A compact dependency map for the
main results is given in Appendix~\ref{app:assumption-map}.

The rest of the paper is organized as follows. Section~\ref{sec2} studies the
exact augmented primal--dual scheme from the proximal-point viewpoint.
Section~\ref{sec3} develops the forward--backward view of the linearized
scheme. Section~\ref{sec4} studies shadow identification and reduced residuals.
Section~\ref{sec5} establishes restart guarantees under fixed-point sharpness on
the visited restart set.
Section~\ref{sec6} presents the numerical experiments, and
Section~\ref{sec7} concludes the paper.

\section{A Proximal-Point View of Augmented Primal--Dual Methods}
\label{sec2}

We first consider the exact augmented Lagrangian primal--dual scheme. At this
stage the two subproblems are understood through their
first-order optimality conditions; the regularity assumptions that make the
induced resolvent single-valued and globally Lipschitz continuous are imposed
in Section~\ref{sec2:2.2}. Motivated by the reflected primal--dual hybrid
gradient update, in which the multiplier is evaluated at the extrapolated
primal point \(2x^{k+1}-x^k\), we study the following augmented primal--dual
scheme:
\begin{equation}\label{unified PD form}
    \begin{aligned}
        x^{k+1} &= \arg \min_{x} \left\{L_{\sigma}(x,y^{k}) + \frac{1}{2}\|x-x^{k}\|_{P}^{2}\right\}, \\
        y^{k+1} &= \arg \max_{y} \left\{L_{\sigma}(2x^{k+1}-x^{k},y) - \frac{1}{2}\|y-y^{k}\|_{Q}^{2}\right\}.
    \end{aligned}
\end{equation}
Here the primal step is performed on the augmented Lagrangian, while the dual
step retains the reflected primal argument. Different choices of the metric
operators $P$ and $Q$ then lead to different augmented primal--dual
realizations, including PDHG-type and augmented Chambolle--Pock-type schemes.
Although these methods have different explicit forms, they are generated by a
single degenerate preconditioned proximal-point map. This representation will
be used to apply reflected Halpern acceleration and convert abstract fixed-point
residual estimates into nonergodic $O(1/k)$ bounds for the KKT residual and
objective-value gap of problem~\eqref{P}.

We begin by recalling the degenerate preconditioned proximal-point
framework and then identify \eqref{unified PD form} as an instance of it.
For readability, Table~\ref{tab:notation} collects the notation used throughout
the paper.
\begin{table}[!t]
\centering
\caption{Main notation used throughout the paper.}
\label{tab:notation}
\footnotesize
\begin{tabular}{@{}>{\raggedright\arraybackslash}p{0.24\textwidth}
                  >{\raggedright\arraybackslash}p{0.68\textwidth}@{}}
\toprule
Symbol & Meaning \\
\midrule
\(w=(x,y)\) & primal--dual variable; \(w^k\) denotes the anchored Halpern state. \\
\(\Phi=f+g\) & objective function in problem~\eqref{P}. \\
\(\mathcal R\) & original KKT mapping, defined in \eqref{eq:kkt_mapping_residual}. \\
\(\mathcal T\) & augmented KKT operator in \eqref{T_AL}, equivalently \(\mathcal R+\mathcal C_\sigma\). \\
\(\mathcal C_\sigma\) & augmented correction
\(\mathcal C_\sigma(x,y)=(\sigma A^\top(Ax-b),0)\), introduced formally in Section~\ref{sec4}. \\
\(\mathcal E_\sigma\) & linear residual map \(\mathcal E_\sigma(u,v)=(u-\sigma A^\top v,v)\). \\
\(\mathcal M\) & primal--dual metric/preconditioner in \eqref{M_AL}. \\
\(\widehat{\mathcal T}\) & exact dPPM fixed-point map \((\mathcal M+\mathcal T)^{-1}\mathcal M\). \\
\(\widehat{\mathcal T}_{\rm lin}\) & linearized PFBS fixed-point map. \\
\(F\) & generic fixed-point map, either \(\widehat{\mathcal T}\) or \(\widehat{\mathcal T}_{\rm lin}\). \\
\(\hat w^k\) & shadow point, i.e., the fixed-point image of \(w^k\). \\
\(\bar w^k\) & reflected point \((1+\gamma)\hat w^k-\gamma w^k\). \\
\(\mathcal S_\gamma\) & reflected fixed-point map \((1+\gamma)F-\gamma I\). \\
\(r_{\mathrm{KKT}}\) & original KKT residual \(\mathrm{dist}(0,\mathcal R(\cdot))\). \\
\(r_F\) & fixed-point residual \(\|w-F(w)\|_{\mathcal M}\). \\
\(W^\star\) & KKT solution set, \(W^\star=\mathcal T^{-1}(0)=\mathcal R^{-1}(0)\). \\
\(\mu_x\) & Schur-complement metric constant, e.g.,
\(\lambda_{\min}(P-A^\top Q^{-1}A)\) in the positive definite case. \\
\bottomrule
\end{tabular}
\end{table}

\subsection{A Preconditioned Proximal Point Reformulation}
\label{sec2:2.1}

Let $\mathcal{H}$ be a real Hilbert space, and consider the monotone inclusion
problem
\begin{equation}\label{0inT}
    \text{find } w\in\mathcal{H}
    \quad \text{such that} \quad
    0\in \mathcal{T}w.
\end{equation}
Given a bounded, self-adjoint, and positive semidefinite linear operator
$\mathcal{M}:\mathcal{H}\to\mathcal{H}$, the associated preconditioned proximal
point iteration for solving \eqref{0inT} is
\begin{equation}\label{PPP}
    w^{k+1} = (\mathcal{M}+\mathcal{T})^{-1}\mathcal{M}w^k.
\end{equation}
When $\mathcal{M}$ is only positive semidefinite, \eqref{PPP} is often called a
degenerate preconditioned proximal point method (dPPM). In this subsection, we
use only the structural assumptions needed for the algebraic reformulation:
$P$ and $Q$ are assumed to be self-adjoint and positive semidefinite. Stronger
conditions ensuring admissibility and regularity of the associated resolvent
will be introduced in the next subsection.

We show that the augmented Lagrangian primal--dual scheme
\eqref{unified PD form} admits an exact proximal-point-type reformulation.
To this end, let $w=(x,y)$ and define
\begin{equation}\label{M_AL}
    \mathcal{M}
    =
    \begin{pmatrix}
        P & A^\top \\
        A & Q
    \end{pmatrix},
\end{equation}
and
\begin{equation}\label{T_AL}
    \mathcal{T}(x,y)
    =
    \begin{pmatrix}
        \partial \Phi(x) - A^\top y + \sigma A^\top(Ax-b) \\
        Ax-b
    \end{pmatrix}.
\end{equation}

\begin{proposition}\label{prop:dppm_representation}
With $\mathcal{M}$ and $\mathcal{T}$ defined in
\eqref{M_AL}--\eqref{T_AL}, the optimality conditions of
\eqref{unified PD form} are equivalent to
\[
    \mathcal M w^k \in \mathcal M w^{k+1}+\mathcal T(w^{k+1}),
    \qquad w^{k+1}=(x^{k+1},y^{k+1}).
\]
Consequently, whenever the corresponding resolvent is well defined and
single-valued at $\mathcal M w^k$, the update is exactly given by \eqref{PPP}.
If, in addition, $\mathcal M\succeq0$, this is a degenerate preconditioned
proximal point representation.
\end{proposition}

\begin{proof}
The displayed inclusion is, by \eqref{M_AL}--\eqref{T_AL}, equivalent to
\[
\begin{cases}
Px^k + A^\top y^k
    \in
    Px^{k+1} + \partial\Phi(x^{k+1}) + \sigma A^\top(Ax^{k+1}-b), \\[0.4em]
Ax^k + Qy^k
    =
    2Ax^{k+1} + Qy^{k+1} - b.
\end{cases}
\]
Equivalently,
\[
\begin{cases}
0 \in \partial\Phi(x^{k+1}) - A^\top y^k
    + \sigma A^\top(Ax^{k+1}-b)
    + P(x^{k+1}-x^k), \\[0.4em]
0 = -A(2x^{k+1}-x^k) + b - Q(y^{k+1}-y^k).
\end{cases}
\]
The first relation is precisely the optimality condition of the $x$-subproblem
in \eqref{unified PD form}, while the second relation is exactly the optimality
condition of the $y$-subproblem. Therefore \eqref{unified PD form} and the
corresponding instance of \eqref{PPP} are equivalent.
\end{proof}

Proposition~\ref{prop:dppm_representation} identifies \eqref{unified PD form}
with the resolvent form underlying dPPM. Different metric choices therefore
produce exact specializations of the same operator-theoretic template. We use
the following naming convention. With \(\sigma=0\) and diagonal Euclidean
metrics, \eqref{unified PD form} reduces to the primal--dual hybrid gradient
method (PDHG) for the equality-constrained saddle formulation. With
\(\sigma\ge0\) and the shifted primal metric
\(P=\tau^{-1}I_n-\sigma A^\top A\), the augmented quadratic term is linearized
in the primal subproblem; this gives the augmented Chambolle--Pock scheme
(CP-AL) studied in~\cite{zhu2023unified}. Finally, if the primal metric remains
\(\tau^{-1}I_n\), then the augmented quadratic term is retained explicitly in
the primal minimization. We call this subproblem-based metric choice the fully
augmented Chambolle--Pock scheme (FA-CP). Table~\ref{tab:exact_metric_choices}
summarizes the corresponding metric choices.
\begin{table}[htbp]
\centering
\small
\caption{Metric choices for the exact augmented primal--dual scheme.}
\label{tab:exact_metric_choices}
\begin{tabular*}{\textwidth}{@{\extracolsep{\fill}}llll@{}}
\toprule
Scheme & $\sigma$ & $P$ & $Q$ \\
\midrule
\makecell[l]{Primal--dual hybrid\\gradient (PDHG)~\cite{PDHG}}
    & $\sigma=0$
    & $\tau^{-1}I_n$
    & $\rho^{-1}I_m$ \\
\makecell[l]{Augmented\\Chambolle--Pock (CP-AL)~\cite{zhu2023unified}}
    & $\sigma\ge0$
    & $\tau^{-1}I_n-\sigma A^\top A$
    & $\rho^{-1}I_m$ \\
\makecell[l]{Fully augmented\\Chambolle--Pock (FA-CP)}
    & $\sigma\ge0$
    & $\tau^{-1}I_n$
    & $\rho^{-1}I_m$ \\
\bottomrule
\end{tabular*}
\end{table}
The \(\sigma=0\) member of FA-CP coincides with the PDHG-type corner, while its
positive-\(\sigma\) members are fully augmented. Thus FA-CP differs from
CP-AL in the primal metric: CP-AL cancels the added quadratic through a shifted
metric, whereas FA-CP keeps the augmented quadratic in the primal
minimization. To the best of our knowledge, the positive-\(\sigma\) members of
this fully augmented CP-type family have not been analyzed for the composite
equality-constrained model~\eqref{P} through the exact dPPM/PFBS and
reflected-Halpern residual framework developed here. In all
three cases the dual update has the common extrapolated form
$y^{k+1}=y^k-\rho(A(2x^{k+1}-x^k)-b)$.

\subsection{Admissibility and Resolvent Regularity}
\label{sec2:2.2}

Section~\ref{sec2:2.1} identifies \eqref{unified PD form} with a
proximal-point-type resolvent representation. We impose verifiable
conditions under which the associated
resolvent is well defined and globally Lipschitz continuous. These conditions
also guarantee that the metric operator $\mathcal M$ is an admissible
preconditioner for $\mathcal T$ in the sense recalled below.

We first recall the notion of an admissible preconditioner.
\begin{definition}[Admissible preconditioner] \label{admissible}
Let $\mathcal{T}:\mathcal{H}\rightrightarrows\mathcal{H}$ be a set-valued operator.
A linear, bounded, self-adjoint, and positive semidefinite operator
$\mathcal{M}:\mathcal{H}\to\mathcal{H}$ is called an \emph{admissible
preconditioner} for $\mathcal{T}$ if the mapping
$\widehat{\mathcal{T}} := (\mathcal{M}+\mathcal{T})^{-1}\mathcal{M}$ is
single-valued and has full domain.
\end{definition}

We specialize to the operator pair associated with the augmented
Lagrangian primal--dual scheme.

\begin{assumption}\label{ass:regularity}
Let $(\mathcal{M},\mathcal{T})$ be the operator pair defined in
\eqref{M_AL}--\eqref{T_AL}. Assume that
\begin{enumerate}
    \item[(i)] $Q$ is self-adjoint and positive definite;
    \item[(ii)] $\mathcal M\succeq 0$;
    \item[(iii)] the operator
    $G:=\partial\Phi+P+\sigma A^\top A$ has a single-valued inverse
    $G^{-1}:\mathbb R^n\to\mathbb R^n$ that is globally Lipschitz continuous.
\end{enumerate}
\end{assumption}

Assumption~\ref{ass:regularity} is readily verified from the primal metric.
Condition~(iii) is the additional primal regularity used below to solve the
resolvent. A convenient sufficient condition is
\[
        P+\sigma A^\top A\succeq \mu I
        \qquad\text{for some }\mu>0.
\]
Indeed, then $G=\partial\psi$, where
\[
        \psi(x):=\Phi(x)+\frac12\langle x,(P+\sigma A^\top A)x\rangle .
\]
The function $\psi$ is $\mu$-strongly convex. Therefore, for every
$\xi\in\mathbb R^n$, the function
$x\mapsto \psi(x)-\langle \xi,x\rangle$ has a unique minimizer; equivalently,
$\xi\in Gx$ for a unique $x$ and $G^{-1}$ has full domain. If
$p_i\in Gx_i$ for $i=1,2$, strong monotonicity gives
\[
        \mu\|x_1-x_2\|^2
        \le \langle p_1-p_2,x_1-x_2\rangle
        \le \|p_1-p_2\|\,\|x_1-x_2\|,
\]
and therefore $\|G^{-1}p_1-G^{-1}p_2\|\le \mu^{-1}\|p_1-p_2\|$.
Thus Assumption~\ref{ass:regularity} holds whenever the primal metric plus the
augmented quadratic term is uniformly positive definite; for the metric choices
in Table~\ref{tab:exact_metric_choices}, this reduces to the familiar
step-size restrictions discussed below.

\begin{remark}[Role of the primal regularity condition]
Assumption~\ref{ass:regularity}(iii) is used only to ensure
single-valuedness and Lipschitz regularity of the resolvent. For the PDHG,
CP-AL, and FA-CP metric choices in Table~\ref{tab:exact_metric_choices}, it is
automatically satisfied whenever the sufficient condition
\(P+\sigma A^\top A\succeq \mu I\) displayed above holds. The dPPM algebraic
identity in Proposition~\ref{prop:dppm_representation} itself does not require
this regularity; the condition is imposed only when convergence and residual
estimates are invoked.
\end{remark}

The next result records the corresponding resolvent regularity and
admissibility.

\begin{proposition}\label{prop:lipschitz_resolvent}
Under Assumption~\ref{ass:regularity}, the resolvent
$(\mathcal{M}+\mathcal{T})^{-1}$ is single-valued, everywhere defined, and
globally Lipschitz continuous on $\mathbb{R}^n\times\mathbb{R}^m$.
Moreover, the operator
$\widehat{\mathcal T}:=(\mathcal M+\mathcal T)^{-1}\mathcal M$ is
single-valued and has full domain, and $\mathcal M$ is an admissible
preconditioner for $\mathcal T$.
\end{proposition}

\begin{proof}
Let $v=(\xi,\zeta)\in\mathbb{R}^n\times\mathbb{R}^m$, and suppose that
$w=(x,y)$ satisfies $v\in (\mathcal{M}+\mathcal{T})(x,y)$.
By the definitions of $\mathcal{M}$ and $\mathcal{T}$, this is equivalent to
\[
\begin{cases}
\xi \in Px + \partial\Phi(x) + \sigma A^\top(Ax-b),\\[0.4em]
\zeta = 2Ax + Qy - b.
\end{cases}
\]
It follows that
\begin{equation}\label{eq:x_resolvent_eq}
x = G^{-1}(\xi+\sigma A^\top b),
\end{equation}
which is uniquely determined by Assumption~\ref{ass:regularity}(iii). Moreover,
\begin{equation}\label{eq:y_resolvent_eq}
y = Q^{-1}(\zeta+b-2Ax),
\end{equation}
which is also uniquely determined because $Q$ is positive definite. Conversely,
for every $v=(\xi,\zeta)$, the pair $(x,y)$ defined by
\eqref{eq:x_resolvent_eq}--\eqref{eq:y_resolvent_eq} satisfies the displayed
system above. Therefore,
$(\mathcal{M}+\mathcal{T})^{-1}$ is single-valued and everywhere defined.

Let $L_G$ be a Lipschitz constant of $G^{-1}$. To prove Lipschitz continuity,
let $v_i=(\xi_i,\zeta_i)$ and
$w_i=(x_i,y_i):=(\mathcal{M}+\mathcal{T})^{-1}v_i$ for $i=1,2$. From
\eqref{eq:x_resolvent_eq} and the Lipschitz continuity of $G^{-1}$,
\[
\|x_1-x_2\|
\le
L_G\|\xi_1-\xi_2\|.
\]
Since $\|\xi_1-\xi_2\|\le \|v_1-v_2\|$, this yields
\[
\|x_1-x_2\|
\le
L_G\|v_1-v_2\|.
\]
Using \eqref{eq:y_resolvent_eq}, we obtain
$y_1-y_2
=
Q^{-1}\bigl((\zeta_1-\zeta_2)-2A(x_1-x_2)\bigr)$,
and thus
\[
\|y_1-y_2\|
\le
\|Q^{-1}\|
\bigl(\|\zeta_1-\zeta_2\|+2\|A\|\|x_1-x_2\|\bigr)
\le
\|Q^{-1}\|
\bigl(1+2\|A\|L_G\bigr)\|v_1-v_2\|.
\]
Therefore,
\[
\|w_1-w_2\|^2
=
\|x_1-x_2\|^2+\|y_1-y_2\|^2
\le
L_{\mathrm{res}}^2\,\|v_1-v_2\|^2,
\]
where
\[
L_{\mathrm{res}}
:=
\left[
L_G^2
+
\|Q^{-1}\|^2
\bigl(1+2\|A\|L_G\bigr)^2
\right]^{1/2}.
\]
This proves that $(\mathcal{M}+\mathcal{T})^{-1}$ is globally Lipschitz continuous.

Since $\mathcal M$ is a bounded linear operator, the composition
$\widehat{\mathcal T}=(\mathcal M+\mathcal T)^{-1}\mathcal M$
is single-valued and defined on the whole space. Since, by assumption,
$\mathcal M$ is self-adjoint and positive semidefinite, the conclusion follows
from Definition~\ref{admissible}.
\end{proof}

In our finite-dimensional setting, the remaining structural assumptions needed
for the abstract reflected Halpern theory are also automatic.

\begin{proposition}\label{prop:structure_sec2}
Under Assumption~\ref{ass:regularity}, the operator $\mathcal T$ is maximal
monotone and $\operatorname{ran}(\mathcal M)$ is closed.
\end{proposition}

\begin{proof}
Define $\widetilde\Phi:\mathbb R^n\times\mathbb R^m\to(-\infty,+\infty]$ by
$\widetilde\Phi(x,y):=\Phi(x)$.
Then
\[
\partial\widetilde\Phi(x,y)=
\begin{pmatrix}
\partial\Phi(x)\\
0
\end{pmatrix}.
\]
Thus $\partial\widetilde\Phi$ is maximal monotone. Next define the linear operator
\[
\mathcal K(x,y):=
\begin{pmatrix}
\sigma A^\top A x-A^\top y\\
Ax
\end{pmatrix}.
\]
For any $w_i=(x_i,y_i)$, $i=1,2$, one has
\[
\langle \mathcal K w_1-\mathcal K w_2,\,w_1-w_2\rangle
=
\sigma\|A(x_1-x_2)\|^2
\ge 0.
\]
Thus $\mathcal K$ is monotone. Since $\mathcal K$ is linear, bounded, and
everywhere defined, the sum $\partial\widetilde\Phi+\mathcal K$ is maximal
monotone. Finally,
\[
\mathcal T
=
\partial\widetilde\Phi+\mathcal K
-
\begin{pmatrix}
\sigma A^\top b\\
b
\end{pmatrix}.
\]
Thus $\mathcal T$ is maximal monotone as well. Because $\mathcal M$ is a linear
operator on the finite-dimensional space $\mathbb R^n\times\mathbb R^m$, its
range is a linear subspace and is therefore closed.
\end{proof}

\begin{remark}[Step-size conditions]
For the representative schemes in Section~\ref{sec2:2.1},
Assumption~\ref{ass:regularity} reduces to familiar step-size conditions. For PDHG,
$P=\tau^{-1}I_n$, $Q=\rho^{-1}I_m$, and $\sigma=0$, with
$P+\sigma A^\top A=\tau^{-1}I_n$; the condition $\mathcal M\succeq0$
reduces, by the Schur complement, to $\tau\rho\|A\|^2\le1$, which is
precisely the classical PDHG step-size condition; see \cite{PDHG}. For CP-AL,
$P=\tau^{-1}I_n-\sigma A^\top A$ and $Q=\rho^{-1}I_m$, and therefore
$P+\sigma A^\top A=\tau^{-1}I_n$, while $\mathcal M\succeq0$ reduces to
$\tau(\sigma+\rho)\|A\|^2 \le 1$, which is precisely the CP-AL step-size
condition; see \cite{zhu2023unified}. For FA-CP,
$P=\tau^{-1}I_n$ and $Q=\rho^{-1}I_m$, which gives
$P+\sigma A^\top A\succeq\tau^{-1}I_n$, and $\mathcal M\succeq0$ again yields
$\tau\rho\|A\|^2\le1$.
In all three cases, $P+\sigma A^\top A\succeq \tau^{-1}I_n$; consequently
Assumption~\ref{ass:regularity}(iii) follows from the sufficient condition
above. Thus the abstract assumptions recover the standard step-size restriction
for PDHG, the natural one for CP-AL, and a PDHG-type condition for the fully
augmented members of FA-CP.
\end{remark}

\subsection{Reflected Halpern Acceleration and $O(1/k)$ KKT Rates}
\label{sec2:2.3}

Given an admissible preconditioner $\mathcal M$, set
$\widehat{\mathcal T}:=(\mathcal M+\mathcal T)^{-1}\mathcal M$. We consider
the reflected Halpern iteration
\begin{equation}\label{eq:reflected_halpern}
    w^{k+1}
    =
    \frac{1}{k+2}w^0
    +
    \frac{k+1}{k+2}
    \Bigl((1+\gamma)\widehat{\mathcal T}(w^k)-\gamma w^k\Bigr),
    \qquad \gamma\in(-1,1].
\end{equation}
Equivalently, defining
\[
    \hat{w}^k := \widehat{\mathcal T}(w^k),
    \qquad
    \bar{w}^k := (1+\gamma)\hat{w}^k-\gamma w^k,
\]
we may rewrite \eqref{eq:reflected_halpern} as
\[
    w^{k+1}
    =
    \frac{1}{k+2}w^0+\frac{k+1}{k+2}\bar{w}^k.
\]
Following the reflected Halpern acceleration framework for dPPM mappings in
\cite{sun2025accelerating}, Algorithm~\ref{alg:halpern_dppm} records the
iteration in the shadow--reflection--anchoring form used throughout the paper.
The shadow sequence \(\{\hat w^k\}\), rather than only the anchored sequence
\(\{w^k\}\), is the object on which the KKT residual estimates are stated.

\begin{algorithm}[H]
\caption{Reflected Halpern-Accelerated dPPM}
\label{alg:halpern_dppm}
{\small
\begin{algorithmic}[1]
\STATE \textbf{Input:} $w^0\in\mathcal H$ and $\gamma\in(-1,1]$.
\FOR{$k=0,1,2,\ldots$}
\STATE \textbf{Step 1.} $\hat w^k=\widehat{\mathcal T}(w^k)$.
\STATE \textbf{Step 2.} $\bar w^k=(1+\gamma)\hat w^k-\gamma w^k$.
\STATE \textbf{Step 3.} $w^{k+1}=(w^0+(k+1)\bar w^k)/(k+2)$.
\ENDFOR
\STATE \textbf{Output:} shadow sequence $\{\hat w^k\}$.
\end{algorithmic}
}
\end{algorithm}

We invoke the abstract reflected Halpern theory for dPPM mappings. The
following proposition is a direct specialization of
\cite[Theorem~2.7, Proposition~2.9]{sun2025accelerating} and records the
statements needed below.

\begin{proposition}\label{prop:general_convergence}
Let $\mathcal{T}:\mathcal{H}\rightrightarrows\mathcal{H}$ be a maximal monotone
operator with $\mathcal{T}^{-1}(0)\neq\emptyset$, and let $\mathcal{M}$ be an
admissible preconditioner for $\mathcal{T}$ with closed range. Let
$\widehat{\mathcal T}:=(\mathcal{M}+\mathcal{T})^{-1}\mathcal{M}$,
and let the sequences $\{w^k\}$, $\{\hat w^k\}$, and $\{\bar w^k\}$ be generated
by Algorithm~\ref{alg:halpern_dppm}. Then the following statements hold.
\begin{enumerate}
    \item[(i)] If $\gamma\in(-1,1]$ and $(\mathcal{M}+\mathcal{T})^{-1}$ is
    continuous, then the sequence $\{\hat w^k\}$ converges strongly to a point
    $w^\ast\in\mathcal{T}^{-1}(0)$. Moreover, if $\gamma\in (-1,1)$, then the
    sequences $\{w^k\}$ and $\{\bar w^k\}$ also converge strongly to the same
    limit $w^\ast$.

    \item[(ii)] If $\gamma\in(-1,1]$, the relaxed residual satisfies
\[
\|w^k-\bar w^k\|_{\mathcal M}
\le
\frac{2\|w^0-w^\ast\|_{\mathcal M}}{k+1},
\qquad
\forall\,k\ge 0,\ \forall\,w^\ast\in\mathcal{T}^{-1}(0).
\]
Equivalently,
\[
\|w^k-\hat w^k\|_{\mathcal M}
\le
\frac{2}{(1+\gamma)(k+1)}
\|w^0-w^\ast\|_{\mathcal M},
\qquad
\forall\,k\ge 0,\ \forall\,w^\ast\in\mathcal{T}^{-1}(0).
\]
\end{enumerate}
\end{proposition}

\begin{proof}
Choose a decomposition $\mathcal{M}=CC^\ast$ as in
\cite[Proposition~2.3]{DPPPA}. In the notation of
\cite{sun2025accelerating}, take $\alpha=2$ and $\rho=1+\gamma$. Our
$\hat w^k$ corresponds to their $\bar w^k$, while our
$\bar w^k$ corresponds to their $\hat w^{k+1}$. Statement~(i) follows from
\cite[Theorem~2.7(a)]{sun2025accelerating}. Statement~(ii) follows from
\cite[Proposition~2.9]{sun2025accelerating}, again with $\alpha=2$ and
$\rho=1+\gamma\in(0,2]$.
\end{proof}

\begin{remark}[Euclidean residual in the Halpern case]
The Lipschitz regularity in Proposition~\ref{prop:lipschitz_resolvent} also
recovers the Euclidean residual estimate used in the spADMM analysis of
\cite{sun2025accelerating}. Specifically, assume that
$(\mathcal M+\mathcal T)^{-1}$ is $L$-Lipschitz continuous and let
$\mathcal M=CC^\ast$. In the Halpern case $\gamma=0$, Corollary~2.11 of
\cite{sun2025accelerating} gives, for every
$w^\ast\in\mathcal T^{-1}(0)$ and every $k\ge0$,
\[
\|w^k-\hat w^k\|
\le
\frac{1}{k+1}\|w^0-w^\ast\|
+
\frac{(5k+1)L\|C\|}{(k+1)^2}\|w^0-w^\ast\|_{\mathcal M}.
\]
Thus the Lipschitz resolvent condition is not merely a well-posedness device: in
the nonreflected Halpern case it converts the degenerate metric estimate into an
ambient Euclidean residual bound. The subsequent analysis uses the
$\mathcal M$-residual estimate in Proposition~\ref{prop:general_convergence}(ii),
which is available throughout the reflected range $\gamma\in(-1,1]$.
\end{remark}

\begin{remark}
The endpoint $\gamma=-1$ is degenerate in the exact dPPM setting as well,
since then $\bar w^k=(1+\gamma)\hat w^k-\gamma w^k=w^k$, and the Halpern update
reduces to $w^{k+1}=\frac{1}{k+2}w^0+\frac{k+1}{k+2}w^k=w^0$.
Accordingly, the abstract convergence and residual statements above are stated
only for $\gamma>-1$.
At the other endpoint, $\gamma=1$ is retained only in the shadow-convergence and
residual statements allowed by the cited reflected Halpern theory. The main KKT
convergence theorem below asserts common convergence of the anchored state and
the shadow sequence, and therefore uses the open interval $\gamma\in(-1,1)$.
\end{remark}

We specialize the abstract residual estimate to the augmented KKT operator
and then convert it back to the original KKT residual.
For $w=(x,y)$, define the original KKT mapping and its residual by
\begin{equation}
\label{eq:kkt_mapping_residual}
\mathcal R(x,y):=\binom{\nabla f(x)+\partial g(x)-A^\top y}{Ax-b},
\qquad
r_{\mathrm{KKT}}(x,y):=\mathrm{dist}\bigl(0,\mathcal R(x,y)\bigr).
\end{equation}
Let $\mathcal E_\sigma$ denote the linear map
$\mathcal E_\sigma(u,v):=(u-\sigma A^\top v,v)$. For a set $S$, write
$\mathrm{dist}_{\mathcal M}(w,S)$ for
$\inf_{u\in S}\|w-u\|_{\mathcal M}$.

\begin{theorem}
\label{thm:kkt_gap_rate}
Suppose that Assumption~\ref{ass:regularity} holds, that
$\mathcal{T}^{-1}(0)\neq\emptyset$, and that $\gamma\in (-1,1)$. The
nonemptiness assumption is the explicit form of the standing
KKT-nonemptiness convention for the augmented mapping. Let
$\{w^k\}$, $\{\hat w^k\}$, and $\{\bar w^k\}$ be generated by
Algorithm~\ref{alg:halpern_dppm}. Then $\{w^k\}$ and $\{\hat w^k\}$ converge
strongly to a common limit
$w^\ast=(x^\ast,y^\ast)\in \mathcal T^{-1}(0)=\mathcal R^{-1}(0)$.
Moreover, for every $k\ge0$,
\[
r_{\mathrm{KKT}}(\hat w^k)
\le
\frac{2\|\mathcal E_\sigma\|\sqrt{\|\mathcal M\|}}
{(1+\gamma)(k+1)}
\,\mathrm{dist}_{\mathcal M}\bigl(w^0,\mathcal T^{-1}(0)\bigr).
\]
Consequently, $r_{\mathrm{KKT}}(\hat w^k)=O(1/k)$. Set
\[
C_\ast:=\sup_{k\ge0}\bigl(\|\hat x^k-x^\ast\|+\|\hat y^k\|\bigr)+\|y^\ast\|.
\]
Then $C_\ast<\infty$ and
\[
\bigl|\Phi(\hat x^k)-\Phi(x^\ast)\bigr|
\le
C_\ast\,r_{\mathrm{KKT}}(\hat w^k),\qquad k\ge0.
\]
In particular, $\bigl|\Phi(\hat x^k)-\Phi(x^\ast)\bigr|=O(1/k)$.
\end{theorem}

\begin{proof}
By Propositions~\ref{prop:structure_sec2} and
\ref{prop:lipschitz_resolvent}, all assumptions of
Proposition~\ref{prop:general_convergence} are satisfied. Since $\gamma\in
(-1,1)$, Proposition~\ref{prop:general_convergence}(i) implies that
$w^k\to w^\ast$ and $\hat w^k\to w^\ast$ for some
$w^\ast\in\mathcal T^{-1}(0)$. By the definition of $\mathcal T$,
its second component is $Ax-b$. On the zero set of $\mathcal T$, the
augmentation term therefore vanishes, and
$\mathcal T^{-1}(0)=\mathcal R^{-1}(0)$.

For any $(x,y)$, the set-valued identity
$\mathcal R(x,y)=\mathcal E_\sigma\bigl(\mathcal T(x,y)\bigr)$ follows
directly from the definition of $\mathcal T$: its first component differs from
that of $\mathcal R$ by $\sigma A^\top(Ax-b)$, while the second components are
identical. Therefore
\[
r_{\mathrm{KKT}}(\hat w^k)
\le
\|\mathcal E_\sigma\|\,\mathrm{dist}\bigl(0,\mathcal T(\hat w^k)\bigr).
\]
Moreover, the resolvent relation for $\hat w^k$ gives
$\mathcal M(w^k-\hat w^k)\in\mathcal T(\hat w^k)$, and therefore
\[
\mathrm{dist}\bigl(0,\mathcal T(\hat w^k)\bigr)
\le
\|\mathcal M(w^k-\hat w^k)\|
\le
\sqrt{\|\mathcal M\|}\,\|w^k-\hat w^k\|_{\mathcal M}.
\]
Since $\bar w^k=(1+\gamma)\hat w^k-\gamma w^k$, we have
$w^k-\bar w^k=(1+\gamma)(w^k-\hat w^k)$,
and Proposition~\ref{prop:general_convergence}(ii), after taking the infimum
over $\mathcal T^{-1}(0)$, yields
\[
\|w^k-\hat w^k\|_{\mathcal M}
=
\frac{1}{1+\gamma}\|w^k-\bar w^k\|_{\mathcal M}
\le
\frac{2}{(1+\gamma)(k+1)}
\,\mathrm{dist}_{\mathcal M}\bigl(w^0,\mathcal T^{-1}(0)\bigr).
\]
Combining the last three estimates proves the residual bound.

It remains to prove the objective estimate. The resolvent relation also implies
$\mathcal T(\hat w^k)\neq\emptyset$, and therefore
$\mathcal R(\hat w^k)\neq\emptyset$. Since $\mathcal R(\hat w^k)$ is closed and
convex in finite dimensions, it contains a least-norm element. Let
$r^k=(r_x^k,r_y^k)\in\mathcal R(\hat w^k)$ satisfy
$\|r^k\|=r_{\mathrm{KKT}}(\hat w^k)$. Then $r_y^k=A\hat x^k-b$ and
$r_x^k\in\partial\Phi(\hat x^k)-A^\top\hat y^k$.
Thus there exists $g^k\in\partial\Phi(\hat x^k)$ such that
$g^k=A^\top\hat y^k+r_x^k$.
By convexity of $\Phi$,
\[
\Phi(\hat x^k)-\Phi(x^\ast)
\le
\langle g^k,\hat x^k-x^\ast\rangle
=
\langle r_x^k,\hat x^k-x^\ast\rangle
+
\langle \hat y^k,A\hat x^k-b\rangle.
\]
Consequently,
\[
\Phi(\hat x^k)-\Phi(x^\ast)
\le
\bigl(\|\hat x^k-x^\ast\|+\|\hat y^k\|\bigr)\,
r_{\mathrm{KKT}}(\hat w^k).
\]

Since $w^\ast\in\mathcal R^{-1}(0)$, there exists
$g^\ast\in\partial\Phi(x^\ast)$ such that $g^\ast=A^\top y^\ast$. Again by convexity,
\[
\Phi(\hat x^k)-\Phi(x^\ast)
\ge
\langle g^\ast,\hat x^k-x^\ast\rangle
=
\langle y^\ast,A\hat x^k-b\rangle
\ge
-\|y^\ast\|\,r_{\mathrm{KKT}}(\hat w^k).
\]
Combining the upper and lower bounds gives
\[
\bigl|\Phi(\hat x^k)-\Phi(x^\ast)\bigr|
\le
\Bigl(\|\hat x^k-x^\ast\|+\|\hat y^k\|+\|y^\ast\|\Bigr)
r_{\mathrm{KKT}}(\hat w^k).
\]
Since $\hat w^k\to w^\ast$, the constant $C_\ast$ is finite. The stated
$O(1/k)$ objective-gap bound follows from the residual estimate.
\end{proof}

\begin{remark}[Worst-case tightness]
The preceding $O(1/k)$ residual estimate is sharp in the general convex
setting. To see this, fix $\gamma\in(-1,1)$ and choose
$\eta\in(0,1)$ with $(1+\gamma)\eta\le1$. For the scalar problem
\[
        \min_{x\in\mathbb R}\ \frac{a}{2}x^2,
        \qquad a:=\frac{\eta}{1-\eta},
\]
viewed as the unconstrained instance of \eqref{P} and equipped with
$\mathcal M=1$, the dPPM map is
$F(x)=(I+\mathcal T)^{-1}x=(1-\eta)x$ and
$r_F(x)=|x-F(x)|=\eta|x|$. Set $q:=1-(1+\gamma)\eta$. Then the reflected map is
$\mathcal S_\gamma(x)=qx$, and the scalar Halpern recursion becomes
\[
        x^{k+1}=\frac{1}{k+2}x^0+\frac{k+1}{k+2}qx^k .
\]
Multiplying by $k+2$ and setting $z_k:=(k+1)x^k$ gives
$z_{k+1}=x^0+qz_k$ with $z_0=x^0$. Therefore
$z_k=(1+q+\cdots+q^k)x^0$. Starting from $x^0\ne0$, this yields
\[
        x^k=\frac{1-q^{k+1}}{(1+\gamma)\eta(k+1)}x^0,\qquad
        r_F(x^k)=\frac{1-q^{k+1}}{(1+\gamma)(k+1)}|x^0|.
\]
The shadow point satisfies $\hat x^k=F(x^k)$, and its KKT residual is
$r_{\rm KKT}(\hat x^k)=|a\hat x^k|=r_F(x^k)$. Thus
\[
        \lim_{k\to\infty}(k+1)r_{\rm KKT}(\hat x^k)
        =
        \frac{|x^0|}{1+\gamma}>0.
\]
This shows that the nonergodic residual bound cannot, in general, be improved
to $o(1/k)$.
\end{remark}

\section{A Forward--Backward View of Linearized Reflected Schemes}
\label{sec3}

This section studies linearized augmented primal--dual schemes. The linearized update
admits an exact preconditioned forward--backward splitting (PFBS)
representation, and its reflected relaxation range is controlled by the Schur
complement quantity
\[
        \mu_x:=\lambda_{\min}(P-A^\top Q^{-1}A).
\]
This representation leads to reflected Halpern iterates with nonergodic
\(O(1/k)\) bounds for the KKT residual and the objective-value gap.

\subsection{A Linearized Forward--Backward Reformulation}
\label{sec3:linearized_reformulation}

The computational motivation for the linearized schemes is to move a smooth
component of the primal objective to the forward step, so that the shadow
\(x\)-update becomes explicit or reduces to a simple proximal operation. This
mechanism applies directly to the PDHG metric and to the shifted CP-AL metric.
For instance, when \(h=f\), the PDHG choice yields the usual
forward--backward/proximal-gradient update, while the CP-AL metric cancels the
augmented quadratic in the backward subproblem and leaves a proximal step for
\(g\). The same linearization does not play this role for the fully augmented
FA-CP metric with \(P=\tau^{-1}I\) and \(\sigma>0\): the term
\(\frac{\sigma}{2}\|Ax-b\|^2\) remains in the primal minimization, and the
resulting subproblem is generally still coupled through \(A\). Thus the
linearized theory below is most useful computationally for the PDHG and CP-AL
corners, while FA-CP is treated later as a subproblem-based augmented method.

To cover different linearization choices in a single notation, let
$h:\mathbb R^n\to\mathbb R$ be convex and continuously differentiable, with
$L_h$-Lipschitz continuous gradient, and assume that
\begin{equation}\label{eq:sec31_h_assump}
\phi_h(x):=f(x)-h(x)+\frac{\sigma}{2}\|Ax-b\|^2
\quad\text{is convex on }\mathbb R^n.
\end{equation}
Recall from \eqref{T_AL} that
\[
\mathcal T(x,y)=
\begin{pmatrix}
\partial g(x)+\nabla f(x)-A^\top y+\sigma A^\top(Ax-b)\\
Ax-b
\end{pmatrix}.
\]
We split $\mathcal T$ as
\begin{equation}\label{eq:AB_split_sec3_general}
\begin{aligned}
\mathcal T&=\mathcal A_h+\mathcal B_h,
\qquad
\mathcal B_h(x,y):=\binom{\nabla h(x)}{0},\\
\mathcal A_h(x,y)&:=
\binom{\partial g(x)+\nabla f(x)-\nabla h(x)-A^\top y+\sigma A^\top(Ax-b)}{Ax-b}.
\end{aligned}
\end{equation}
The associated PFBS mapping is
\[
\widehat{\mathcal T}_{\rm lin}(w)
:=
(\mathcal M+\mathcal A_h)^{-1}\bigl(\mathcal Mw-\mathcal B_h(w)\bigr),
\qquad
\mathcal M=
\begin{pmatrix}
P & A^\top\\
A & Q
\end{pmatrix}.
\]

The following basic facts justify the use of this splitting in a
preconditioned forward--backward step.

\begin{proposition}\label{prop:split_h_general}
Suppose that \(h\) is convex and continuously differentiable, that \(\nabla h\)
is \(L_h\)-Lipschitz continuous, and that \eqref{eq:sec31_h_assump} holds.
Then the operator $\mathcal A_h$ is maximally monotone, while $\mathcal B_h$ is
monotone and globally Lipschitz continuous with Lipschitz constant $L_h$.
\end{proposition}

\begin{proof}
Define
\[
\varphi_h(x):=g(x)+f(x)-h(x)+\frac{\sigma}{2}\|Ax-b\|^2.
\]
By \eqref{eq:sec31_h_assump}, the function $\varphi_h$ is proper, closed, and
convex; therefore $\partial\varphi_h$ is maximally monotone. Since
\[
\mathcal A_h(x,y)=
\begin{pmatrix}
\partial\varphi_h(x)\\ 0
\end{pmatrix}
+
\begin{pmatrix}
0 & -A^\top\\ A & 0
\end{pmatrix}
\binom{x}{y}
+
\binom{0}{-b},
\]
and the block operator is bounded and skew-adjoint, $\mathcal A_h$ is maximally
monotone. Moreover, $\mathcal B_h=\nabla\widetilde h$ with
$\widetilde h(x,y):=h(x)$; therefore $\mathcal B_h$ is monotone. Its Lipschitz bound is
immediate from the $L_h$-Lipschitz continuity of $\nabla h$:
\[
\|\mathcal B_h(w)-\mathcal B_h(w')\|
=
\|\nabla h(x)-\nabla h(x')\|
\le L_h\|x-x'\|
\le L_h\|w-w'\|.
\]
\end{proof}

\begin{remark}[Special choices of the forward part]
The two-parameter partial-linearization family is contained in the present
framework. If
\[
h(x)=\psi_{\theta,\delta}(x):=\theta f(x)+\delta\frac{\sigma}{2}\|Ax-b\|^2,
\qquad \theta,\delta\in[0,1],
\]
then
\[
\nabla h(x)=\theta\nabla f(x)+\delta\sigma A^\top(Ax-b),
\]
and \eqref{eq:AB_split_sec3_general} reduces to the splitting used in the
corresponding partial-linearization scheme. In particular, $h=f$ corresponds
to linearizing only the smooth objective term, while
$h=\theta f+\delta\frac{\sigma}{2}\|Ax-b\|^2$ yields a family that also
partially linearizes the augmented quadratic penalty.

The parameter \(\delta\) has a concrete computational role. In the primal
subproblem, this choice leaves only
\((1-\delta)\frac{\sigma}{2}\|Ax-b\|^2\) in the backward part and treats the
remaining \(\delta\)-fraction by the forward correction
\(\delta\sigma A^\top(Ax^k-b)\). Thus \(\delta=0\) keeps the full augmented
quadratic implicit, whereas \(\delta=1\) moves it completely to the forward
step. Intermediate values provide a tradeoff between retaining useful
augmented curvature and reducing the cost of the primal minimization. This can
be relevant when \(A^\top A\) is dense or poorly structured: keeping the
quadratic implicit may require a coupled linear solve or inner iteration,
whereas the forward treatment only uses products with \(A\) and \(A^\top\).
\end{remark}

\begin{assumption}\label{ass:sec31_metric}
Throughout this section,
\[
Q\succ0,
\qquad
P-A^\top Q^{-1}A\succ0.
\]
Equivalently, $\mathcal M$ is positive definite on
$\mathbb R^n\times\mathbb R^m$.
\end{assumption}

This is a Schur-complement condition. For example, if
$Q=\rho^{-1}I_m$ and $P=\tau^{-1}I_n$, it holds whenever
$\tau\rho\|A\|^2<1$. For the shifted augmented Chambolle--Pock metric
$P=\tau^{-1}I_n-\sigma A^\top A$ and $Q=\rho^{-1}I_m$, it is implied by
$\tau(\rho+\sigma)\|A\|^2<1$.

Together with Proposition~\ref{prop:split_h_general}, this metric assumption
ensures that $\widehat{\mathcal T}_{\rm lin}$ is single-valued and everywhere
defined. The abstract PFBS map then gives the following explicit primal--dual
update.

\begin{proposition}\label{prop:alg_equiv_linearized}
Suppose that \(h\) is convex and continuously differentiable, that \(\nabla h\)
is \(L_h\)-Lipschitz continuous, and that \eqref{eq:sec31_h_assump} and
Assumption~\ref{ass:sec31_metric} hold. Let $w^k=(x^k,y^k)$ and define
$\hat w^k=(\hat x^k,\hat y^k):=\widehat{\mathcal T}_{\rm lin}(w^k)$. Then
$(\hat x^k,\hat y^k)$ is characterized by
\begin{subequations}\label{eq:alg_equiv_linearized}
\begin{align}
\hat x^k
&\in \arg\min_x\Bigl\{
 f(x)+g(x)-h(x)+\langle\nabla h(x^k),x-x^k\rangle-\langle y^k,Ax-b\rangle 
 \notag\\
&\hspace{5.3em} +\frac{\sigma}{2}\|Ax-b\|^2+\frac12\|x-x^k\|_P^2
\Bigr\}, \label{eq:alg_equiv_linearized_x}\\
\hat y^k
&=y^k+Q^{-1}\bigl(b-A(2\hat x^k-x^k)\bigr).\label{eq:alg_equiv_linearized_y}
\end{align}
\end{subequations}
\end{proposition}

\begin{proof}
By definition,
\[
\hat w^k=(\mathcal M+\mathcal A_h)^{-1}\bigl(\mathcal Mw^k-\mathcal B_h(w^k)\bigr)
\iff
\mathcal Mw^k-\mathcal B_h(w^k)\in\mathcal M\hat w^k+\mathcal A_h(\hat w^k).
\]
Comparing the two components gives
\[
\begin{aligned}
Px^k+A^\top y^k-\nabla h(x^k)
&\in P\hat x^k+\partial g(\hat x^k)+\nabla f(\hat x^k)-\nabla h(\hat x^k)
\\
&\hspace{5.3em}+\sigma A^\top(A\hat x^k-b),\\
Ax^k+Qy^k&=2A\hat x^k+Q\hat y^k-b.
\end{aligned}
\]
That is,
\[
0\in \partial g(\hat x^k)+\nabla f(\hat x^k)-\nabla h(\hat x^k)+\nabla h(x^k)
-A^\top y^k+\sigma A^\top(A\hat x^k-b)+P(\hat x^k-x^k),
\]
and
\[
Q(\hat y^k-y^k)=b-A(2\hat x^k-x^k).
\]
Now define
\[
\begin{aligned}
\Xi_k(x):={}&f(x)+g(x)-h(x)+\langle\nabla h(x^k),x-x^k\rangle
-\langle y^k,Ax-b\rangle\\
&\qquad +\frac{\sigma}{2}\|Ax-b\|^2+\frac12\|x-x^k\|_P^2.
\end{aligned}
\]
Because $f-h+\frac{\sigma}{2}\|Ax-b\|^2$ is convex by
\eqref{eq:sec31_h_assump}, $g$ is proper, closed, and convex, and
Assumption~\ref{ass:sec31_metric} implies $P\succ0$, the function $\Xi_k$ is
proper, closed, and strongly convex. Its first-order optimality condition is
the displayed primal inclusion above, and is therefore equivalent to
\eqref{eq:alg_equiv_linearized_x}; the displayed dual equation is precisely
\eqref{eq:alg_equiv_linearized_y}.
\end{proof}

Thus the linearization induced by $h$ replaces $h(x)$ in the primal subproblem
by its affine model $h(x^k)+\langle\nabla h(x^k),x-x^k\rangle$; after dropping
the irrelevant constant $h(x^k)$, this contributes the term
$-h(x)+\langle\nabla h(x^k),x-x^k\rangle$ in
\eqref{eq:alg_equiv_linearized_x}.

Under Assumption~\ref{ass:sec31_metric}, the operator $\mathcal M$ is
invertible. We therefore set
\[
\mathcal C:=\mathcal M^{-1}\mathcal B_h,
\qquad
\mathcal G:=I-\mathcal C,
\qquad
\mathcal J:=(\mathcal M+\mathcal A_h)^{-1}\mathcal M,
\]
which gives
\[
\widehat{\mathcal T}_{\rm lin}=\mathcal J\circ\mathcal G.
\]

\subsection{Averagedness and the Relaxation Range}
\label{sec3:averagedness_relaxation}

We use the $\mathcal M$-metric notation
$\langle u,v\rangle_{\mathcal M}:=\langle u,\mathcal Mv\rangle$ and
$\|u\|_{\mathcal M}:=\sqrt{\langle u,u\rangle_{\mathcal M}}$; unqualified
norms are Euclidean. Recall that a mapping $R$ is nonexpansive in this metric if
\(\|R u-R v\|_{\mathcal M}\le \|u-v\|_{\mathcal M}\) for all $u,v$, and is
\(\alpha\)-averaged, with \(\alpha\in(0,1)\), if
\[
        R=(1-\alpha)I+\alpha N
\]
for some nonexpansive mapping $N$ in the same metric. Firm nonexpansiveness
means \(1/2\)-averagedness. We also use the standard equivalent
characterization
\[
\|R u-R v\|_{\mathcal M}^2
\le
\|u-v\|_{\mathcal M}^2
-
\frac{1-\alpha}{\alpha}
\|(I-R)u-(I-R)v\|_{\mathcal M}^2 .
\]

The forward operator $\mathcal B_h$ only acts on the
primal component; the key metric quantity is the amount by which
$\mathcal M$ controls that component. Define
\[
\mu_x:=\lambda_{\min}\!\bigl(P-A^\top Q^{-1}A\bigr)>0.
\]
Writing $S_x:=P-A^\top Q^{-1}A$, a completion of the square gives, for every
primal--dual direction $(u,v)$,
\[
\|(u,v)\|_{\mathcal M}^2
=\|u\|_{S_x}^2+\bigl\|Q^{1/2}(v+Q^{-1}Au)\bigr\|^2 .
\]
Consequently,
\begin{equation}\label{eq:sec32_xdom}
\|(u,v)\|_{\mathcal M}^2\ge \mu_x\|u\|^2 .
\end{equation}
Thus closeness in the $\mathcal M$-metric implies Euclidean closeness of the
primal components. This is the bridge that allows the $L_h$-Lipschitz continuity
of $\nabla h$ to be expressed as an averagedness estimate for the forward step
in the $\mathcal M$-metric.

\begin{proposition}\label{prop:sec32_averaged}
Suppose that \(h\) is convex and continuously differentiable, that \(\nabla h\)
is \(L_h\)-Lipschitz continuous, and that \eqref{eq:sec31_h_assump} and
Assumption~\ref{ass:sec31_metric} hold, with $\mu_x>L_h/2$.
If $L_h>0$, then $\mathcal G=I-\mathcal C$ is $\alpha_{\rm fw}$-averaged in the
$\mathcal M$-metric with $\alpha_{\rm fw}:=L_h/(2\mu_x)\in(0,1)$. If
$L_h=0$, then $\mathcal G$ is an isometry in the $\mathcal M$-metric. In all
cases, $\mathcal J=(\mathcal M+\mathcal A_h)^{-1}\mathcal M$ is firmly
nonexpansive in the $\mathcal M$-metric, and
$\widehat{\mathcal T}_{\rm lin}=\mathcal J\circ\mathcal G$ is
$\alpha_{\rm lin}$-averaged, where
\[
\alpha_{\rm lin}:=\frac{1}{2-\frac{L_h}{2\mu_x}}.
\]
\end{proposition}

\begin{proof}
By Proposition~\ref{prop:split_h_general}, $\mathcal A_h$ is maximally
monotone. Since $\mathcal M\succ0$, the resolvent
$\mathcal J=(\mathcal M+\mathcal A_h)^{-1}\mathcal M$ is firmly nonexpansive
in the $\mathcal M$-metric.

If $L_h=0$, then $\nabla h$ is constant, and therefore
$\mathcal G=I-\mathcal M^{-1}\mathcal B_h$ is a translation. Therefore
$\|\mathcal Gw-\mathcal Gw'\|_{\mathcal M}=\|w-w'\|_{\mathcal M}$ for all
$w,w'$. Applying the firm nonexpansiveness of $\mathcal J$ to the translated
inputs $\mathcal Gw$ and $\mathcal Gw'$ gives
\[
\|\widehat{\mathcal T}_{\rm lin}w-\widehat{\mathcal T}_{\rm lin}w'\|_{\mathcal M}^2
\le
\langle
\widehat{\mathcal T}_{\rm lin}w-\widehat{\mathcal T}_{\rm lin}w',
w-w'\rangle_{\mathcal M}
\]
because $\mathcal Gw-\mathcal Gw'=w-w'$. Thus
$\widehat{\mathcal T}_{\rm lin}$ is firmly nonexpansive, i.e.,
$1/2$-averaged, which is the value of $\alpha_{\rm lin}$ when $L_h=0$.

It remains to consider $L_h>0$.
Let $w=(x,y)$, $w'=(x',y')$, and set
\[
d:=w-w',\qquad \xi:=\nabla h(x)-\nabla h(x'),\qquad c:=\mathcal Cw-\mathcal Cw'.
\]
Since $\mathcal B_h(w)-\mathcal B_h(w')=(\xi,0)$ and $\mathcal M\mathcal C=\mathcal B_h$,
\[
\langle c,d\rangle_{\mathcal M}
=\langle \mathcal B_h(w)-\mathcal B_h(w'),w-w'\rangle
=\langle \xi,x-x'\rangle.
\]
Because $h$ is convex and $\nabla h$ is $L_h$-Lipschitz continuous, the
Baillon--Haddad theorem~\cite{baillon1977quelques} yields
\[
\langle \xi,x-x'\rangle\ge \frac{1}{L_h}\|\xi\|^2.
\]
On the other hand, \eqref{eq:sec32_xdom} implies
$\|z\|_{\mathcal M}^2\ge \mu_x\|u\|^2$ for every $z=(u,v)$. Thus
\[
\|(\xi,0)\|_{\mathcal M^{-1}}^2
=\sup_{z\neq0}\frac{\langle(\xi,0),z\rangle^2}{\|z\|_{\mathcal M}^2}
\le \frac{1}{\mu_x}\|\xi\|^2.
\]
Here $\|\cdot\|_{\mathcal M^{-1}}$ denotes the norm induced by
$\mathcal M^{-1}$.
Since $\mathcal M c=(\xi,0)$,
\[
\|c\|_{\mathcal M}^2=\|\mathcal Mc\|_{\mathcal M^{-1}}^2
\le \frac{1}{\mu_x}\|\xi\|^2.
\]
Combining the two estimates, we obtain
\[
\langle c,d\rangle_{\mathcal M}
\ge \frac{1}{L_h}\|\xi\|^2
\ge \frac{\mu_x}{L_h}\|c\|_{\mathcal M}^2.
\]
Thus $\mathcal C$ is $\beta$-cocoercive in the $\mathcal M$-metric with
$\beta:=\mu_x/L_h$.

Set $q:=\mathcal Gw-\mathcal Gw'=d-c$. Then
\[
\|q\|_{\mathcal M}^2
=\|d-c\|_{\mathcal M}^2
\le \|d\|_{\mathcal M}^2-(2\beta-1)\|c\|_{\mathcal M}^2.
\]
Since $\mu_x>L_h/2$, one has $\beta>1/2$, which gives
$\alpha_{\rm fw}=1/(2\beta)=L_h/(2\mu_x)\in(0,1)$ and
\[
\|\mathcal Gw-\mathcal Gw'\|_{\mathcal M}^2
\le \|w-w'\|_{\mathcal M}^2
-\frac{1-\alpha_{\rm fw}}{\alpha_{\rm fw}}
\|(I-\mathcal G)w-(I-\mathcal G)w'\|_{\mathcal M}^2 .
\]
Thus $\mathcal G$ is $\alpha_{\rm fw}$-averaged.

Next let
\[
t:=\widehat{\mathcal T}_{\rm lin}w-\widehat{\mathcal T}_{\rm lin}w',
\qquad q:=\mathcal Gw-\mathcal Gw'.
\]
Then $t=\mathcal J(\mathcal Gw)-\mathcal J(\mathcal Gw')$, and firm
nonexpansiveness gives
\[
\|t\|_{\mathcal M}^2+\|q-t\|_{\mathcal M}^2\le \|q\|_{\mathcal M}^2.
\]
Together with the averagedness inequality for $\mathcal G$ this yields
\[
\|t\|_{\mathcal M}^2+\|q-t\|_{\mathcal M}^2
+\frac{1-\alpha_{\rm fw}}{\alpha_{\rm fw}}\|d-q\|_{\mathcal M}^2
\le \|d\|_{\mathcal M}^2.
\]
Now set $r:=d-q$ and $s:=q-t$, for which $d-t=r+s$. Writing
$\eta:=(1-\alpha_{\rm fw})/\alpha_{\rm fw}$, we compute
\[
(1+\eta)\bigl(\|s\|_{\mathcal M}^2+\eta\|r\|_{\mathcal M}^2\bigr)
-\eta\|r+s\|_{\mathcal M}^2
=\|s-\eta r\|_{\mathcal M}^2\ge0.
\]
Therefore
\[
\|s\|_{\mathcal M}^2+\eta\|r\|_{\mathcal M}^2
\ge \frac{\eta}{1+\eta}\|r+s\|_{\mathcal M}^2
=(1-\alpha_{\rm fw})\|d-t\|_{\mathcal M}^2.
\]
Substituting this bound into the previous inequality gives
\[
\|t\|_{\mathcal M}^2+(1-\alpha_{\rm fw})\|d-t\|_{\mathcal M}^2
\le \|d\|_{\mathcal M}^2.
\]
Therefore $\widehat{\mathcal T}_{\rm lin}$ is $\alpha_{\rm lin}$-averaged with
\[
\frac{1-\alpha_{\rm lin}}{\alpha_{\rm lin}}=1-\alpha_{\rm fw},
\qquad\text{i.e.,}\qquad
\alpha_{\rm lin}=\frac{1}{2-\alpha_{\rm fw}}=\frac{1}{2-\frac{L_h}{2\mu_x}}.
\]
\end{proof}

The averagedness constant immediately determines a sufficient admissible range
for the reflected relaxation parameter. We do not use, or claim, a converse
necessity statement for this range.

\begin{theorem}\label{thm:sec32_relaxation}
Under the assumptions of Proposition~\ref{prop:sec32_averaged}, define
$\mathcal S_\gamma:=(1+\gamma)\widehat{\mathcal T}_{\rm lin}-\gamma I$.
Then $\mathcal S_\gamma$ is nonexpansive in the $\mathcal M$-metric whenever
\[
-1< \gamma\le 1-\frac{L_h}{2\mu_x}.
\]
\end{theorem}

\begin{proof}
By Proposition~\ref{prop:sec32_averaged},
$\widehat{\mathcal T}_{\rm lin}=(1-\alpha_{\rm lin})I+\alpha_{\rm lin}\mathcal N$
for some nonexpansive map $\mathcal N$ in the $\mathcal M$-metric. Therefore
\[
\mathcal S_\gamma=
\bigl(1-(1+\gamma)\alpha_{\rm lin}\bigr)I
+(1+\gamma)\alpha_{\rm lin}\mathcal N.
\]
Thus $\mathcal S_\gamma$ is nonexpansive whenever
$(1+\gamma)\alpha_{\rm lin}\le1$ and the coefficients are nonnegative. Since
$\alpha_{\rm lin}>0$, these two requirements are equivalent to
\[
-1<\gamma\le \alpha_{\rm lin}^{-1}-1=1-\frac{L_h}{2\mu_x},
\]
which is exactly the claimed range.
\end{proof}

\begin{remark}[Endpoint]
\label{rem:linearized-endpoint-convention}
The nonexpansiveness proof permits the endpoint
$\gamma=1-L_h/(2\mu_x)$. In the algorithmic statements below we use the
corresponding open interval, so that endpoint residual cases do not have to be
separated from convergence statements. The endpoint can still be included when
only the nonexpansive estimate is invoked.
\end{remark}

For the two main choices used later, the following Schur complements and
convenient sufficient conditions are obtained.

\begin{corollary}\label{cor:sec32_CP}
In the special case $h=f$, where $L_h=L_f$, the following statements hold:
\begin{enumerate}
\item[(i)] if $P=\tau^{-1}I-\sigma A^\top A$ and $Q=\rho^{-1}I$, then
\[
\mu_x=\lambda_{\min}\!\bigl(\tau^{-1}I-(\sigma+\rho)A^\top A\bigr),
\]
and the threshold $\mu_x>L_f/2$ is implied by
\[
\tau\Bigl((\sigma+\rho)\|A\|^2+\frac{L_f}{2}\Bigr)<1;
\]
\item[(ii)] if $P=\tau^{-1}I$ and $Q=\rho^{-1}I$, then
\[
\mu_x=\lambda_{\min}\!\bigl(\tau^{-1}I-\rho A^\top A\bigr),
\]
and the threshold $\mu_x>L_f/2$ is implied by
\[
\tau\Bigl(\rho\|A\|^2+\frac{L_f}{2}\Bigr)<1.
\]
\end{enumerate}
\end{corollary}

\begin{proof}
When $h=f$, we have $L_h=L_f$. In case (i),
$Q^{-1}=\rho I$ and
$P-A^\top Q^{-1}A=\tau^{-1}I-(\sigma+\rho)A^\top A$. Therefore
\[
\mu_x\ge \tau^{-1}-(\sigma+\rho)\|A\|^2,
\]
and the displayed condition in (i) implies $\mu_x>L_f/2$. Case (ii) is
identical, with $P-A^\top Q^{-1}A=\tau^{-1}I-\rho A^\top A$, and the displayed
condition in (ii) again implies $\mu_x>L_f/2$.
\end{proof}

\begin{remark}
A direct curvature-compensation argument yields the more conservative condition
$\mu_x\ge L_h$. The operator-theoretic PFBS analysis shows that averagedness
already holds under the weaker threshold $\mu_x>L_h/2$. In particular, the
admissible relaxation range is governed by the Schur complement
quantity $\mu_x=\lambda_{\min}(P-A^\top Q^{-1}A)$.
\end{remark}

\subsection{Reflected Halpern PFBS and Nonergodic KKT Rates}
\label{sec3:halpern_pfbs_rates}

Theorem~\ref{thm:sec32_relaxation} gives the map-level nonexpansiveness needed
for the reflected Halpern PFBS iteration, with
$\gamma\in\bigl(-1,\,1-L_h/(2\mu_x)\bigr)$.
Using Proposition~\ref{prop:alg_equiv_linearized}, the PFBS shadow step can be
written explicitly in primal--dual variables. Algorithm~\ref{alg:halpern_pfbs}
is the linearized counterpart of Algorithm~\ref{alg:halpern_dppm}: first compute
the PFBS shadow point, then apply the same reflection and Halpern anchoring.

\begin{algorithm}[H]
\caption{Reflected Halpern-Accelerated Linearized PFBS}
\label{alg:halpern_pfbs}
{\small
\begin{algorithmic}[1]
\STATE \textbf{Input:} $w^0=(x^0,y^0)$ and
$\gamma\in\bigl(-1,\,1-L_h/(2\mu_x)\bigr)$.
\FOR{$k=0,1,2,\ldots$}
\STATE \textbf{Step 1.} Compute
\[
\hat x^k=\arg\min_x
\left\{
\begin{aligned}
& f(x)+g(x)-h(x)+\langle\nabla h(x^k),x-x^k\rangle
 -\langle y^k,Ax-b\rangle \\
&\qquad +\frac{\sigma}{2}\|Ax-b\|^2
 +\frac12\|x-x^k\|_P^2
\end{aligned}
\right\}
\]
\STATE \textbf{Step 2.} Set
$\hat y^k=y^k+Q^{-1}\bigl(b-A(2\hat x^k-x^k)\bigr)$ and
$\hat w^k=(\hat x^k,\hat y^k)$.
\STATE \textbf{Step 3.} $\bar w^k=(1+\gamma)\hat w^k-\gamma w^k$.
\STATE \textbf{Step 4.} $w^{k+1}=(w^0+(k+1)\bar w^k)/(k+2)$.
\ENDFOR
\STATE \textbf{Output:} shadow sequence $\{\hat w^k\}$.
\end{algorithmic}
}
\end{algorithm}

The following proposition records convergence and fixed-point residual estimates
for this reflected Halpern PFBS scheme.

\begin{proposition}\label{prop:halpern_pfbs_basic}
Assume that \(h\) is convex and continuously differentiable, that \(\nabla h\)
is \(L_h\)-Lipschitz continuous, and that \eqref{eq:sec31_h_assump},
Assumption~\ref{ass:sec31_metric}, $\mu_x>L_h/2$, and
$\mathcal T^{-1}(0)\neq\emptyset$. The last condition is the explicit
KKT-nonemptiness requirement inherited from the standing convention. Fix
$\gamma\in\bigl(-1,\,1-L_h/(2\mu_x)\bigr)$ and define
$\mathcal S_\gamma:=(1+\gamma)\widehat{\mathcal T}_{\rm lin}-\gamma I$. Let
the sequences $\{w^k\}$, $\{\hat w^k\}$, and $\{\bar w^k\}$ be generated by
Algorithm~\ref{alg:halpern_pfbs}.
Then there exists $w^\ast\in\mathcal T^{-1}(0)$ such that
\begin{equation}\label{eq:pfbs_conv}
w^k\to w^\ast,
\qquad
\bar w^k\to w^\ast,
\qquad
\hat w^k\to w^\ast,
\end{equation}
and, for every $k\ge0$,
\begin{equation}\label{eq:pfbs_residuals}
\begin{aligned}
\|w^k-\bar w^k\|_{\mathcal M}
&\le \frac{2}{k+1}\,\mathrm{dist}_{\mathcal M}\bigl(w^0,\mathcal T^{-1}(0)\bigr),\\
\|w^k-\hat w^k\|_{\mathcal M}
&\le \frac{2}{(1+\gamma)(k+1)}\,\mathrm{dist}_{\mathcal M}\bigl(w^0,\mathcal T^{-1}(0)\bigr).
\end{aligned}
\end{equation}
\end{proposition}

\begin{proof}
By Theorem~\ref{thm:sec32_relaxation}, $\mathcal S_\gamma$ is nonexpansive in
the $\mathcal M$-metric. Here and below,
\(\operatorname{Fix}(R):=\{w:\,Rw=w\}\) denotes the fixed-point set of a map
\(R\). Since $1+\gamma>0$,
\[
w=\mathcal S_\gamma(w)
\iff (1+\gamma)\bigl(\widehat{\mathcal T}_{\rm lin}(w)-w\bigr)=0
\iff w=\widehat{\mathcal T}_{\rm lin}(w),
\]
which gives $\operatorname{Fix}(\mathcal S_\gamma)=\operatorname{Fix}(\widehat{\mathcal T}_{\rm lin})$.
By the definition of $\widehat{\mathcal T}_{\rm lin}$,
\[
\begin{aligned}
w=\widehat{\mathcal T}_{\rm lin}(w)
&\iff \mathcal Mw-\mathcal B_h(w)\in \mathcal Mw+\mathcal A_h(w)\\
&\iff 0\in(\mathcal A_h+\mathcal B_h)(w)=\mathcal T(w),
\end{aligned}
\]
therefore $\operatorname{Fix}(\mathcal S_\gamma)=\mathcal T^{-1}(0)$.

Algorithm~\ref{alg:halpern_pfbs} is precisely the Halpern iteration associated
with the nonexpansive map $\mathcal S_\gamma$ in the Hilbert space
$(\mathbb R^n\times\mathbb R^m,\langle\cdot,\cdot\rangle_{\mathcal M})$,
with anchor $w^0$ and weights $1/(k+2)$. The Halpern convergence theorem for
nonexpansive mappings therefore yields $w^k\to w^\ast$ for some
$w^\ast\in\mathcal T^{-1}(0)$; see, e.g.,
\cite{halpern1967fixed,lieder2021convergence}. Continuity of
$\mathcal S_\gamma$ gives $\bar w^k=\mathcal S_\gamma(w^k)\to w^\ast$, and
\[
\hat w^k=\frac{1}{1+\gamma}(\bar w^k+\gamma w^k)
\]
then implies $\hat w^k\to w^\ast$, proving \eqref{eq:pfbs_conv}.

The corresponding Halpern residual estimate \cite{lieder2021convergence} gives
\[
\|w^k-\bar w^k\|_{\mathcal M}
\le \frac{2}{k+1}\,\mathrm{dist}_{\mathcal M}\bigl(w^0,\mathcal T^{-1}(0)\bigr).
\]
Since $w^k-\bar w^k=(1+\gamma)(w^k-\hat w^k)$, the second estimate in
\eqref{eq:pfbs_residuals} follows immediately.
\end{proof}

It remains to translate the fixed-point residual estimate into the KKT residual
and the objective-value gap of the original constrained problem. Set
$\Phi:=f+g$, and recall from \eqref{eq:kkt_mapping_residual} the original KKT
mapping $\mathcal R$ and residual $r_{\mathrm{KKT}}$.
For $w=(x,y)$, we also write
$r_{\mathrm{KKT}}(w):=r_{\mathrm{KKT}}(x,y)$.
Let $\mathcal E_\sigma$ denote the linear map
$\mathcal E_\sigma(u,v):=(u-\sigma A^\top v,v)$.

\begin{theorem}
\label{thm:halpern_pfbs_rate}
Under the assumptions of Proposition~\ref{prop:halpern_pfbs_basic}, including
the standing KKT-nonemptiness condition
$\mathcal T^{-1}(0)\neq\emptyset$, let
$w^\ast=(x^\ast,y^\ast)$ be the common limit from \eqref{eq:pfbs_conv}. Then
$w^\ast\in\mathcal T^{-1}(0)=\mathcal R^{-1}(0)$ and, for every $k\ge0$,
\[
r_{\mathrm{KKT}}(\hat w^k)
\le
\frac{2\|\mathcal E_\sigma\|}{(1+\gamma)(k+1)}
\left(\sqrt{\|\mathcal M\|}+\frac{L_h}{\sqrt{\mu_x}}\right)
\mathrm{dist}_{\mathcal M}\bigl(w^0,\mathcal T^{-1}(0)\bigr),
\]
in particular $r_{\mathrm{KKT}}(\hat w^k)=O(1/k)$. Set
\[
C_\ast:=\sup_{k\ge0}\bigl(\|\hat x^k-x^\ast\|+\|\hat y^k\|\bigr)+\|y^\ast\|.
\]
Then $C_\ast<\infty$ and
\[
\bigl|\Phi(\hat x^k)-\Phi(x^\ast)\bigr|
\le C_\ast\,r_{\mathrm{KKT}}(\hat w^k),
\]
and consequently $|\Phi(\hat x^k)-\Phi(x^\ast)|=O(1/k)$.
\end{theorem}

\begin{proof}
By Proposition~\ref{prop:halpern_pfbs_basic},
$w^k\to w^\ast$ and $\hat w^k\to w^\ast$ for some $w^\ast\in\mathcal T^{-1}(0)$.
Since the second component of $\mathcal T$ is $Ax-b$, the augmentation term
vanishes on $\mathcal T^{-1}(0)$; therefore
$\mathcal T^{-1}(0)=\mathcal R^{-1}(0)$.

From the definition of $\widehat{\mathcal T}_{\rm lin}$,
\[
\hat w^k=(\mathcal M+\mathcal A_h)^{-1}\bigl(\mathcal Mw^k-\mathcal B_h(w^k)\bigr),
\]
and therefore
\[
\mathcal Mw^k-\mathcal B_h(w^k)\in \mathcal M\hat w^k+\mathcal A_h(\hat w^k).
\]
It follows that
\[
d^k:=\mathcal M(w^k-\hat w^k)+\mathcal B_h(\hat w^k)-\mathcal B_h(w^k)
\in \mathcal T(\hat w^k),
\]
and therefore $\mathrm{dist}(0,\mathcal T(\hat w^k))\le \|d^k\|$. Since
\[
\|\mathcal B_h(\hat w^k)-\mathcal B_h(w^k)\|
=\|\nabla h(\hat x^k)-\nabla h(x^k)\|
\le L_h\|\hat x^k-x^k\|
\le \frac{L_h}{\sqrt{\mu_x}}\|w^k-\hat w^k\|_{\mathcal M},
\]
where the last step uses \eqref{eq:sec32_xdom}, and
\[
\|\mathcal M(w^k-\hat w^k)\|\le \sqrt{\|\mathcal M\|}\,\|w^k-\hat w^k\|_{\mathcal M},
\]
we obtain
\[
\mathrm{dist}\bigl(0,\mathcal T(\hat w^k)\bigr)
\le \left(\sqrt{\|\mathcal M\|}+\frac{L_h}{\sqrt{\mu_x}}\right)
\|w^k-\hat w^k\|_{\mathcal M}.
\]
Combining this with \eqref{eq:pfbs_residuals} yields
\[
\mathrm{dist}\bigl(0,\mathcal T(\hat w^k)\bigr)
\le \frac{2}{(1+\gamma)(k+1)}
\left(\sqrt{\|\mathcal M\|}+\frac{L_h}{\sqrt{\mu_x}}\right)
\mathrm{dist}_{\mathcal M}\bigl(w^0,\mathcal T^{-1}(0)\bigr).
\]
Now, for any $(x,y)$ and any $(u,v)\in\mathcal T(x,y)$,
\[
u\in \nabla f(x)+\partial g(x)-A^\top y+\sigma A^\top(Ax-b),
\qquad v=Ax-b,
\]
we have $\mathcal E_\sigma(u,v)\in\mathcal R(x,y)$. Therefore
\[
r_{\mathrm{KKT}}(\hat w^k)
\le \|\mathcal E_\sigma\|\,\mathrm{dist}\bigl(0,\mathcal T(\hat w^k)\bigr),
\]
which proves the displayed residual estimate.

For the objective-value bound, the preceding inclusion implies
$\mathcal R(\hat w^k)\neq\emptyset$. This set is closed and convex in finite
dimensions; choose a least-norm element
$r^k=(r_x^k,r_y^k)\in\mathcal R(\hat w^k)$. Then
$\|r^k\|=r_{\mathrm{KKT}}(\hat w^k)$,
$r_y^k=A\hat x^k-b$ and
$r_x^k\in\partial\Phi(\hat x^k)-A^\top\hat y^k$. Thus there exists
$\zeta^k\in\partial\Phi(\hat x^k)$ with
$\zeta^k=A^\top\hat y^k+r_x^k$. By convexity of $\Phi$,
\[
\begin{aligned}
\Phi(\hat x^k)-\Phi(x^\ast)
&\le \langle \zeta^k,\hat x^k-x^\ast\rangle \\
&=\langle r_x^k,\hat x^k-x^\ast\rangle+\langle \hat y^k,A\hat x^k-b\rangle \\
&\le (\|\hat x^k-x^\ast\|+\|\hat y^k\|)r_{\mathrm{KKT}}(\hat w^k).
\end{aligned}
\]
Likewise, since $w^\ast\in\mathcal R^{-1}(0)$, there exists
$\zeta^\ast\in\partial\Phi(x^\ast)$ such that $\zeta^\ast=A^\top y^\ast$, and
convexity gives
\[
\Phi(\hat x^k)-\Phi(x^\ast)
\ge \langle \zeta^\ast,\hat x^k-x^\ast\rangle
=\langle y^\ast,A\hat x^k-b\rangle
\ge -\|y^\ast\|r_{\mathrm{KKT}}(\hat w^k).
\]
Since $\hat w^k\to w^\ast$, the constant $C_\ast$ is finite. Combining the
upper and lower bounds yields the displayed objective-value estimate, and the
$O(1/k)$ conclusion follows from the residual estimate.
\end{proof}

\begin{remark}[Worst-case tightness]
The nonergodic residual bound in Theorem~\ref{thm:halpern_pfbs_rate} is
worst-case tight even within the linearized PFBS family. Fix
$\gamma\in(-1,1)$ and choose $\eta\in(0,1)$ such that
$(1+\gamma)\eta<1$ and $\gamma< 1-\eta/2$. Consider the scalar instance
with $g=0$, $A=0$, $b=0$, $\sigma=0$, $P=1$, $Q=1$, and
\[
        f(x)=h(x)=\frac{\eta}{2}x^2 .
\]
Then $L_h=\eta$, $\mu_x=1$, and the admissible relaxation condition is
satisfied. The PFBS shadow map is
$F(x)=\widehat{\mathcal T}_{\rm lin}(x)=(1-\eta)x$; the dual coordinate remains
zero if initialized at zero. Starting from $x^0\ne0$, the reflected Halpern
recursion gives, with
$q:=1-(1+\gamma)\eta$,
\[
        x^k=\frac{1-q^{k+1}}{(1+\gamma)\eta(k+1)}x^0,
        \qquad
        \hat x^k=(1-\eta)x^k .
\]
The KKT residual at the shadow point is
$r_{\mathrm{KKT}}(\hat x^k,0)=\eta|\hat x^k|$. Therefore
\[
        \lim_{k\to\infty}(k+1)r_{\mathrm{KKT}}(\hat x^k,0)
        =
        \frac{(1-\eta)|x^0|}{1+\gamma}>0.
\]
Thus the global nonergodic KKT-residual estimate cannot, in general, be
improved to $o(1/k)$.
\end{remark}

\section{Shadow Identification and Reduced Geometry for Reflected Halpern Schemes}
\label{sec4}

For reflected Halpern schemes, the relevant local object is the shadow point,
not the anchored Halpern state. Let
$W^\star:=\mathcal T^{-1}(0)$ denote the KKT set. For
$F\in\{\widehat{\mathcal T},\widehat{\mathcal T}_{\rm lin}\}$, a reflected
Halpern step has the form
\[
        \hat w^k=F(w^k),\qquad
        \bar w^k=(1+\gamma)\hat w^k-\gamma w^k,
        \qquad
        w^{k+1}=\frac{1}{k+2}w^0+\frac{k+1}{k+2}\bar w^k.
\]
Thus $w^{k+1}$ is obtained by anchoring and reflecting the shadow
$\hat w^k$. Even when $\hat x^k$ lies on the active manifold, the
state $x^{k+1}$ need not lie on that manifold.

The local analysis is therefore carried out at the shadow points. We prove
finite identification of $\hat x^k$ for both fixed-point maps, derive an
exact reduced residual identity under an affine-face model, and describe the
perturbed reduced dynamics retained by the anchored Halpern state.

Throughout this section, all norms and distances are Euclidean unless another
metric is explicitly specified.

\subsection{Shadow Certificates and Finite Identification}
\label{subsec:shadow-certificates-identification}

We first record that the augmented term changes the geometry of the operator,
but not the solution set. Recall from \eqref{eq:kkt_mapping_residual} the
original KKT mapping \(\mathcal R\), and define the augmented correction by
\[
        \mathcal C_\sigma(x,y):=
        \binom{\sigma A^\top(Ax-b)}{0}.
\]
Then $\mathcal T=\mathcal R+\mathcal C_\sigma$.

\begin{proposition}
\label{prop:augmented-term-preserves-kkt-set}
One has $\mathcal T^{-1}(0)=\mathcal R^{-1}(0)=W^\star$. Moreover, for every
$w=(x,y)\in\mathbb R^n\times\mathbb R^m$ and every
$w^\star=(x^\star,y^\star)\in W^\star$,
\[
        \big\langle \mathcal C_\sigma(w),w-w^\star\big\rangle
        =\sigma\|Ax-b\|^2.
\]
\end{proposition}

\begin{proof}
Since $\mathcal T=\mathcal R+\mathcal C_\sigma$, and since the common second
component $Ax-b$ forces $\mathcal C_\sigma(x,y)=0$ at every zero of either
operator, their zero sets coincide. If
$w^\star=(x^\star,y^\star)\in W^\star$, then $Ax^\star=b$, and
therefore
\[
        \big\langle \mathcal C_\sigma(w),w-w^\star\big\rangle
        =\sigma\big\langle A^\top(Ax-b),x-x^\star\big\rangle
        =\sigma\|Ax-b\|^2.
\]
\end{proof}

The next proposition derives the subgradient certificates generated by the
shadow point. The statement is formulated simultaneously for the exact map and
the linearized map.
For $F\in\{\widehat{\mathcal T},\widehat{\mathcal T}_{\rm lin}\}$ and
$\hat w=F(w)$, write $w=(x,y)$ and
$\hat w=(\hat x,\hat y)$. Define the correction term
\[
        \Delta_F(w,\hat w)
        :=
        \begin{cases}
        0,
        & F=\widehat{\mathcal T},\\[1mm]
        \nabla h(\hat x)-\nabla h(x),
        & F=\widehat{\mathcal T}_{\rm lin}.
        \end{cases}
\]

\begin{proposition}
\label{prop:shadow-subgradient-certificates}
Let $F\in\{\widehat{\mathcal T},\widehat{\mathcal T}_{\rm lin}\}$, let
$w=(x,y)$, and let $\hat w=(\hat x,\hat y)=F(w)$. Then there exists
$v_F(w)\in\partial g(\hat x)$ such that
\begin{equation}
\label{eq:shadow-certificate}
        v_F(w)
        =
        A^\top y
        -\nabla f(\hat x)
        -\sigma A^\top(A\hat x-b)
        -P(\hat x-x)
        +\Delta_F(w,\hat w).
\end{equation}
Consequently, if $w^j\to w^\star=(x^\star,y^\star)\in W^\star$ and
$\hat w^j=F(w^j)\to w^\star$, then
\begin{equation}
\label{eq:shadow-certificate-limit}
        v_F(w^j)\to
        v^\star:=A^\top y^\star-\nabla f(x^\star)
        \in\partial g(x^\star).
\end{equation}
\end{proposition}

\begin{proof}
We first consider the exact dPPM map. Since
\[
        \hat w=(\mathcal M+\mathcal T)^{-1}\mathcal M w,
\]
we have
\[
        \mathcal M w\in \mathcal M\hat w+\mathcal T(\hat w).
\]
Comparing the primal component gives
\[
        0\in
        \nabla f(\hat x)+\partial g(\hat x)
        -A^\top y
        +\sigma A^\top(A\hat x-b)
        +P(\hat x-x).
\]
Thus there exists $v_F(w)\in\partial g(\hat x)$ satisfying
\[
        v_F(w)
        =
        A^\top y
        -\nabla f(\hat x)
        -\sigma A^\top(A\hat x-b)
        -P(\hat x-x),
\]
which is exactly \eqref{eq:shadow-certificate} with $\Delta_F=0$.

For the linearized PFBS map, the defining relation is
\[
        (\mathcal M-\mathcal B_h)w\in (\mathcal M+\mathcal A_h)\hat w.
\]
Equivalently, the primal optimality condition reads
\[
        0\in
        \partial g(\hat x)
        +\nabla f(\hat x)-\nabla h(\hat x)
        +\nabla h(x)
        -A^\top y
        +\sigma A^\top(A\hat x-b)
        +P(\hat x-x).
\]
Therefore there exists $v_F(w)\in\partial g(\hat x)$ such that
\[
        v_F(w)
        =
        A^\top y
        -\nabla f(\hat x)
        +\nabla h(\hat x)-\nabla h(x)
        -\sigma A^\top(A\hat x-b)
        -P(\hat x-x),
\]
which is again \eqref{eq:shadow-certificate}.

Now assume $w^j\to w^\star\in W^\star$ and
$\hat w^j\to w^\star$. Then
\[
        \hat x^j\to x^\star,
        \qquad
        y^j\to y^\star,
        \qquad
        A\hat x^j-b\to 0,
        \qquad
        \hat x^j-x^j\to 0.
\]
In the linearized case, continuity of $\nabla h$ gives
\(\nabla h(\hat x^j)-\nabla h(x^j)\to 0\). Passing to the limit in
\eqref{eq:shadow-certificate} yields
\(v_F(w^j)\to A^\top y^\star-\nabla f(x^\star)\). Since
$w^\star\in W^\star$, the KKT condition gives
\(A^\top y^\star-\nabla f(x^\star)\in\partial g(x^\star)\). This proves
\eqref{eq:shadow-certificate-limit}.
\end{proof}

We use the standard notion of partial smoothness for convex functions in the
sense of Lewis~\cite{lewis2002active}, with the equivalent tangent-space
form of normal sharpness in the convex setting. Here a
\(C^2\)-manifold means an embedded manifold with twice continuously
differentiable local charts; \(T_{\mathcal M}(x)\) denotes the tangent space to
a manifold \(\mathcal M\) at \(x\), \(\operatorname{par} S:=\operatorname{span}(S-S)\)
denotes the parallel subspace of a convex set \(S\), and \(\operatorname{ri} S\)
denotes its relative interior. Let $g$ be proper, closed, and convex, and let
$x^\star\in\operatorname{dom}g$. We say that $g$ is partly smooth at
$x^\star$ relative to a $C^2$-manifold $\mathcal M_g(x^\star)$ if, locally
around $x^\star$, the restriction of $g$ to $\mathcal M_g(x^\star)$ is
$C^2$, the sharpness relation
\[
        T_{\mathcal M_g(x^\star)}(x^\star)
        =
        \operatorname{par}(\partial g(x^\star))^\perp
\]
holds, and the subdifferential mapping $\partial g$ is continuous at
$x^\star$ along $\mathcal M_g(x^\star)$ in the Painlev{\'e}--Kuratowski sense.
The manifold $\mathcal M_g(x^\star)$ is then called the active manifold of
$g$ at $x^\star$.
The following assumption is the corresponding standard identifiability
condition at the limiting KKT point; see, e.g.,
\cite{hare2004identifying}.

\begin{assumption}[Partial smoothness and nondegeneracy]
\label{ass:shadow-identification}
Let $w^\star=(x^\star,y^\star)\in W^\star$. Assume that
\begin{enumerate}
\item[(i)] $g$ is proper, closed, convex, and partly smooth at $x^\star$ relative
to a $C^2$-manifold $\mathcal M_g(x^\star)$;
\item[(ii)] the nondegeneracy condition
\(A^\top y^\star-\nabla f(x^\star)
\in \operatorname{ri}\partial g(x^\star)\) holds.
\end{enumerate}
\end{assumption}

This assumption is local at the limiting KKT point. It holds for many standard
regularizers and constraints under the usual strict-complementarity condition:
for instance, for an indicator of a polyhedron or a box, the manifold is the
identified face and the relative-interior condition is the standard
nondegeneracy condition on the active multipliers.

The next theorem applies the standard identification principle for partly
smooth convex functions to the shadow sequence.

\begin{theorem}
\label{thm:finite-identification-shadow-sequence}
Let $F\in\{\widehat{\mathcal T},\widehat{\mathcal T}_{\rm lin}\}$, and let
$\{w^j\}_{j\ge0}$ be any sequence and define the shadow sequence
\[
        \hat w^j=F(w^j).
\]
Assume that
\begin{equation}
\label{eq:state-shadow-convergence}
        w^j\to w^\star,
        \qquad
        \hat w^j\to w^\star
\end{equation}
for some $w^\star=(x^\star,y^\star)\in W^\star$. If
Assumption~\ref{ass:shadow-identification} holds at $w^\star$, then there
exists an integer $J_{\rm id}\ge0$ such that
\[
        \hat x^j\in \mathcal M_g(x^\star),
        \qquad
        \forall j\ge J_{\rm id}.
\]
\end{theorem}

\begin{proof}
By Proposition~\ref{prop:shadow-subgradient-certificates}, for every $j$ there
exists
\[
        v_F(w^j)\in\partial g(\hat x^j)
\]
such that
\[
        v_F(w^j)\to v^\star:=A^\top y^\star-\nabla f(x^\star).
\]
Moreover, \eqref{eq:state-shadow-convergence} gives
\[
        \hat x^j\to x^\star.
\]
By Assumption~\ref{ass:shadow-identification}(ii),
\[
        v^\star\in\operatorname{ri}\partial g(x^\star).
\]
Thus the graph points
\[
        (\hat x^j,v_F(w^j))\in\operatorname{gph}\partial g
\]
converge to
\[
        (x^\star,v^\star),
        \qquad
        v^\star\in\operatorname{ri}\partial g(x^\star).
\]
The standard identification theorem for partly smooth convex functions
\cite{lewis2002active,hare2004identifying} then implies that
\[
        \hat x^j\in\mathcal M_g(x^\star)
\]
for all sufficiently large $j$.
\end{proof}

\begin{remark}[Application to reflected and restarted Halpern iterates]
\label{rem:identification-applies-to-halpern}
Theorem~\ref{thm:finite-identification-shadow-sequence} is intentionally stated
for an arbitrary sequence $\{w^j\}$. For the reflected Halpern schemes of
Sections~2 and~3, the convergence condition
\eqref{eq:state-shadow-convergence} is supplied by the global convergence
results for the corresponding nonexpansive reflected maps. For restarted
schemes, the same theorem applies to any globally indexed restarted trajectory
whenever the state and shadow sequences converge to the same KKT point. The
conclusion is always identification of the shadow points $\hat x^j$, not
identification of the Halpern states $x^j$.
\end{remark}

\begin{remark}[The Halpern state need not identify in finite time]
\label{rem:halpern-state-need-not-identify}
Finite identification of the Halpern state $x^j$ is false in general. Consider
the one-dimensional problem with no linear constraint,
\[
        f\equiv 0,
        \qquad
        g=\delta_{\{0\}}.
\]
The active manifold is $\mathcal M_g=\{0\}$, and the proximal map satisfies
$F(x)=0$ for every $x$. For the ordinary Halpern choice $\gamma=0$,
\[
        \hat x^k=0,
        \qquad
        \bar x^k=0,
        \qquad
        x^{k+1}=\frac{1}{k+2}x^0.
\]
Thus $\hat x^k\in\mathcal M_g$ for every $k$, but if $x^0\neq0$, then
$x^k\notin\mathcal M_g$ for every finite $k$. This shows that the shadow
sequence is the correct identification object for Halpern-type methods.
\end{remark}

\subsection{Reduced Residuals on the Identified Affine Face}
\label{sec:reduced_geometry_after_identification}

Partial smoothness yields finite identification of a smooth active manifold. To
obtain an exact reduced residual identity, we impose the following stronger
local model: the identified manifold is locally an affine face and the
subdifferential has a full normal-space fiber on that face. This is a sufficient
model for the exact identity below, not a standing assumption for all structured
nonsmooth terms.

\begin{assumption}[Affine-face model]
\label{ass:affine-face-model}
Let $w^\star=(x^\star,y^\star)\in W^\star$. Assume that there exist an affine
set $\mathcal F\subset\mathbb R^n$ with $x^\star\in\mathcal F$, a neighborhood
$U_x$ of $x^\star$, a symmetric matrix $H_{\mathcal F}\succeq0$, and a vector
$c_{\mathcal F}\in\mathbb R^n$ such that, writing
\[
        \mathcal L_{\mathcal F}:=\operatorname{span}(\mathcal F-\mathcal F),
        \qquad
        \Pi_{\mathcal F}:\mathbb R^n\to\mathcal L_{\mathcal F}
\]
for the tangent space and its orthogonal projector, one has
\[
        \partial g(x)
        =
        H_{\mathcal F}x+c_{\mathcal F}+\mathcal L_{\mathcal F}^{\perp},
        \qquad
        \forall x\in\mathcal F\cap U_x.
\]
\end{assumption}

\begin{remark}[Role of the affine-face model]
\label{rem:role-affine-face-model}
Assumption~\ref{ass:affine-face-model} is satisfied, for example, by functions
of the form \(g(x)=\frac12\langle x,Hx\rangle+\langle c,x\rangle+\delta_{\mathcal F}(x)\)
on an affine set \(\mathcal F\), and more generally by piecewise
linear-quadratic terms on a fixed affine region whose subdifferential fiber is
an affine translate of the full normal space. The assumption is stronger than
partial smoothness and is not intended to cover all polyhedral or piecewise
linear-quadratic models. It is used to obtain the exact reduced residual
identity below. For box constraints and more general polyhedral constraints,
the sharpness verification in Section~\ref{sec:sharpness_verification}
proceeds through metric subregularity or Hoffman--Robinson error bounds rather
than through this exact identity.
\end{remark}

For $u\in\mathcal L_{\mathcal F}$, define the affine embedding
\[
        \Psi(u,y):=(x^\star+u,y),
\]
and, for $(u,y)\in\mathcal L_{\mathcal F}\times\mathbb R^m$, define the
reduced KKT mapping
\begin{equation}
\label{eq:reduced-kkt-mapping}
        \mathcal T_{\rm red}(u,y)
        :=
        \begin{pmatrix}
        \begin{aligned}
        \Pi_{\mathcal F}\Big(&
        \nabla f(x^\star+u)+H_{\mathcal F}(x^\star+u)+c_{\mathcal F}\\
        &{}-A^\top y+\sigma A^\top A u
        \Big)
        \end{aligned}\\[1mm]
        Au
        \end{pmatrix}.
\end{equation}
Here we used $Ax^\star=b$, which gives
\[
        A(x^\star+u)-b=Au.
\]

Under the affine-face model, the full augmented KKT residual at points on the
face has an exact reduced representation.

\begin{theorem}
\label{thm:reduced-residual-identity-affine-face}
Suppose Assumption~\ref{ass:affine-face-model} holds. Then there exists a
neighborhood $U_{\rm red}$ of
$z^\star:=(0,y^\star)$ in
$\mathcal L_{\mathcal F}\times\mathbb R^m$ such that, for every
$z=(u,y)\in U_{\rm red}$,
\begin{equation}
\label{eq:exact-reduced-residual-identity}
        \operatorname{dist}\bigl(0,\mathcal T(\Psi(z))\bigr)
        =
        \|\mathcal T_{\rm red}(z)\|.
\end{equation}
\end{theorem}

\begin{proof}
By Assumption~\ref{ass:affine-face-model}, after shrinking $U_{\rm red}$ if
necessary, every $z=(u,y)\in U_{\rm red}$ satisfies
\(x:=x^\star+u\in\mathcal F\cap U_x\). Define
\[
        a(z)
        :=
        \nabla f(x^\star+u)
        +H_{\mathcal F}(x^\star+u)
        +c_{\mathcal F}
        -A^\top y
        +\sigma A^\top A u.
\]
Since $x=x^\star+u\in\mathcal F$, the affine-face representation gives
\[
        \mathcal T(\Psi(z))
        =
        \begin{pmatrix}
        a(z)+\mathcal L_{\mathcal F}^{\perp}\\
        Au
        \end{pmatrix},
        \qquad
        \mathcal T_{\rm red}(z)
        =
        \begin{pmatrix}
        \Pi_{\mathcal F}a(z)\\
        Au
        \end{pmatrix}.
\]
For any $n\in\mathcal L_{\mathcal F}^{\perp}$,
\[
        \|a(z)+n\|^2
        =
        \|\Pi_{\mathcal F}a(z)\|^2
        +\|(I-\Pi_{\mathcal F})a(z)+n\|^2
        \ge
        \|\Pi_{\mathcal F}a(z)\|^2.
\]
Equality is attained at
\(n=-(I-\Pi_{\mathcal F})a(z)\in\mathcal L_{\mathcal F}^{\perp}\). Therefore
\(\operatorname{dist}(0,a(z)+\mathcal L_{\mathcal F}^{\perp})
=\|\Pi_{\mathcal F}a(z)\|\).
Combining this identity with the product structure of $\mathcal T(\Psi(z))$ yields
\[
        \operatorname{dist}\bigl(0,\mathcal T(\Psi(z))\bigr)^2
        =
        \|\Pi_{\mathcal F}a(z)\|^2+
        \|Au\|^2
        =
        \|\mathcal T_{\rm red}(z)\|^2,
\]
which proves \eqref{eq:exact-reduced-residual-identity}.
\end{proof}

\begin{corollary}
\label{cor:reduced-error-bound-transfer}
Let $U_{\rm red}$ be as in
Theorem~\ref{thm:reduced-residual-identity-affine-face}, and set
\[
        Z_{\rm red}:=U_{\rm red}\cap \mathcal T_{\rm red}^{-1}(0).
\]
Then $z^\star\in Z_{\rm red}$, and therefore $Z_{\rm red}$ is nonempty. Moreover,
\[
        z\in Z_{\rm red}
        \quad\Longrightarrow\quad
        \Psi(z)\in W^\star,
\]
and
\begin{equation}
\label{eq:embedding-distance-bound}
        \operatorname{dist}\bigl(\Psi(z),W^\star\bigr)
        \le
        \operatorname{dist}\bigl(z,Z_{\rm red}\bigr),
        \qquad
        \forall z\in U_{\rm red}.
\end{equation}
Consequently, if there exist a neighborhood $V_{\rm red}\subset U_{\rm red}$
of $z^\star$ in $\mathcal L_{\mathcal F}\times\mathbb R^m$ and a constant
$\alpha_{\rm red}>0$ such that
\begin{equation}
\label{eq:reduced-error-bound}
        \alpha_{\rm red}\operatorname{dist}(z,Z_{\rm red})
        \le
        \|\mathcal T_{\rm red}(z)\|,
        \qquad
        \forall z\in V_{\rm red},
\end{equation}
then
\[
        \alpha_{\rm red}\operatorname{dist}(w,W^\star)
        \le
        \operatorname{dist}(0,\mathcal T(w)),
        \qquad
        \forall w\in\Psi(V_{\rm red}).
\]
\end{corollary}

\begin{proof}
For $z=(u,y)\in U_{\rm red}$, set
\[
        a(z):=
        \nabla f(x^\star+u)
        +H_{\mathcal F}(x^\star+u)
        +c_{\mathcal F}
        -A^\top y
        +\sigma A^\top A u.
\]
Because $w^\star\in W^\star$ and
$x^\star\in\mathcal F\cap U_x$, the affine-face model and the KKT condition
imply
\[
        a(z^\star)
        =
        \nabla f(x^\star)+H_{\mathcal F}x^\star+c_{\mathcal F}
        -A^\top y^\star
        \in\mathcal L_{\mathcal F}^{\perp}.
\]
The second component of $\mathcal T_{\rm red}(z^\star)$ is also zero. Thus
$\mathcal T_{\rm red}(z^\star)=0$, and $z^\star\in Z_{\rm red}$.

Suppose $z\in Z_{\rm red}$. Then
\[
        \Pi_{\mathcal F}a(z)=0,
        \qquad
        Au=0.
\]
The first condition implies $a(z)\in\mathcal L_{\mathcal F}^{\perp}$, and therefore
\[
        0\in a(z)+\mathcal L_{\mathcal F}^{\perp}.
\]
Together with $Au=0$, this gives $0\in \mathcal T(\Psi(z))$, and therefore
$\Psi(z)\in W^\star$.

Because $\Psi$ is an isometry from
$\mathcal L_{\mathcal F}\times\mathbb R^m$ into
$\mathbb R^n\times\mathbb R^m$, taking the infimum over $\bar z\in Z_{\rm red}$
gives
\[
        \operatorname{dist}\bigl(\Psi(z),W^\star\bigr)
        \le
        \operatorname{dist}\bigl(z,Z_{\rm red}\bigr),
\]
which proves \eqref{eq:embedding-distance-bound}.

Finally, if \eqref{eq:reduced-error-bound} holds and $w=\Psi(z)$ with
$z\in V_{\rm red}$, then by
\eqref{eq:embedding-distance-bound}, \eqref{eq:reduced-error-bound}, and
\eqref{eq:exact-reduced-residual-identity},
\[
        \alpha_{\rm red}\operatorname{dist}(w,W^\star)
        \le
        \alpha_{\rm red}\operatorname{dist}(z,Z_{\rm red})
        \le
        \|\mathcal T_{\rm red}(z)\|
        =
        \operatorname{dist}(0,\mathcal T(w)).
\]
This proves the claimed face-restricted error bound.
\end{proof}

The reduced error bound \eqref{eq:reduced-error-bound} may hold even when the
local solution set is not isolated. The following Jacobian condition is a simple
sufficient condition for the isolated case.

\begin{proposition}
\label{prop:reduced-jacobian-sufficient-condition}
Suppose Assumption~\ref{ass:affine-face-model} holds and $f$ is twice
continuously differentiable near $x^\star$. Then $\mathcal T_{\rm red}$ is
continuously differentiable near $z^\star=(0,y^\star)$, with Jacobian
\[
        D\mathcal T_{\rm red}(z^\star)
        =
        \begin{pmatrix}
        \Pi_{\mathcal F}
        \bigl(\nabla^2 f(x^\star)+H_{\mathcal F}
        +\sigma A^\top A\bigr)\Pi_{\mathcal F}
        &
        -\Pi_{\mathcal F}A^\top\\[1mm]
        A\Pi_{\mathcal F}
        &
        0
        \end{pmatrix},
\]
viewed as a linear operator on
$\mathcal L_{\mathcal F}\times\mathbb R^m$. If this operator is nonsingular,
then there exist a neighborhood $V_{\rm red}$ of $z^\star$ in
$\mathcal L_{\mathcal F}\times\mathbb R^m$ and a constant $\alpha_{\rm red}>0$
such that
\[
        \alpha_{\rm red}\|z-z^\star\|
        \le
        \|\mathcal T_{\rm red}(z)\|,
        \qquad
        \forall z\in V_{\rm red}.
\]
In particular, after possibly shrinking $V_{\rm red}$,
$\mathcal T_{\rm red}$ satisfies \eqref{eq:reduced-error-bound} with
$Z_{\rm red}=\{z^\star\}$.
\end{proposition}

\begin{proof}
The formula \eqref{eq:reduced-kkt-mapping} shows that $\mathcal T_{\rm red}$ is
continuously differentiable near $z^\star$ whenever $f$ is $C^2$ near
$x^\star$. Differentiating the two components gives the displayed Jacobian.

Because $w^\star\in W^\star$ and Assumption~\ref{ass:affine-face-model}
holds at $x^\star$, the vector
\[
        \nabla f(x^\star)+H_{\mathcal F}x^\star+c_{\mathcal F}
        -A^\top y^\star
\]
belongs to $\mathcal L_{\mathcal F}^{\perp}$, and $Ax^\star=b$. Thus
$\mathcal T_{\rm red}(z^\star)=0$.

If $D\mathcal T_{\rm red}(z^\star)$ is nonsingular, then by the inverse
function theorem $\mathcal T_{\rm red}$ is locally invertible around
$z^\star$, and its local inverse is Lipschitz. Since
$\mathcal T_{\rm red}(z^\star)=0$, there exist $V_{\rm red}$ and
$c_{\rm red}>0$ such that
\[
        \|z-z^\star\|
        \le
        c_{\rm red}\|\mathcal T_{\rm red}(z)-\mathcal T_{\rm red}(z^\star)\|
        =
        c_{\rm red}\|\mathcal T_{\rm red}(z)\|,
        \qquad
        \forall z\in V_{\rm red}.
\]
Taking $\alpha_{\rm red}=1/c_{\rm red}$ gives the claimed local error bound.
\end{proof}

Combining finite identification with the reduced residual identity gives the
following consequence for the identified shadow points.

\begin{corollary}
\label{cor:identified-shadow-reduced-residuals}
Let $F\in\{\widehat{\mathcal T},\widehat{\mathcal T}_{\rm lin}\}$, and let
$\hat w^j=F(w^j)$ be a shadow sequence satisfying the assumptions of
Theorem~\ref{thm:finite-identification-shadow-sequence}. Suppose that
Assumption~\ref{ass:affine-face-model} holds and that the identified manifold
locally coincides with the affine set $\mathcal F$. Then there exists
$J_{\rm red}\ge0$ such that, for all $j\ge J_{\rm red}$,
\[
        \hat x^j\in\mathcal F,
        \qquad
        \hat z^j:=(\hat x^j-x^\star,\hat y^j)
        \in\mathcal L_{\mathcal F}\times\mathbb R^m,
\]
and
\begin{equation}
\label{eq:identified-shadow-reduced-residual-identity}
        \operatorname{dist}\bigl(0,\mathcal T(\hat w^j)\bigr)
        =
        \|\mathcal T_{\rm red}(\hat z^j)\|,
        \qquad
        \forall j\ge J_{\rm red}.
\end{equation}
\end{corollary}

\begin{proof}
Finite identification gives $\hat x^j\in\mathcal F$ for all sufficiently
large $j$. Since $\hat w^j\to w^\star$, increasing $J_{\rm red}$ if
necessary also gives $\hat z^j\in U_{\rm red}$ for all $j\ge J_{\rm red}$.
Therefore $\hat w^j=\Psi(\hat z^j)$ for all large $j$, and
\eqref{eq:identified-shadow-reduced-residual-identity} follows from
Theorem~\ref{thm:reduced-residual-identity-affine-face}.
\end{proof}

\begin{remark}[Residual control by fixed-point steps]
The same inclusions give a direct residual control by the fixed-point step. In
the exact dPPM case $F=\widehat{\mathcal T}$,
\[
        \hat w^j=(\mathcal M+\mathcal T)^{-1}\mathcal M w^j
\]
implies
\[
        \mathcal M(w^j-\hat w^j)\in \mathcal T(\hat w^j).
\]
Therefore
\[
        \|\mathcal T_{\rm red}(\hat z^j)\|
        =
        \operatorname{dist}\bigl(0,\mathcal T(\hat w^j)\bigr)
        \le
        \|\mathcal M(w^j-\hat w^j)\|
        \le
        \|\mathcal M\|\,\|w^j-\hat w^j\|,
        \qquad j\ge J_{\rm red}.
\]
In the linearized PFBS case $F=\widehat{\mathcal T}_{\rm lin}$,
\[
        \hat w^j
        =(\mathcal M+\mathcal A_h)^{-1}(\mathcal M-\mathcal B_h)w^j
\]
implies
\[
        \mathcal M(w^j-\hat w^j)
        +\mathcal B_h(\hat w^j)-\mathcal B_h(w^j)
        \in
        (\mathcal A_h+\mathcal B_h)(\hat w^j)
        =\mathcal T(\hat w^j).
\]
Therefore
\[
        \|\mathcal T_{\rm red}(\hat z^j)\|
        =
        \operatorname{dist}\bigl(0,\mathcal T(\hat w^j)\bigr)
        \le
        \|\mathcal M(w^j-\hat w^j)
        +\mathcal B_h(\hat w^j)-\mathcal B_h(w^j)\|.
\]
Since $\mathcal B_h(x,y)=(\nabla h(x),0)$ and $\nabla h$ is $L_h$-Lipschitz,
\[
        \|\mathcal B_h(\hat w^j)-\mathcal B_h(w^j)\|
        \le
        L_h\|\hat x^j-x^j\|
        \le
        L_h\|\hat w^j-w^j\|.
\]
Consequently, for all $j\ge J_{\rm red}$,
\[
        \|\mathcal T_{\rm red}(\hat z^j)\|
        \le
        (\|\mathcal M\|+L_h)\|w^j-\hat w^j\|.
\]
\end{remark}

\subsection{Perturbed Reduction of the Halpern State}
\label{subsec:perturbed-reduced-halpern-dynamics}

The previous subsection gives exact reduced residual identities at identified
shadow points. The Halpern state itself is different: it may retain a normal
component relative to the identified affine face. The next proposition records
this distinction in a compact form. It is used only as a local structural
description; the restart analysis in Section~5 is formulated in terms of
fixed-point residuals.

Assume throughout this subsection that Assumption~\ref{ass:affine-face-model}
holds. Let $F\in\{\widehat{\mathcal T},\widehat{\mathcal T}_{\rm lin}\}$ and
let $\hat w^j=F(w^j)$, with $w^j\to w^\star$ and
$\hat w^j\to w^\star$. Suppose that the shadow sequence has already
identified $\mathcal F$ and that $\hat x^j\in\mathcal F$ for all
sufficiently large $j$. For such $j$, decompose the Halpern state as
\[
\begin{aligned}
        x^j&=x^\star+u^j+n^j,\\
        u^j&:=\Pi_{\mathcal F}(x^j-x^\star)
        \in\mathcal L_{\mathcal F},\\
        n^j&:=(I-\Pi_{\mathcal F})(x^j-x^\star)
        \in\mathcal L_{\mathcal F}^{\perp}.
\end{aligned}
\]
Since $\hat x^j\in\mathcal F$, write
\(\hat x^j=x^\star+\hat u^j\), with
\(\hat u^j\in\mathcal L_{\mathcal F}\). Define
\(z^j:=(u^j,y^j)\) and
\(\hat z^j:=(\hat u^j,\hat y^j)\), and introduce the reduced metric
operator
\[
        M_{\rm red}
        :=
        \begin{pmatrix}
        \Pi_{\mathcal F}P\Pi_{\mathcal F} & \Pi_{\mathcal F}A^\top\\[1mm]
        A\Pi_{\mathcal F} & Q
        \end{pmatrix}.
\]
Finally, define the normal defect operator
\[
        D_\perp n
        :=
        \begin{pmatrix}
        \Pi_{\mathcal F}Pn\\[1mm]
        An
        \end{pmatrix},
        \qquad
        n\in\mathcal L_{\mathcal F}^{\perp}.
\]
For the linearized PFBS map, also define the reduced forward operator
\[
        \mathcal B_{{\rm red},h}(u,y)
        :=
        \begin{pmatrix}
        \Pi_{\mathcal F}\nabla h(x^\star+u)\\[1mm]
        0
        \end{pmatrix},
\]
and the reduced backward operator
\[
        \mathcal A_{{\rm red},h}(u,y)
        :=
        \begin{pmatrix}
        \begin{aligned}
        \Pi_{\mathcal F}\Big(&
        \nabla f(x^\star+u)-\nabla h(x^\star+u)
        +H_{\mathcal F}(x^\star+u)\\
        &{}+c_{\mathcal F}-A^\top y+\sigma A^\top A u
        \Big)
        \end{aligned}\\[1mm]
        Au
        \end{pmatrix}.
\]
Then \(\mathcal T_{\rm red}=\mathcal A_{{\rm red},h}+\mathcal B_{{\rm red},h}\).

\begin{proposition}[Perturbed reduced state relations]
\label{prop:perturbed-reduced-state-relations}
For every sufficiently large $j$, the following statements hold.
\begin{enumerate}
\item[(i)] If $F=\widehat{\mathcal T}$, then
\begin{equation}
\label{eq:perturbed-reduced-exact}
        M_{\rm red}z^j
        =
        M_{\rm red}\hat z^j
        +
        \mathcal T_{\rm red}(\hat z^j)
        -
        D_\perp n^j.
\end{equation}
\item[(ii)] If $F=\widehat{\mathcal T}_{\rm lin}$, then
\begin{equation}
\label{eq:perturbed-reduced-linearized}
        M_{\rm red}z^j
        -\mathcal B_{{\rm red},h}(z^j)
        =
        M_{\rm red}\hat z^j
        +\mathcal A_{{\rm red},h}(\hat z^j)
        -D_\perp n^j
        +E_h^j,
\end{equation}
where
\[
        E_h^j
        :=
        \begin{pmatrix}
        \Pi_{\mathcal F}\bigl(
        \nabla h(x^\star+u^j+n^j)
        -\nabla h(x^\star+u^j)
        \bigr)\\[1mm]
        0
        \end{pmatrix}.
\]
If $\nabla h$ is locally Lipschitz near $x^\star$, then
\[
        \|E_h^j\|=O(\|n^j\|).
\]
\end{enumerate}
In both cases, when $n^j=0$ the normal perturbation vanishes and the projected
state relation becomes an autonomous reduced relation.
\end{proposition}

\begin{proof}
Write $\Pi=\Pi_{\mathcal F}$ for readability. We consider only indices for
which $\hat x^j\in\mathcal F$.

For the exact map, the resolvent relation
\[
        \mathcal M w^j\in \mathcal M\hat w^j+\mathcal T(\hat w^j)
\]
gives the primal and dual shadow optimality conditions
\[
\begin{aligned}
0&\in \nabla f(\hat x^j)+\partial g(\hat x^j)
-A^\top y^j+\sigma A^\top(A\hat x^j-b)+P(\hat x^j-x^j),\\
Q(\hat y^j-y^j)&=b-A(2\hat x^j-x^j).
\end{aligned}
\]
Using the affine-face representation, substituting
$\hat x^j=x^\star+\hat u^j$ and
$x^j=x^\star+u^j+n^j$, and projecting the primal condition onto
$\mathcal L_{\mathcal F}$ gives
\[
\begin{aligned}
\Pi Pu^j+\Pi A^\top y^j
&=
\Pi P\hat u^j
+\Pi\Bigl(
\nabla f(\hat x^j)+H_{\mathcal F}\hat x^j+c_{\mathcal F}
+\sigma A^\top A\hat u^j
\Bigr)
-\Pi Pn^j,\\
Au^j+Qy^j
&=
2A\hat u^j+Q\hat y^j-An^j .
\end{aligned}
\]
By the definitions of $M_{\rm red}$, $\mathcal T_{\rm red}$, and $D_\perp$,
these two identities are exactly \eqref{eq:perturbed-reduced-exact}.

For the linearized map, the PFBS relation
\[
        (\mathcal M-\mathcal B_h)w^j
        \in
        (\mathcal M+\mathcal A_h)\hat w^j
\]
gives the same dual equation and the primal condition
\[
0\in
\partial g(\hat x^j)+\nabla f(\hat x^j)-\nabla h(\hat x^j)
+\nabla h(x^j)-A^\top y^j
+\sigma A^\top(A\hat x^j-b)+P(\hat x^j-x^j).
\]
Projecting this condition onto $\mathcal L_{\mathcal F}$, using the same dual
equation as above, and adding and subtracting $\Pi\nabla h(x^\star+u^j)$ gives
\[
\begin{aligned}
&\Pi Pu^j+\Pi A^\top y^j-\Pi\nabla h(x^\star+u^j)\\
&\quad =
\Pi P\hat u^j
+\Pi\bigl(\nabla f(\hat x^j)-\nabla h(\hat x^j)\bigr)\\
&\qquad
+\Pi\bigl(H_{\mathcal F}\hat x^j+c_{\mathcal F}
+\sigma A^\top A\hat u^j\bigr)
-\Pi Pn^j\\
&\qquad
+\Pi\bigl(\nabla h(x^\star+u^j+n^j)
-\nabla h(x^\star+u^j)\bigr),\\
&Au^j+Qy^j
=
2A\hat u^j+Q\hat y^j-An^j .
\end{aligned}
\]
This is precisely \eqref{eq:perturbed-reduced-linearized}. Since
$w^j\to w^\star$, both $x^\star+u^j+n^j$ and $x^\star+u^j$ lie in any fixed
neighborhood of $x^\star$ for all sufficiently large $j$. The local Lipschitz
continuity of $\nabla h$ therefore gives
\[
        \|E_h^j\|
        \le
        L\|n^j\|
\]
for some local Lipschitz constant $L$, proving $\|E_h^j\|=O(\|n^j\|)$.
\end{proof}

\begin{remark}[Normal component induced by reflection and anchoring]
\label{rem:normal-component-recurrence}
The perturbation in Proposition~\ref{prop:perturbed-reduced-state-relations}
is caused by the fact that the Halpern state $x^j$ need not lie on the
identified affine face, even after the shadow point $\hat x^j$ has
identified it.

To see this explicitly, consider one epoch of a restarted reflected Halpern
scheme, where \(n\) indexes the epoch and \(k\) the inner iteration. Let
\[
        \nu_{n,k}:=(I-\Pi_{\mathcal F})(x_{n,k}-x^\star)
\]
be the normal component of the Halpern state, and suppose that
$\hat x_{n,k}\in\mathcal F$. Then
\[
        (I-\Pi_{\mathcal F})(\hat x_{n,k}-x^\star)=0.
\]
The reflected point satisfies
\[
        \bar x_{n,k}
        =
        (1+\gamma)\hat x_{n,k}-\gamma x_{n,k},
\]
and therefore
\[
        \bar\nu_{n,k}
        :=
        (I-\Pi_{\mathcal F})(\bar x_{n,k}-x^\star)
        =
        -\gamma \nu_{n,k}.
\]
The Halpern update
\[
        x_{n,k+1}
        =
        \frac{1}{k+2}x_{n,0}
        +
        \frac{k+1}{k+2}\bar x_{n,k}
\]
gives
\[
        \nu_{n,k+1}
        =
        \frac{1}{k+2}\nu_{n,0}
        -
        \frac{\gamma(k+1)}{k+2}\nu_{n,k}.
\]
This linear recurrence also gives the decay rate of the normal drift
within such an identified epoch. Setting
\(s_{n,k}:=(k+1)\nu_{n,k}\), we obtain
\[
        s_{n,k+1}=\nu_{n,0}-\gamma s_{n,k}.
\]
Solving this recursion gives, for every fixed \(\gamma\in(-1,1]\),
\[
        s_{n,k}
        =
        \frac{1+\gamma(-\gamma)^k}{1+\gamma}\nu_{n,0},
        \qquad
        \nu_{n,k}
        =
        \frac{1+\gamma(-\gamma)^k}{(1+\gamma)(k+1)}\nu_{n,0}.
\]
Thus, throughout the reflected range \(\gamma\in(-1,1]\), the normal component
decays as \(O(1/k)\) along the epoch, although it need not vanish in finite
time. An identically zero normal sequence is obtained only under the
invariance condition that the epoch anchor lies on the affine face.
Consequently, the autonomous reduced dynamics without the perturbation term are
valid for the Halpern state only under the additional invariance condition
$n^j=0$. Without this condition, the correct late-stage description is the
perturbed reduced representation \eqref{eq:perturbed-reduced-exact} or
\eqref{eq:perturbed-reduced-linearized}.
\end{remark}

\begin{remark}[Scope of the affine-face model]
\label{rem:scope-affine-face-model}
Assumption~\ref{ass:affine-face-model} is stronger than partial smoothness or
piecewise linear-quadratic structure alone. It is appropriate when, on the
identified face, $g$ is locally a quadratic function plus the indicator of an
affine set, or more generally when the identified normal fiber coincides with
all of $\mathcal L_{\mathcal F}^{\perp}$. For generic $\ell_1$ regularizers,
nonnegativity constraints, box constraints, and simplex constraints, the fixed
normal fiber after identification is usually a cone or a strict subset of
$\mathcal L_{\mathcal F}^{\perp}$ rather than the whole normal space. In those
cases the tangential projected dynamics remain informative, but the exact
residual identity \eqref{eq:exact-reduced-residual-identity} should be replaced
by a polyhedral error-bound or metric-subregularity argument.
\end{remark}

The results in this section clarify the role of identification in reflected
Halpern and restarted reflected Halpern schemes. The finite identification
property belongs to the shadow sequence $\hat x^j=F(w^j)$, not to the
Halpern state $x^j$ itself. After identification, the KKT residual at the
shadow points is exactly represented by the reduced residual
$\|\mathcal T_{\rm red}(\hat z^j)\|$ on the tangent space of the active
affine face. This provides the local geometric mechanism needed for the
restart analysis in Section~5.

At the same time, the Halpern state may retain a nonzero normal component.
The resulting late-stage dynamics are therefore reduced only up to a normal
perturbation, unless the algorithm is initialized or restarted directly on the
identified face. This distinction is essential for a correct local theory of
reflected Halpern primal--dual methods.

\section{Restart Strategies under Fixed-Point Sharpness}
\label{sec5}

Restart converts sublinear residual decay into a linear epoch-wise contraction
when the fixed-point residual satisfies a sharpness condition on the set visited
by the restarted trajectory. We develop this conditional analysis for a generic
map
\[
        F\in\{\widehat{\mathcal T},\widehat{\mathcal T}_{\rm lin}\},
\]
covering both exact and linearized augmented primal--dual schemes. We prove
linear convergence of restart anchors under fixed-point sharpness on the visited
restart set, and then verify the required sharpness from KKT error bounds,
including the reduced error-bound mechanism on the identified face. Thus the
restart theorems are global only when sharpness holds on the full relevant
trajectory set. In the applications below, sharpness is obtained from local
error bounds, and the conclusions are local or tail linear convergence of
restart anchors after the trajectory has entered the corresponding neighborhood.

\subsection{Restarted Reflected Halpern Under Sharpness}
\label{sec:sharp_restart}

Fix one of these two maps, together with its associated
preconditioner $\mathcal M$, and let
\[
        \mathcal H:=\mathbb R^n\times\mathbb R^m,
        \qquad
        F\in\{\widehat{\mathcal T},\widehat{\mathcal T}_{\rm lin}\}.
\]
For both choices of $F$, the fixed-point set coincides with the KKT set
$W^\star=\mathcal T^{-1}(0)$ introduced in Section~\ref{sec4}. Throughout this
section we work in the standing KKT-nonempty case \(W^\star\neq\emptyset\).
The relaxation parameter
$\gamma$ is chosen from the admissible range for the selected map; the
reflected map
\[
        \mathcal S_\gamma:=(1+\gamma)F-\gamma I
\]
is $\mathcal M$-nonexpansive. Concretely, this means
$\gamma\in(-1,1]$ for the exact dPPM map and
$\gamma\in(-1,1-L_h/(2\mu_x))$ for the linearized PFBS algorithmic
statements. The endpoint $\gamma=1-L_h/(2\mu_x)$ remains available for the
linearized map-level nonexpansiveness estimate of
Theorem~\ref{thm:sec32_relaxation}.

For $w\in\mathcal H$, define the fixed-point residual
\[
        r_F(w):=\|w-F(w)\|_{\mathcal M}.
\]
Since
\[
        w-\mathcal S_\gamma(w)=(1+\gamma)(w-F(w)),
\]
one has
\begin{equation}
\label{eq:reflected_residual_relation}
        \|w-\mathcal S_\gamma(w)\|_{\mathcal M}
        =(1+\gamma)r_F(w).
\end{equation}
Thus the residual of the reflected map and the residual of $F$ are equivalent up
to the scalar $1+\gamma$.

\begin{definition}[Sharpness]
\label{def:sharpness}
Let $\mathcal K\subset\mathcal H$. We say that $F$ satisfies the
$\alpha$-sharpness condition on $\mathcal K$ if $\alpha>0$ and
\begin{equation}
\label{eq:sharpness_F}
        \alpha\operatorname{dist}_{\mathcal M}(w,W^\star)
        \le
        r_F(w),
        \qquad
        \forall w\in\mathcal K,
\end{equation}
where
\[
        \operatorname{dist}_{\mathcal M}(w,W^\star)
        :=
        \inf_{w^\star\in W^\star}\|w-w^\star\|_{\mathcal M}.
\]
\end{definition}
When $\mathcal M$ is singular, $\operatorname{dist}_{\mathcal M}$ is understood
as the semidistance induced by the $\mathcal M$-seminorm; positive definiteness
is imposed explicitly whenever a genuine distance is required.

\begin{remark}[Local versus global interpretation of restart guarantees]
\label{rem:restart-local-global}
The fixed-frequency and adaptive restart theorems below are conditional on the
sharpness inequality holding on the set of restart anchors, or more generally on
the set visited by the restarted trajectory. If this sharpness holds globally on
the relevant space, the resulting epoch contraction is global. If sharpness is
obtained only from local metric subregularity, partial smoothness, an identified
active structure, or a Hoffman--Robinson error bound in a neighborhood, then the
linear convergence conclusion is local or tail-only. In this local case, the
theory does not assert that an arbitrary initialization starts inside the
sharpness neighborhood; it applies after the iterates enter that neighborhood
and remain in the relevant set. Adaptive restart removes the need to know the
sharpness constant in advance, but it does not remove the assumption that the
relevant sharpness property holds on the visited set.
\end{remark}

The restarted reflected Halpern scheme is organized into epochs. Given the
anchor $w_{n,0}$ of epoch $n$, the inner iterates are
\begin{equation}
\label{eq:epoch_halpern}
        w_{n,k+1}
        =
        \frac{1}{k+2}w_{n,0}
        +
        \frac{k+1}{k+2}\mathcal S_\gamma(w_{n,k}),
        \qquad k\ge0.
\end{equation}
When an epoch stops at length $K_n$, the basic wrapper uses the last Halpern
state as the next anchor:
\begin{equation}
\label{eq:restart_epoch_update}
        w_{n+1,0}:=w_{n,K_n}.
\end{equation}
Algorithm~\ref{alg:restart_wrapper} records only the outer restart wrapper. It
is used with either the exact dPPM map of Algorithm~\ref{alg:halpern_dppm} or
the linearized PFBS map of Algorithm~\ref{alg:halpern_pfbs}; the inner reflected
Halpern step is the epoch iteration \eqref{eq:epoch_halpern}.

\begin{algorithm}[t]
\caption{Restart Wrapper for Reflected Halpern Schemes}
\label{alg:restart_wrapper}
{\small
\begin{algorithmic}[1]
\STATE \textbf{Input:} map $F$, relaxation $\gamma$, initial anchor $w_{0,0}$, restart rule.
\FOR{$n=0,1,2,\ldots$}
    \STATE Run \eqref{eq:epoch_halpern} from anchor $w_{n,0}$ until the restart rule is met.
    \STATE Denote the stopping index by $K_n$.
    \STATE Set $w_{n+1,0}\leftarrow w_{n,K_n}$.
\ENDFOR
\end{algorithmic}
}
\end{algorithm}

We use two restart rules. The first is a fixed-frequency rule. If a sharpness
constant $\alpha$ is known and a target contraction factor $\beta\in(0,1)$ is
chosen, set
\begin{equation}
\label{eq:fixed_restart_length}
        K_n\equiv k^\star,
        \qquad
        k^\star:=\left\lceil \frac{2}{\alpha\beta(1+\gamma)}\right\rceil .
\end{equation}
The second is an adaptive rule based only on the observed residual. Given
$\beta\in(0,1)$ and an initial epoch length $K_0\in\mathbb N$, restart when
\begin{equation}
\label{eq:adaptive_restart_F}
\begin{cases}
        r_F(w_{n,k})\le \beta r_F(w_{n,0}), & n\ge1,\\[0.3em]
        k\ge K_0, & n=0.
\end{cases}
\end{equation}

The single estimate needed for both restart rules is the following epochwise
residual bound.

\begin{lemma}[Epoch residual bound]
\label{lem:epoch_residual_bound}
For every epoch $n\ge0$ and every $k\ge0$,
\[
        (1+\gamma)r_F(w_{n,k})
        =
        \|w_{n,k}-\mathcal S_\gamma(w_{n,k})\|_{\mathcal M}
        \le
        \frac{2}{k+1}
        \operatorname{dist}_{\mathcal M}(w_{n,0},W^\star).
\]
Equivalently,
\begin{equation}
\label{eq:epoch_residual_bound_rf}
        r_F(w_{n,k})
        \le
        \frac{2}{(1+\gamma)(k+1)}
        \operatorname{dist}_{\mathcal M}(w_{n,0},W^\star).
\end{equation}
\end{lemma}

\begin{proof}
For the admissible choices of $\gamma$, the map $\mathcal S_\gamma$ is
$\mathcal M$-nonexpansive and
\[
        \operatorname{Fix}(\mathcal S_\gamma)
        =
        \operatorname{Fix}(F)
        =
        W^\star.
\]
The residual estimates in Propositions~\ref{prop:general_convergence}
and~\ref{prop:halpern_pfbs_basic} apply to the two possible choices of $F$.
Applied to the inner iteration \eqref{eq:epoch_halpern}, they give, for every
$w^\star\in W^\star$,
\[
        \|w_{n,k}-\mathcal S_\gamma(w_{n,k})\|_{\mathcal M}
        \le
        \frac{2}{k+1}\|w_{n,0}-w^\star\|_{\mathcal M}.
\]
Taking the infimum over $w^\star\in W^\star$ and using
\eqref{eq:reflected_residual_relation} proves the claim.
\end{proof}

\begin{remark}[Residual-dominated restart candidates]
\label{rem:restart_candidate_selection}
The estimates below admit a direct extension to residual-dominated restart
candidate selection. Suppose that, at a restart check, an auxiliary candidate
\(\widetilde w_{n,k}\) is available and the next anchor is selected as a point
\[
        c_{n,k}\in\{w_{n,k},\widetilde w_{n,k}\},
        \qquad
        r_F(c_{n,k})\le \eta\, r_F(w_{n,k})
\]
for a fixed \(\eta\ge1\), with \(w_{n+1,0}:=c_{n,K_n}\). If the selected
candidates remain in the set on which sharpness is imposed, the proofs below
apply with the factor \(\eta\) multiplying the residual bound. In particular,
for the greedy choice \(\eta=1\), which selects the candidate with the smaller
fixed-point residual, the stated restart constants are unchanged. This is the
form used in the LP implementation: the candidate set consists of the current
primal--dual point and the Halpern-updated point.
\end{remark}

\subsection{Restart-Anchor Contraction under Fixed-Point Sharpness}
\label{sec:restart_linear}

Fixed-point sharpness on the visited restart set converts the $O(1/k)$ inner
residual decay into a linear contraction of restart anchors across epochs.

\begin{theorem}[Fixed-frequency restart]
\label{thm:fixed_restart_linear}
Let the restart points $\{w_{n,0}\}_{n\ge0}$ be generated by
\eqref{eq:epoch_halpern} and \eqref{eq:restart_epoch_update} with the fixed epoch
length \eqref{eq:fixed_restart_length} for a prescribed $\beta\in(0,1)$.
Assume that all iterates remain in a set $\mathcal K\subset\mathcal H$ on which
$F$ satisfies the $\alpha$-sharpness condition \eqref{eq:sharpness_F}. Then
\begin{equation}
\label{eq:fixed_restart_linear}
        \operatorname{dist}_{\mathcal M}(w_{n,0},W^\star)
        \le
        \beta^n
        \operatorname{dist}_{\mathcal M}(w_{0,0},W^\star),
        \qquad
        \forall n\ge0.
\end{equation}
\end{theorem}

\begin{proof}
By \eqref{eq:restart_epoch_update}, $w_{n+1,0}=w_{n,k^\star}$. Since the
iterates lie in $\mathcal K$, sharpness at $w_{n,k^\star}$ gives
\[
        \operatorname{dist}_{\mathcal M}(w_{n+1,0},W^\star)
        \le
        \frac{1}{\alpha}r_F(w_{n,k^\star}).
\]
Using \eqref{eq:epoch_residual_bound_rf},
\[
        \operatorname{dist}_{\mathcal M}(w_{n+1,0},W^\star)
        \le
        \frac{2}{\alpha(1+\gamma)(k^\star+1)}
        \operatorname{dist}_{\mathcal M}(w_{n,0},W^\star).
\]
The choice of $k^\star$ in \eqref{eq:fixed_restart_length} implies
\[
        \frac{2}{\alpha(1+\gamma)(k^\star+1)}
        \le
        \frac{2}{\alpha(1+\gamma)k^\star}
        \le
        \beta.
\]
Therefore the restart anchors contract by a factor not exceeding $\beta$ at every
epoch, and iteration gives \eqref{eq:fixed_restart_linear}.
\end{proof}

The adaptive rule does not require prior knowledge of $\alpha$. For this result
we assume $\mathcal M\succ0$; then $\operatorname{dist}_{\mathcal M}$ is a
genuine distance and the residual-to-distance ratio below is well defined away
from the solution set.

\begin{theorem}[Adaptive restart]
\label{thm:adaptive_restart_linear}
Let the restart points $\{w_{n,0}\}_{n\ge0}$ be generated by
\eqref{eq:epoch_halpern} and \eqref{eq:restart_epoch_update} with the adaptive
rule \eqref{eq:adaptive_restart_F}. Assume that $\mathcal M\succ0$, that all
iterates remain in a set $\mathcal K\subset\mathcal H$, and that $F$ satisfies
the $\alpha$-sharpness condition \eqref{eq:sharpness_F} on $\mathcal K$. Then the
following statements hold.
\begin{enumerate}
\item[(i)] Every epoch $n\ge1$ terminates after finitely many inner iterations.
If $w_{n,0}\notin W^\star$, define
\[
        \alpha_n
        :=
        \frac{r_F(w_{n,0})}
        {\operatorname{dist}_{\mathcal M}(w_{n,0},W^\star)},
\]
then $\alpha_n\ge\alpha$ and the stopping index satisfies
\begin{equation}
\label{eq:adaptive_restart_bound}
        K_n
        \le
        \left\lceil\frac{2}{\alpha_n\beta(1+\gamma)}\right\rceil
        \le
        \left\lceil\frac{2}{\alpha\beta(1+\gamma)}\right\rceil .
\end{equation}
If $w_{n,0}\in W^\star$, then $r_F(w_{n,0})=0$ and the restart condition holds
at $k=0$.

\item[(ii)] The fixed-point residuals at the restart points satisfy the linear
bound:
\begin{equation}
\label{eq:adaptive_residual_linear}
        r_F(w_{n,0})
        \le
        \beta^{n-1}r_F(w_{1,0}),
        \qquad
        \forall n\ge1.
\end{equation}

\item[(iii)] The restart points satisfy the linear $\mathcal M$-distance bound:
\begin{equation}
\label{eq:adaptive_distance_linear}
        \operatorname{dist}_{\mathcal M}(w_{n,0},W^\star)
        \le
        \frac{2}{\alpha(1+\gamma)(K_0+1)}
        \beta^{n-1}
        \operatorname{dist}_{\mathcal M}(w_{0,0},W^\star),
        \qquad
        \forall n\ge1.
\end{equation}
\end{enumerate}
\end{theorem}

\begin{proof}
We prove the three claims in order. For (i), fix $n\ge1$. If
$w_{n,0}\in W^\star$, then $r_F(w_{n,0})=0$, and the adaptive rule is satisfied
at $k=0$. Otherwise, sharpness gives
$\alpha_n\ge\alpha$. By \eqref{eq:epoch_residual_bound_rf}, for every $k\ge0$,
\[
        r_F(w_{n,k})
        \le
        \frac{2}{(1+\gamma)(k+1)}
        \operatorname{dist}_{\mathcal M}(w_{n,0},W^\star)
        =
        \frac{2}{\alpha_n(1+\gamma)(k+1)}r_F(w_{n,0}).
\]
Thus the condition $r_F(w_{n,k})\le\beta r_F(w_{n,0})$ holds once
$k+1\ge2/(\alpha_n\beta(1+\gamma))$, proving finite termination and
\eqref{eq:adaptive_restart_bound}.

For (ii), the adaptive rule and the restart update imply, for every $n\ge1$,
\[
        r_F(w_{n+1,0})
        =
        r_F(w_{n,K_n})
        \le
        \beta r_F(w_{n,0}).
\]
Iterating this inequality proves \eqref{eq:adaptive_residual_linear}.

For (iii), sharpness at $w_{n,0}$ and \eqref{eq:adaptive_residual_linear} give
\[
        \operatorname{dist}_{\mathcal M}(w_{n,0},W^\star)
        \le
        \frac{1}{\alpha}\beta^{n-1}r_F(w_{1,0}),
        \qquad n\ge1.
\]
Since the initial epoch has length $K_0$, $w_{1,0}=w_{0,K_0}$, and
\eqref{eq:epoch_residual_bound_rf} yields
\[
        r_F(w_{1,0})
        \le
        \frac{2}{(1+\gamma)(K_0+1)}
        \operatorname{dist}_{\mathcal M}(w_{0,0},W^\star).
\]
Combining the last two inequalities proves \eqref{eq:adaptive_distance_linear}.
\end{proof}

\subsection{Verifying Sharpness from KKT Error Bounds}
\label{sec:sharpness_verification}

The restart theorems above require the fixed-point sharpness condition
\eqref{eq:sharpness_F}. We connect this condition with standard KKT error
bounds. This is the only verification direction needed for the restart theory.
In particular, we do not require a converse equivalence between fixed-point
sharpness and metric subregularity.

Let $U\subset\mathcal H$. We say that the augmented KKT mapping satisfies an
error bound on $U$ if there exists $\alpha_U>0$ such that
\begin{equation}
\label{eq:kkt_error_bound_U}
        \alpha_U\operatorname{dist}(u,W^\star)
        \le
        \operatorname{dist}(0,\mathcal T(u)),
        \qquad
        \forall u\in U.
\end{equation}
This is the usual metric-subregularity estimate, restricted to the set $U$.

The first lemma converts Euclidean sharpness estimates into the
$\mathcal M$-metric sharpness used in the restart analysis.

\begin{lemma}[Metric conversion of sharpness]
\label{lem:metric_conversion_sharpness}
Assume $\mathcal M\succ0$, and let
$\lambda_{\min}(\mathcal M)$ and $\lambda_{\max}(\mathcal M)$ denote its extremal
eigenvalues. If a mapping $F$ satisfies
\[
        \alpha_E\operatorname{dist}(w,W^\star)
        \le
        \|w-F(w)\|,
        \qquad
        \forall w\in\mathcal K,
\]
then $F$ satisfies \eqref{eq:sharpness_F} on $\mathcal K$ with
\[
        \alpha_M
        :=
        \alpha_E
        \sqrt{\frac{\lambda_{\min}(\mathcal M)}
        {\lambda_{\max}(\mathcal M)}}.
\]
\end{lemma}

\begin{proof}
For every $w$,
\[
        \operatorname{dist}_{\mathcal M}(w,W^\star)
        \le
        \sqrt{\lambda_{\max}(\mathcal M)}\operatorname{dist}(w,W^\star),
\]
whereas
\[
        r_F(w)=\|w-F(w)\|_{\mathcal M}
        \ge
        \sqrt{\lambda_{\min}(\mathcal M)}\|w-F(w)\|.
\]
Combining these inequalities with the Euclidean sharpness estimate proves the
claim.
\end{proof}

\begin{proposition}[KKT error bounds imply fixed-point sharpness]
\label{prop:kkt_error_bounds_imply_sharpness}
Assume that the KKT error bound \eqref{eq:kkt_error_bound_U} holds on a set
$U\subset\mathcal H$. Let $\mathcal K\subset\mathcal H$ satisfy
\[
        \widehat{\mathcal T}(\mathcal K)
        \cup
        \widehat{\mathcal T}_{\rm lin}(\mathcal K)
        \subset U.
\]
Then the following statements hold.
\begin{enumerate}
\item[(i)] The exact dPPM map satisfies
\[
        \alpha_{{\rm ex},U}^E\operatorname{dist}(w,W^\star)
        \le
        \|w-\widehat{\mathcal T}(w)\|,
        \qquad
        \forall w\in\mathcal K,
\]
with
\[
        \alpha_{{\rm ex},U}^E
        :=
        \frac{\alpha_U}{\alpha_U+\|\mathcal M\|}.
\]

\item[(ii)] The linearized PFBS map satisfies
\[
        \alpha_{{\rm lin},U}^E\operatorname{dist}(w,W^\star)
        \le
        \|w-\widehat{\mathcal T}_{\rm lin}(w)\|,
        \qquad
        \forall w\in\mathcal K,
\]
with
\[
        \alpha_{{\rm lin},U}^E
        :=
        \frac{\alpha_U}{\alpha_U+\|\mathcal M\|+L_h}.
\]

\item[(iii)] If $\mathcal M\succ0$, then both maps satisfy the sharpness
condition \eqref{eq:sharpness_F} on $\mathcal K$, with the constants from
parts~(i)--(ii) multiplied by
$\sqrt{\lambda_{\min}(\mathcal M)/\lambda_{\max}(\mathcal M)}$.
\end{enumerate}
\end{proposition}

\begin{proof}
For the exact map, fix $w\in\mathcal K$ and set
$u:=\widehat{\mathcal T}(w)$. Then $u\in U$, and the resolvent relation gives
\[
        \mathcal M(w-u)\in\mathcal T(u).
\]
Therefore
\[
        \alpha_U\operatorname{dist}(u,W^\star)
        \le
        \operatorname{dist}(0,\mathcal T(u))
        \le
        \|\mathcal M(w-u)\|
        \le
        \|\mathcal M\|\|w-u\|.
\]
The triangle inequality yields
\[
        \operatorname{dist}(w,W^\star)
        \le
        \|w-u\|+\operatorname{dist}(u,W^\star)
        \le
        \left(1+\frac{\|\mathcal M\|}{\alpha_U}\right)\|w-u\|,
\]
which proves part~(i).

For the linearized map, set $u:=\widehat{\mathcal T}_{\rm lin}(w)$. The PFBS
relation gives
\[
        (\mathcal M-\mathcal B_h)(w)
        \in
        (\mathcal M+\mathcal A_h)(u),
\]
and therefore
\[
        \mathcal M(w-u)+\mathcal B_h(u)-\mathcal B_h(w)
        \in
        (\mathcal A_h+\mathcal B_h)(u)
        =
        \mathcal T(u).
\]
Since $u\in U$, the error bound and the $L_h$-Lipschitz continuity of
$\mathcal B_h$ give
\[
\begin{aligned}
        \alpha_U\operatorname{dist}(u,W^\star)
        &\le
        \operatorname{dist}(0,\mathcal T(u)) \\
        &\le
        \|\mathcal M(w-u)+\mathcal B_h(u)-\mathcal B_h(w)\| \\
        &\le
        (\|\mathcal M\|+L_h)\|w-u\|.
\end{aligned}
\]
The same triangle-inequality argument proves part~(ii). Part~(iii) follows from
Lemma~\ref{lem:metric_conversion_sharpness}.
\end{proof}

\begin{remark}[Local-to-tail interpretation]
\label{rem:local_to_tail_sharpness}
Proposition~\ref{prop:kkt_error_bounds_imply_sharpness} is local in the same
sense as metric subregularity. The error-bound set \(U\) need only contain
the forward images of the part of the trajectory on which the restart theorem
is applied. Thus, if a KKT error bound holds in a neighborhood \(U\) of a
solution \(w^\star\), and if a relevant tail set
\(\mathcal K_{\rm tail}\) of restart points satisfies
\[
        \widehat{\mathcal T}(\mathcal K_{\rm tail})
        \cup
        \widehat{\mathcal T}_{\rm lin}(\mathcal K_{\rm tail})
        \subset U,
\]
then the fixed-point sharpness condition holds on this tail set in the
positive definite metric setting. Since both maps are continuous under the
standing assumptions of Sections~\ref{sec2}--\ref{sec3}, this inclusion is
automatic for all sufficiently late points in such a tail set whenever the
restarted trajectory converges to \(w^\star\). The restart results therefore
give local or tail linear convergence of restart anchors after the trajectory
has entered the error-bound neighborhood.
\end{remark}

\begin{proposition}[Box-constrained linear and quadratic programs]
\label{prop:box_lp_qp_error_bound}
Let
\[
        C:=[l,u]\subset\mathbb R^n,\qquad g=\delta_C,
\]
where \(C\) is a nonempty box. Suppose that either
\[
        f(x)=c^\top x
        \quad\text{or}\quad
        f(x)=\frac12 x^\top Hx+c^\top x,\qquad H\succeq0.
\]
Then, for every \(w^\star\in W^\star\), there exist a neighborhood
\(U\) of \(w^\star\) and a constant \(\alpha_U>0\) such that the augmented KKT
mapping satisfies
\[
        \alpha_U\operatorname{dist}(w,W^\star)
        \le
        \operatorname{dist}(0,\mathcal T(w)),
        \qquad
        \forall w\in U.
\]
Consequently, if \(\mathcal M\succ0\) and
\[
        \widehat{\mathcal T}(\mathcal K)
        \cup
        \widehat{\mathcal T}_{\rm lin}(\mathcal K)
        \subset U,
\]
then both the exact and linearized fixed-point maps satisfy the sharpness
condition \eqref{eq:sharpness_F} on \(\mathcal K\).
\end{proposition}

\begin{proof}
For the stated choices of \(f\), write \(\nabla f(x)=Hx+c\), with \(H=0\) in
the linear case. Since \(g=\delta_C\), we have \(\partial g=N_C\), and
\[
        \mathcal T(x,y)
        =
        \begin{pmatrix}
        (H+\sigma A^\top A)x-A^\top y+c-\sigma A^\top b+N_C(x)\\
        Ax-b
        \end{pmatrix}.
\]
The graph of \(N_C\) is a finite union of polyhedral sets: each piece is
obtained by fixing which lower bounds, upper bounds, and inactive box
constraints are active. Adding the displayed linear terms and the equality
component \(Ax-b\) preserves piecewise polyhedrality. Thus \(\mathcal T\) is a
piecewise polyhedral multifunction.

The Hoffman--Robinson error-bound theorem for piecewise polyhedral
multifunctions~\cite{hoffman1952approximate,robinson1981continuity} implies
metric subregularity of \(\mathcal T\) at \((w^\star,0)\). Thus there are a
neighborhood \(U\) of \(w^\star\) and a constant \(\kappa_U>0\) such that
\[
        \operatorname{dist}(w,\mathcal T^{-1}(0))
        \le
        \kappa_U\operatorname{dist}(0,\mathcal T(w)),
        \qquad
        \forall w\in U.
\]
Since \(W^\star=\mathcal T^{-1}(0)\), the asserted error bound follows with
\(\alpha_U=1/\kappa_U\). The final statement is exactly
Proposition~\ref{prop:kkt_error_bounds_imply_sharpness}.
\end{proof}

\begin{corollary}[Sharpness after affine-face identification]
\label{cor:identified_face_sharpness_for_restart}
Suppose the hypotheses of Corollary~\ref{cor:reduced-error-bound-transfer} hold,
and the augmented KKT mapping satisfies a face-restricted error bound on
$V:=\Psi(V_{\rm red})$. Let $\mathcal K_{\rm late}\subset\mathcal H$ be a set
such that
\[
        \widehat{\mathcal T}(\mathcal K_{\rm late})
        \cup
        \widehat{\mathcal T}_{\rm lin}(\mathcal K_{\rm late})
        \subset V.
\]
If $\mathcal M\succ0$, then both fixed-point maps satisfy the sharpness
condition \eqref{eq:sharpness_F} on $\mathcal K_{\rm late}$.
\end{corollary}

\begin{proof}
Corollary~\ref{cor:reduced-error-bound-transfer} gives a KKT error bound of the
form \eqref{eq:kkt_error_bound_U} on $U=V$. Applying
Proposition~\ref{prop:kkt_error_bounds_imply_sharpness} and the metric
conversion in Lemma~\ref{lem:metric_conversion_sharpness} proves the result.
\end{proof}

Consequently, the restart theorems in Section~\ref{sec:restart_linear} can be
used in two ways. If sharpness is known directly for the fixed-point residual,
Theorems~\ref{thm:fixed_restart_linear} and~\ref{thm:adaptive_restart_linear}
apply under their stated metric assumptions. If instead a KKT error bound or
metric-subregularity estimate is available,
Proposition~\ref{prop:kkt_error_bounds_imply_sharpness}
transfers it to fixed-point sharpness. Proposition~\ref{prop:box_lp_qp_error_bound}
provides this verification for the box-constrained linear and convex quadratic
programs used in the numerical section. In the identified affine-face regime,
Corollary~\ref{cor:identified_face_sharpness_for_restart} supplies the same
verification from the reduced residual theory of Section~\ref{sec4}.

\section{Numerical Experiments}
\label{sec6}

We test the augmented primal--dual family on linear and convex quadratic
programs. The experiments are organized by problem class and by the
computational role of each update. We do not benchmark every family member on
every problem class: LP experiments focus on equality-form solvers of the PDHG
and CP-AL types, while QP experiments test both subproblem-based and linearized
variants.

\subsection{Experimental Protocol and Method Scope}
\label{subsec:exp_protocol}

All reported GPU experiments were run on one H100 GPU with 80\,GB HBM3 memory
and CUDA~12.8.1.
External baselines were run under the same hardware conditions whenever they
appear in a direct comparison, with presolve disabled where stated. Within
each benchmark family, all methods use the same stopping tolerance and
wall-clock or iteration budget. All comparative statements below are therefore
relative to the protocol specified for the corresponding benchmark, including
the reported presolve convention. We report solved counts, time-limit counts,
total runtime when informative, and the shifted geometric mean
\[
\mathrm{SGM10}(t_1,\ldots,t_N)
:=
\exp\!\left(\frac{1}{N}\sum_{i=1}^N \log(t_i+10)\right)-10 .
\]
When a run reaches the time limit, the time limit is used in the computation
of \(\mathrm{SGM10}\).

For the restarted variants proposed here, restart decisions are made using the
fixed-point residual of the underlying map, as in Section~\ref{sec5}. For the
LP implementation we also use the residual-dominated restart-candidate selection
described in Remark~\ref{rem:restart_candidate_selection}. At a restart check,
the candidate is chosen greedily from the current primal--dual point and the
Halpern-updated point; in the notation of Section~\ref{sec5}, this corresponds
to the case \(\eta=1\). Let \(t\) denote the current inner-iteration index in
epoch \(n\), let \(c_{n,t}\) be the selected candidate at that check, and set
\[
        R_{n,t}:=r_F(c_{n,t}),\qquad
        r_F(w):=\|w-F(w)\|_{\mathcal M},
\]
where \(F\) is the corresponding base map. By
\eqref{eq:reflected_residual_relation}, this is equivalent, up to the scalar
\(1+\gamma\), to using the residual of the reflected map. Following the
adaptive restart criteria used in HPR-LP, \texttt{cuPDLPx}, and
\texttt{PDHCG-II}~\cite{hpr,cupdlpx,pdhcgii}, a restart is performed if one of
the following conditions is met:
\begin{enumerate}
\item[(i)] \emph{Sufficient decay:}
      \(R_{n,t}\le \beta_{\rm suff}R_{n,0}\).
\item[(ii)] \emph{Necessary decay plus no local progress:}
      \(R_{n,t}\le \beta_{\rm nec}R_{n,0}\) and
      \(R_{n,t}>R_{n,t-1}\).
\item[(iii)] \emph{Long inner loop:}
      \(t\ge \beta_{\rm art}k_{\rm tot}\), where \(k_{\rm tot}\) is the
      cumulative iteration count at the check.
\end{enumerate}
We use
\((\beta_{\rm suff},\beta_{\rm nec},\beta_{\rm art})=(0.2,0.8,0.36)\)
for the LP experiments and the corresponding PDHCG-II restart parameters for
the QP experiments.

\begin{remark}[Normalized CP-AL parametrization]
\label{rem:cpal_sigma_update}
For the LP experiments, we use the normalized one-parameter CP-AL subfamily
\[
        P=\sigma(\lambda I_n-A^\top A),\qquad Q=\sigma^{-1}I_m,
\]
or equivalently \(\tau^{-1}=\sigma\lambda\) and \(\rho=\sigma\). This is not
a new algorithmic family, but a convenient parametrization in which
\(\lambda\) controls the primal metric while \(\sigma\) acts simultaneously as
augmentation parameter and dual stepsize. The CP-AL condition
\(\tau(\sigma+\rho)\|A\|^2\le1\) becomes
\(\lambda\ge 2\|A\|^2\) under this parametrization. We choose \(\lambda\)
using a power-iteration estimate of \(\|A\|^2\), with a small safety factor
above this lower bound, and initialize \(\sigma_0=0.5\,\lambda^{-1/2}\).
The parameter \(\sigma\) is updated only at restart epochs. If
\(\Delta x\) and \(\Delta y\) are the primal and dual displacements over the
completed epoch, define
\[
        d_x^2:=\lambda\|\Delta x\|^2-\|A\Delta x\|^2,
        \qquad
        \widehat\sigma:=\frac{\|\Delta y\|}{\sqrt{d_x^2}}.
\]
When this estimate is numerically reliable, we use the damped log-scale update
\[
        \log\sigma^+=(1-\theta)\log\sigma+\theta\log\widehat\sigma,
        \qquad \theta=0.5,
\]
with a fixed safeguard interval proportional to \(\lambda^{-1/2}\). If the
estimate is unreliable, \(\sigma\) is left unchanged; the safeguard interval is
used only to avoid degenerate scalings of the metric. Thus
\texttt{RHR-CP-AL} denotes the restarted Halpern-reflected realization of the
\texttt{CP-AL} base map equipped with this safeguarded adaptive \(\sigma\)
rule.
\end{remark}

For all experiments we report a relative KKT score associated with the model
\eqref{P}. Recall that its KKT conditions are
\[
        0\in \nabla f(x)+\partial g(x)-A^\top y,
        \qquad
        Ax=b .
\]
All multipliers used in the residual evaluation are expressed in this sign
convention.
For a reported primal-dual pair \((x,y)\) with finite objective value, define
\[
\begin{aligned}
r_{\mathrm p}(x)
&:=\max\left\{
\frac{\|Ax-b\|_\infty}{1+\|b\|_\infty},
\frac{\operatorname{dist}_\infty(x,\operatorname{dom}g)}
{1+\|x\|_\infty}
\right\},\\
r_{\mathrm d}(x,y)
&:=
\frac{\operatorname{dist}_\infty\!\bigl(A^\top y-\nabla f(x),\partial g(x)\bigr)}
{1+\|\nabla f(x)\|_\infty+\|A^\top y\|_\infty}.
\end{aligned}
\]
Here \(\operatorname{dist}_\infty\) denotes distance in the \(\ell_\infty\)
norm, with the convention that the distance to the empty set is \(+\infty\).
The relative objective gap is
\[
        r_{\mathrm g}(x,y)
        :=
        \frac{|\Phi(x)-d(y)|}{1+|\Phi(x)|+|d(y)|},
        \qquad
        d(y):=b^\top y+\inf_{z\in\mathbb R^n}
        \{f(z)+g(z)-\langle A^\top y,z\rangle\}.
\]
We then set
\[
        r_{\mathrm{KKT}}:=\max\{r_{\mathrm p},r_{\mathrm d},r_{\mathrm g}\}.
\]
For the equality-box LP and QP test problems below, \(g\) is the indicator of
the box \([l,u]\); these definitions reduce to the usual primal feasibility,
dual feasibility, and relative primal-dual gap residuals.
A run is declared solved when
\(r_{\mathrm p}\le \varepsilon_{\mathrm{feas}}\),
\(r_{\mathrm d}\le \varepsilon_{\mathrm{feas}}\), and
\(r_{\mathrm g}\le \varepsilon_{\mathrm{opt}}\). When a table states a single
tolerance \(\varepsilon\), we use
\(\varepsilon_{\mathrm{feas}}=\varepsilon_{\mathrm{opt}}=\varepsilon\).
When fixed-point residual traces are compared across methods, each residual is
normalized by its initial value because the underlying fixed-point metrics
differ.

We distinguish base maps from solver names. The names \texttt{PDHG},
\texttt{CP-AL}, \texttt{FA-CP}, \texttt{Lin-PDHG}, and \texttt{Lin-CP-AL}
refer to operator maps or algorithmic families. In numerical tables, the
prefix \texttt{RHR-} denotes the method obtained by applying the restarted
reflected Halpern wrapper to the corresponding base map. External or
previously published solvers keep their original names; in particular, the LP
PDHG-type comparators are \texttt{cuPDLPx}~\cite{cupdlpx} and
\texttt{cuPDLP-C}, the LP Halpern baseline is
\texttt{HPR-LP}~\cite{hpr}, and the QP external solvers include
\texttt{PDQP}~\cite{pdqp}, \texttt{HPR-QP}~\cite{hprqp}, and the
\texttt{PDHCG} variants reported in~\cite{pdhcgii}. This naming convention is
used consistently below.

\subsection{Linear Programming Experiments}
\label{subsec:lp_experiments}

We begin with equality-constrained box LPs of the form
\begin{equation}\label{eq:eq_lp_test}
\min_{x \in \mathbb{R}^{n}} \ c^{\top}x
\quad \mathrm{s.t.}\quad
Ax=b,\qquad l \le x \le u .
\end{equation}
This is the most direct LP realization of the framework: the equality
constraint is represented natively, and the augmentation acts directly on the
residual \(Ax-b\). The LP subsection has four roles. First, controlled
randomized families isolate scaling and near dependence in \(A\), so that the
effect of reflection and restarted Halpern anchoring can be tested without
instance selection from a benchmark library. Second, the main benchmark uses
the full Gurobi-presolved MIPLIB LP-relaxation collection used in
\texttt{cuPDLPx}~\cite{cupdlpx,miplib2017}. Third, the public Mittelmann LP
benchmark used in HPR-LP~\cite{hpr} provides an additional test set with the
same no-extra-presolve protocol. Fourth, a small set of representative MIPLIB
instances provides instance-level context for the aggregate table.

The synthetic instances are generated by sampling a sparse matrix
\(A\in\mathbb{R}^{m\times n}\), enforcing nonzero rows and columns, injecting
near dependence among selected rows and columns, and then applying diagonal row
and column scalings. A feasible point \(x^{\mathrm{fea}}\in[l,u]\) is planted
and \(b:=Ax^{\mathrm{fea}}\). Within each family we generate
\(1000\) instances, consisting of \(400\) small, \(400\) medium, and \(200\)
large problems with
\[
(m,n,\mathrm{density})\in
\{(150,400,0.045),(250,700,0.025),(400,1000,0.012)\}.
\]
The four synthetic LP families are defined in
Table~\ref{tab:eq_lp_families}. Here \(\alpha_r\) and \(\alpha_c\) are the
logarithmic row and column scaling exponents, \(p_r\) and \(p_c\) are the
fractions of rows and columns used in the dependence step, and
\(\gamma_{\rm dep}\) is the dependence coefficient.

\begin{table}[!htbp]
\centering
\scriptsize
\caption{Synthetic families for the equality-constrained box-LP test set.}
\label{tab:eq_lp_families}
\begin{tabular}{@{}lcccccl@{}}
\toprule
Family & \(\alpha_r\) & \(\alpha_c\) & \(p_r\) & \(p_c\) & \(\gamma_{\rm dep}\) & Interpretation \\
\midrule
Baseline        & 2 & 2 & 0.02 & 0.02 & 0.98   & Mild scaling and weak dependence \\
Ill-scaled      & 5 & 5 & 0.02 & 0.02 & 0.98   & Severe scaling and weak dependence \\
Near-dependent  & 2 & 2 & 0.28 & 0.28 & 0.9995 & Mild scaling and strong dependence \\
Hybrid          & 3 & 3 & 0.18 & 0.18 & 0.995  & Simultaneous scaling and dependence \\
\bottomrule
\end{tabular}
\end{table}

For this synthetic LP class, the main comparison is between \texttt{cuPDLPx}
and \texttt{RHR-CP-AL}. The terminal score is
\(s:=r_{\mathrm{KKT}}\). All methods are run without presolve, with
\(\varepsilon_{\mathrm{opt}}=\varepsilon_{\mathrm{feas}}=10^{-8}\), iteration
limit \(10^5\), and time limit \(60\) seconds. Table~\ref{tab:eq_lp_main}
reports the median terminal score and the number of instances satisfying two
accuracy thresholds. This comparison isolates the effect of equality geometry;
the emphasis is on terminal quality rather than wall-clock time. The
improvement for \texttt{RHR-CP-AL} is most pronounced in the near-dependent
and hybrid families, which are the synthetic cases with the most degenerate
equality geometry.

\begin{table}[!htbp]
\centering
\scriptsize
\caption{Equality-constrained box-LP results on \(1000\) instances per family.
For each family, we report the median terminal score \(s\) and the numbers of
instances satisfying \(s\le 10^{-4}\) and \(s\le 10^{-6}\).}
\label{tab:eq_lp_main}
\begin{tabular}{@{}lcccccc@{}}
\toprule
Family & \multicolumn{2}{c}{Median \(s\)}
& \multicolumn{2}{c}{\(\#\{s \le 10^{-4}\}\)}
& \multicolumn{2}{c}{\(\#\{s \le 10^{-6}\}\)} \\
\cmidrule(lr){2-3}\cmidrule(lr){4-5}\cmidrule(lr){6-7}
 & \texttt{cuPDLPx} & \texttt{RHR-CP-AL}
 & \texttt{cuPDLPx} & \texttt{RHR-CP-AL}
 & \texttt{cuPDLPx} & \texttt{RHR-CP-AL} \\
\midrule
Baseline       & \textbf{\(9.42 \times 10^{-9}\)} & \(9.69 \times 10^{-9}\) & 999 & \textbf{1000} & \textbf{997} & 996 \\
Ill-scaled     & \textbf{\(9.75 \times 10^{-9}\)} & \(6.78 \times 10^{-8}\) & 961 & \textbf{991} & 871 & \textbf{933} \\
Near-dependent & \(3.37 \times 10^{-2}\) & \textbf{\(5.68 \times 10^{-6}\)} & 163 & \textbf{926} & 1 & \textbf{64} \\
Hybrid         & \(1.83 \times 10^{-6}\) & \textbf{\(8.02 \times 10^{-7}\)} & 748 & \textbf{993} & 456 & \textbf{557} \\
\bottomrule
\end{tabular}
\end{table}

The randomized families above are intended as controlled stress tests rather
than as the main solver benchmark. We therefore use the Gurobi-presolved LP
relaxations from MIPLIB~2017~\cite{miplib2017}. Following
\texttt{cuPDLPx}~\cite{cupdlpx}, this benchmark contains \(379\)
instances, split by the number of nonzeros into Small, Medium, and Large
groups. The row-bounded constraints \(\ell_A\le Ax\le u_A\) are handled by
the equivalent slack-variable equality form, with the slack variables
eliminated in the GPU implementation. We use the adaptive version of
\texttt{RHR-CP-AL}. The benchmark files are already Gurobi-presolved; during
these runs, all solvers are called with any additional presolve disabled. We
follow the time-limit convention of \texttt{cuPDLPx}: \(3600\) seconds for the
Small and Medium groups and \(18000\) seconds for the Large group. The
\texttt{RHR-CP-AL} row in Table~\ref{tab:lp_miplib_full} is produced by a
single executable with one automatic policy inside the solver. The policy uses
the same guarded implementation choices throughout the full benchmark:
adaptive restart, reflected Halpern anchoring, restart-point selection, and
safeguarded CP-AL parameter updates. The \texttt{RHR-CP-AL} jobs differ only by
tolerance and ordinary batch chunks; no profile-specific tuning rows are used.

Table~\ref{tab:lp_miplib_full} reports solved counts and \(\mathrm{SGM10}\) time
in seconds for two stopping tolerances. The comparison
includes the CUDA/C baselines \texttt{cuPDLPx(C)}, \texttt{cuPDLP-C}, and
\texttt{HPR-LP-C}, together with the Julia implementations of \texttt{cuPDLP} and
\texttt{HPR-LP}. With additional solver-side presolve disabled as described
above, \texttt{RHR-CP-AL} attains the best \(\mathrm{SGM10}\) in each split and
in the aggregate at both tolerances, while matching the largest total solved count.
The only unsolved instance for \texttt{RHR-CP-AL} is
\texttt{neos-4535459-waipa} in the Large group, which reaches the time limit at
both tolerances.

\begin{table}[htbp]
\centering
\scriptsize
\setlength{\tabcolsep}{3pt}
\caption{Solved counts and \(\mathrm{SGM10}\) time in seconds on
Gurobi-presolved MIPLIB LP relaxations. Additional solver-side presolve is
disabled; unsolved runs are counted at the corresponding time limit. In each
block, boldface in the Count column marks all methods attaining the largest
solved count, while boldface in the Time column marks the smallest
\(\mathrm{SGM10}\) among those methods.}
\label{tab:lp_miplib_full}
\begin{tabular}{@{}llrrrrrrrr@{}}
\toprule
\(\varepsilon\) & Method
& \multicolumn{2}{c}{Small (268)}
& \multicolumn{2}{c}{Medium (93)}
& \multicolumn{2}{c}{Large (18)}
& \multicolumn{2}{c}{Total (379)} \\
\cmidrule(lr){3-4}\cmidrule(lr){5-6}\cmidrule(lr){7-8}\cmidrule(l){9-10}
& & Count & Time & Count & Time & Count & Time & Count & Time \\
\midrule
\(10^{-4}\)
& \texttt{cuPDLPx(C)}
& \textbf{268} & 0.980 & \textbf{93} & 2.699
& \textbf{17} & 15.168 & \textbf{378} & 1.836 \\
& \texttt{cuPDLP-C}
& 259 & 5.787 & 84 & 14.945 & 15 & 38.675 & 358 & 8.633 \\
& \texttt{cuPDLP.jl}
& 258 & 12.647 & 84 & 25.623 & \textbf{17} & 31.457 & 359 & 16.047 \\
& \texttt{HPR-LP.jl}
& 265 & 2.507 & \textbf{93} & 5.138 & \textbf{17} & 16.736 & 375 & 3.588 \\
& \texttt{HPR-LP-C}
& \textbf{268} & 0.781 & \textbf{93} & 2.809 & \textbf{17} & 15.443 & \textbf{378} & 1.715 \\
& \texttt{RHR-CP-AL}
& \textbf{268} & \textbf{0.636} & \textbf{93} & \textbf{2.620}
& \textbf{17} & \textbf{14.947} & \textbf{378} & \textbf{1.550} \\
\midrule
\(10^{-8}\)
& \texttt{cuPDLPx(C)}
& \textbf{268} & 3.442 & \textbf{93} & 9.196
& \textbf{17} & 41.391 & \textbf{378} & 5.635 \\
& \texttt{cuPDLP-C}
& 263 & 15.093 & 86 & 27.358 & \textbf{17} & 69.267 & 366 & 19.220 \\
& \texttt{cuPDLP.jl}
& 248 & 35.049 & 82 & 64.784 & 16 & 143.849 & 346 & 44.080 \\
& \texttt{HPR-LP.jl}
& 264 & 6.976 & 88 & 18.023 & \textbf{17} & 54.677 & 369 & 10.457 \\
& \texttt{HPR-LP-C}
& \textbf{268} & 2.451 & 91 & 8.223 & \textbf{17} & 36.181 & 376 & 4.549 \\
& \texttt{RHR-CP-AL}
& \textbf{268} & \textbf{2.146} & \textbf{93} & \textbf{7.497}
& \textbf{17} & \textbf{36.068} & \textbf{378} & \textbf{4.153} \\
\bottomrule
\end{tabular}
\end{table}

The same solver set is also evaluated on the public \(49\)-instance Mittelmann
LP benchmark used in HPR-LP~\cite{hpr}. All methods are run on the same
\texttt{.mps} files with a
\(1000\)-second time limit and no additional solver-side presolve where the
solver interface exposes this option. The \texttt{RHR-CP-AL} row is produced
by the same unified codebase and executable as Table~\ref{tab:lp_miplib_full},
with the automatic policy selected inside the solver rather than by
profile-specific instance lists. With the common protocol described above,
Table~\ref{tab:lp_mittelmann_full} shows that \texttt{RHR-CP-AL} matches the
largest solved count and gives the smallest \(\mathrm{SGM10}\) at both
tolerances.

\begin{table}[htbp]
\centering
\scriptsize
\caption{Solved counts and \(\mathrm{SGM10}\) time in seconds on the
Mittelmann LP benchmark. The time limit is \(1000\) seconds for every instance;
unsolved and failed runs are counted at the time limit. Boldface in the Count
column marks all methods attaining the largest solved count, while boldface in
the Time column marks the smallest \(\mathrm{SGM10}\) among those methods.}
\label{tab:lp_mittelmann_full}
\begin{tabular*}{\textwidth}{@{\extracolsep{\fill}}lrrrr@{}}
\toprule
\multirow{2}{*}{Method}
& \multicolumn{2}{c}{\(\varepsilon=10^{-4}\)}
& \multicolumn{2}{c}{\(\varepsilon=10^{-8}\)} \\
\cmidrule(lr){2-3}\cmidrule(l){4-5}
& Count & Time & Count & Time \\
\midrule
\texttt{cuPDLPx(C)} & 44 & 14.350 & 41 & 42.249 \\
\texttt{cuPDLP-C} & 36 & 57.160 & 35 & 98.158 \\
\texttt{cuPDLP.jl} & 40 & 49.012 & 33 & 133.659 \\
\texttt{HPR-LP.jl} & 44 & 18.682 & 41 & 60.923 \\
\texttt{HPR-LP-C} & \textbf{47} & 10.780 & \textbf{44} & 31.162 \\
\texttt{RHR-CP-AL} & \textbf{47} & \textbf{8.427} & \textbf{44} & \textbf{28.965} \\
\bottomrule
\end{tabular*}
\end{table}

Table~\ref{tab:lp_miplib_representative} gives instance-level context for the
same final automatic-policy run. It lists representative \(10^{-8}\) cases from
the subset in which \texttt{RHR-CP-AL} has a shorter runtime than both
baselines. These rows are not used for aggregate claims; the aggregate
comparisons in this subsection are based on
Tables~\ref{tab:lp_miplib_full} and~\ref{tab:lp_mittelmann_full}.

\begin{table}[!htbp]
\centering
\scriptsize
\caption{Representative \(10^{-8}\) MIPLIB instances from the final
automatic-policy run. The instances shown are cases in which
\texttt{RHR-CP-AL} has a shorter runtime than both \texttt{cuPDLPx(C)} and
\texttt{HPR-LP-C}. Entries are runtimes in seconds; \texttt{TL} denotes the
split time limit. For aggregate results, see Table~\ref{tab:lp_miplib_full}.}
\label{tab:lp_miplib_representative}
\begin{tabular}{@{}llrrr@{}}
\toprule
Instance & Split & \texttt{cuPDLPx(C)} & \texttt{HPR-LP-C} & \texttt{RHR-CP-AL} \\
\midrule
\texttt{neos-4391920-timok} & Small & 310 & 164 & \textbf{52.9} \\
\texttt{app1-2} & Small & 38.9 & 24.2 & \textbf{11.6} \\
\texttt{irish-electricity} & Small & 10.6 & 59.5 & \textbf{5.69} \\
\texttt{neos-4292145-piako} & Small & 2.89 & 20.4 & \textbf{1.60} \\
\texttt{neos-4413714-turia} & Small & 2780 & 392 & \textbf{277} \\
\texttt{supportcase19} & Medium & 3230 & \texttt{TL} & \textbf{1320} \\
\texttt{neos-3025225-shelon} & Medium & 233 & 246 & \textbf{132} \\
\texttt{ivu06} & Medium & 410 & 319 & \textbf{289} \\
\texttt{square47} & Large & 62.7 & 68.0 & \textbf{52.7} \\
\texttt{nucorsav} & Large & 9.37 & 20.7 & \textbf{8.75} \\
\bottomrule
\end{tabular}
\end{table}

\FloatBarrier

\subsection{Quadratic Programming Experiments}
\label{subsec:qp_experiments}

We consider convex QPs of the form
\begin{equation}\label{eq:eq_qp_test}
\min_{x \in \mathbb{R}^{n}} \ \frac12 x^\top Hx+c^\top x
\quad \mathrm{s.t.}\quad
Ax=b,\qquad l\le x\le u,
\end{equation}
with \(H\succeq0\). The QP experiments have three roles. The first is a
geometry-controlled augmentation test: it identifies regimes where a positive
augmentation parameter in \texttt{RHR-FA-CP} improves over an unaugmented
PDHCG-type baseline. The second is a standard benchmark comparison against
external QP solvers. The third isolates linearization cost. Its subproblem-based
rows are \texttt{PDHCG-II}, \texttt{RHR-CP-AL}, and \texttt{RHR-FA-CP}; its
explicit rows are \texttt{RHR-Lin-PDHG} and \texttt{RHR-Lin-CP-AL}. We do not
report a linearized \texttt{FA-CP} variant. Its linearization still leaves a
nontrivial primal subproblem, whereas the linearized comparison concerns
variants whose main update becomes explicit. The synthetic and linearized
experiments therefore examine the algorithmic effects of augmentation and
explicit updates, while the Maros--Meszaros benchmark compares against
established QP solvers.

The theory in Sections~\ref{sec2}--\ref{sec5} is stated for exact subproblem
maps. In the QP implementation, diagonal Hessian subproblems are solved by
componentwise projection formulas, whereas sparse or low-rank non-diagonal
Hessian subproblems are solved inexactly by the
projected-gradient/Barzilai--Borwein inner solver used in
\texttt{PDHCG-II}~\cite{pdhcgii}. The same inner-solver mechanism is used for
the subproblem-based \texttt{RHR-CP-AL} and \texttt{RHR-FA-CP} rows; in
\texttt{RHR-FA-CP} the inner problem includes the augmented quadratic term.
Accordingly, these QP rows should be interpreted as practical inexact
realizations of the exact subproblem maps analyzed above. The exact-map
convergence guarantees in Sections~\ref{sec2}--\ref{sec5} do not directly cover
these inexact inner-solver implementations. The linearized variants instead use
explicit projected updates and avoid these inner solves.

\subsubsection*{Geometry-driven synthetic QPs}

The main QP experiment is designed around a single geometric question: when
\(H\) provides little curvature in \(\mathrm{range}(A^\top)\) and the equality
matrix \(A\) is ill-conditioned, can the augmented term in \texttt{RHR-FA-CP}
improve an unaugmented PDHCG-type baseline? This is the regime in which a
positive augmentation parameter is expected to help, since the additional
\(A^\top A\)-curvature acts in the weak row-space directions.

All synthetic QPs in this subsection are generated from the same planted-KKT
model. We first sample \(A\), \(H\succeq0\), a primal point \(x^\star\in[l,u]\),
equality multiplier \(y^\star\), and bound multiplier \(s^\star\) satisfying
the sign conditions at the active bounds. We then set
\[
        b=Ax^\star,\qquad
        c=A^\top y^\star-Hx^\star-s^\star .
\]
Thus the generated instance has a controlled KKT point, and different random
families correspond to different parameter regimes of the same construction. To
isolate the augmentation mechanism, we report targeted families in regimes where
the preceding geometry suggests that positive augmentation can be useful:
weak row-space curvature, ill-conditioned or nearly dependent equality
constraints, and nontrivial box activity. We also include one balanced-curvature
ill-scaled family as a control case. For the row-space-weak profile, we
decompose \(\mathbb{R}^n\) into \(\mathrm{range}(A^\top)\) and
\(\mathrm{null}(A)\), draw the eigenvalues of \(H\) from \([10^{-8},10^{-6}]\)
on \(\mathrm{range}(A^\top)\), and draw them from \([10^{-1},1]\) on
\(\mathrm{null}(A)\). The balanced profile uses \([10^{-3},10^{-1}]\) on both
subspaces. The equality matrices are either ill-scaled with weak dependence or
mildly near-dependent with moderate row/column perturbations. The
interior-dominant and moderately active regimes plant about \(1\%\) and \(8\%\)
of the variables at each bound, respectively. We generate \(100\) instances per
reported family, with a \(40/40/20\) split over
\((m,n,\mathrm{density})=(60,160,0.060),(100,240,0.040)\), and
\((160,320,0.025)\).

This experiment keeps the same restarted reflected Halpern wrapper and compares
the unaugmented PDHCG-type baseline with positive-\(\sigma\)
\texttt{RHR-FA-CP} variants.
This isolates the effect of augmentation while keeping the outer acceleration
and residual criterion fixed. The candidate set is
\[
\sigma\in\{10^{-4},10^{-3},10^{-2},10^{-1},3\times10^{-1},1\}.
\]
The tolerance is \(10^{-6}\), and the time limit is \(30\) seconds per
instance.
For each reported family, the selection rule is fixed in advance: maximize the
solved count first and break ties by total runtime. Unsolved runs, if any, are
charged at the time limit when computing total time. Table~\ref{tab:eq_qp_v2b_sigma}
reports the resulting mechanism comparison. The table lists the total
wall-clock time over the \(100\) instances for \(\sigma=0\) and for the
selected positive value of \(\sigma\); the selected \(\sigma\) is shown in
parentheses, and the speedup is the ratio of the two times.

\begin{table}[!htbp]
\centering
\caption{Targeted equality-box QP mechanism experiment comparing
\(\sigma=0\) and a selected positive \(\sigma\) in \texttt{RHR-FA-CP}. Each row
aggregates total runtime over the \(100\) instances in the family.}
\label{tab:eq_qp_v2b_sigma}
\begin{tabular*}{\textwidth}{@{\extracolsep{\fill}}lllccc@{}}
\toprule
\makecell[l]{Curvature} & \makecell[l]{Constraint\\family} & \makecell[l]{Activity\\regime}
& \makecell[c]{\(\sigma=0\)\\time}
& \makecell[c]{\(\sigma>0\) time\\(selected \(\sigma\))}
& Speedup \\
\midrule
\makecell[l]{Balanced} & \makecell[l]{Ill-scaled} & \makecell[l]{Moderately\\active}
& \(337\)s & \(\mathbf{308}\)s \((\sigma=10^{-2})\) & \(\mathbf{1.09\times}\) \\
\makecell[l]{Row-space\\weak} & \makecell[l]{Mildly near-\\dependent} & \makecell[l]{Moderately\\active}
& \(438\)s & \(\mathbf{337}\)s \((\sigma=1)\) & \(\mathbf{1.30\times}\) \\
\makecell[l]{Row-space\\weak} & \makecell[l]{Ill-scaled} & \makecell[l]{Interior-\\dominant}
& \(512\)s & \(\mathbf{489}\)s \((\sigma=3\times 10^{-1})\) & \(\mathbf{1.05\times}\) \\
\bottomrule
\end{tabular*}
\end{table}

\FloatBarrier

Figure~\ref{fig:eq_qp_facp_sigma_advantage} provides two mechanism views. To
avoid mixing two different notions of improvement, the left panel focuses on
the same targeted families as Table~\ref{tab:eq_qp_v2b_sigma}. The right panel
shows a representative row-space-weak, ill-scaled, interior-dominant instance
drawn from the same planted-KKT generator.

In this representative instance, the three runs decrease similarly in the
early phase, but the positive-augmentation runs reach substantially lower
terminal residual levels. At \(12000\) iterations, the displayed residual is
about \(8.5\times10^{-10}\) for \(\sigma=0\), \(4.9\times10^{-11}\) for
\(\sigma=10^{-2}\), and \(9.8\times10^{-12}\) for \(\sigma=0.3\). This
representative trace is not used for aggregate claims; it shows a case where
weak row-space curvature and ill-conditioned equality geometry are favorable to
positive augmentation.

\begin{figure}[H]
\centering
\includegraphics[width=0.70\linewidth]{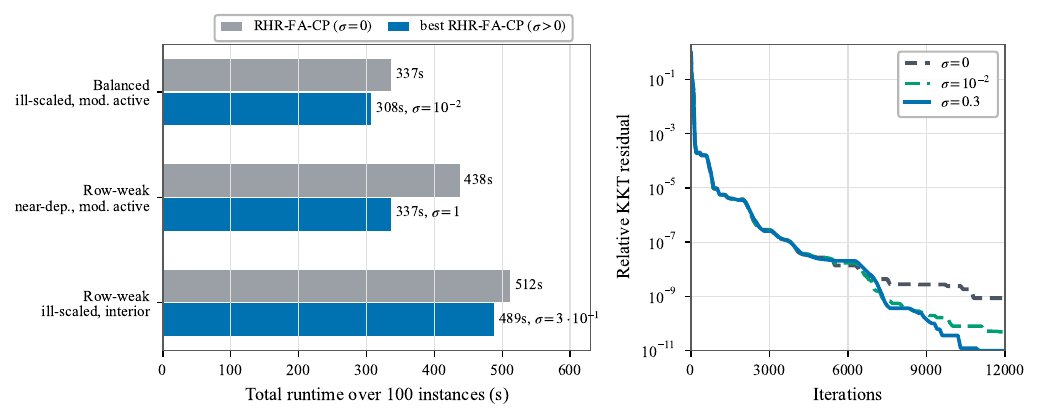}
\captionsetup{font=small}
\caption{Mechanism trace for positive augmentation in the random equality-box
QP experiment. The left panel compares the unaugmented PDHCG-type baseline
with the selected positive-\(\sigma\) \texttt{RHR-FA-CP} variant on the targeted
families at equal solved count; labels at the bar ends report total runtime
and the selected positive \(\sigma\). The right panel shows a
row-space-weak, ill-scaled, interior-dominant instance for zero augmentation,
\(\sigma=10^{-2}\), and \(\sigma=0.3\).}
\label{fig:eq_qp_facp_sigma_advantage}
\end{figure}

\FloatBarrier

\subsubsection*{Maros--Meszaros full benchmark}

We report a standard benchmark comparison on the \(134\)-instance
Maros--Meszaros QP benchmark~\cite{maros1999repository}, matching the full
benchmark protocol used in recent GPU QP solver
comparisons~\cite{pdqp,hprqp,pdhcgii}. For \texttt{RHR-FA-CP}, instances that are not
already in equality-box form are handled through the same lifted-and-eliminated
formulation used by the implementation. We compare \texttt{RHR-FA-CP} with
\texttt{PDQP}~\cite{pdqp}, \texttt{PDHCG}, \texttt{HPR-QP}~\cite{hprqp}, and
the C implementation \texttt{PDHCG-II-C} from~\cite{pdhcgii}. Presolve is
disabled where supported by the wrappers, the tolerance is \(10^{-6}\), and
the time limit is \(1000\) seconds per instance. Failed or unsolved
runs are charged with the time limit when computing \(\mathrm{SGM10}\), the
shifted geometric mean with shift \(10\). Table~\ref{tab:eq_qp_maros134}
summarizes the comparison.

\begin{table}[!htbp]
\centering
\small
\caption{Maros--Meszaros full QP benchmark with presolve disabled where
supported by the solver wrappers. The largest solved count and smallest
\(\mathrm{SGM10}\) are shown in boldface. Runs not solved within the benchmark
protocol are charged at the \(1000\)-second time limit.}
\label{tab:eq_qp_maros134}
\begin{tabular*}{\textwidth}{@{\extracolsep{\fill}}lcc@{}}
\toprule
Method & Solved & \(\mathrm{SGM10}\) \\
\midrule
\texttt{PDQP} & \(118/134\) & \(23.926\) \\
\texttt{PDHCG} & \(111/134\) & \(112.433\) \\
\texttt{HPR-QP} & \(124/134\) & \(16.111\) \\
\texttt{PDHCG-II-C} & \(124/134\) & \(14.242\) \\
\texttt{RHR-FA-CP} & \(\mathbf{125/134}\) & \(\mathbf{12.468}\) \\
\bottomrule
\end{tabular*}
\end{table}

\FloatBarrier

With presolve disabled where supported by the wrappers and the remaining
settings specified above, the same \texttt{RHR-FA-CP} implementation obtains
the best solved count and \(\mathrm{SGM10}\) among the tested QP baselines. The
preceding synthetic experiment gives the complementary geometry-controlled view
of when positive augmentation is useful.

\FloatBarrier

\subsubsection*{Linearized variants}

The third QP regime uses the same planted-KKT generator to test the numerical
meaning of the linearized theory. Among the linearized methods we include
only \texttt{RHR-Lin-PDHG} and \texttt{RHR-Lin-CP-AL}, because these are the
variants for which linearization turns the main step into an explicit update.
The subproblem-based \texttt{PDHCG-II}, \texttt{RHR-CP-AL}, and
\texttt{RHR-FA-CP} rows are included as subproblem-based baselines. A
linearized \texttt{FA-CP} variant is not reported: even after linearization it
still requires solving a nontrivial primal subproblem and therefore does not
isolate the explicit-update mechanism. The aggregate comparison is reported
in Table~\ref{tab:eq_qp_linearized}.

\begin{table}[!htbp]
\centering
\small
\caption{Unified planted-KKT large-scale QP regime for testing explicit
linearized variants (\(1000\) instances). All methods solve all instances;
\(\mathrm{SGM10}\) is the shifted geometric mean of the wrapper-reported
elapsed time in seconds, and the last column reports the median elapsed cost
per \(10^4\) iterations.}
\label{tab:eq_qp_linearized}
\begin{tabular}{@{}lcccc@{}}
\toprule
Method & Solved & Total elapsed & \(\mathrm{SGM10}\) & Median sec. per \(10^4\) iters. \\
\midrule
\texttt{PDHCG-II} & \(1000/1000\) & \(3485\)s & \(3.37\) & \(43.80\) \\
\texttt{RHR-CP-AL} & \(1000/1000\) & \(3546\)s & \(3.43\) & \(37.04\) \\
\texttt{RHR-FA-CP} & \(1000/1000\) & \(2885\)s & \(2.79\) & \(32.79\) \\
\texttt{RHR-Lin-PDHG} & \(1000/1000\) & \(3321\)s & \(3.22\) & \(14.29\) \\
\texttt{RHR-Lin-CP-AL} & \(1000/1000\) & \(\mathbf{2525}\)s & \(\mathbf{2.49}\) & \(\mathbf{12.50}\) \\
\bottomrule
\end{tabular}
\end{table}

\FloatBarrier

The linearized test regime keeps \(A\) well scaled and weakly dependent,
uses sparse moderate-curvature Hessians, and increases the problem sizes to
make the cost of inner subproblem solves visible. We generate \(1000\)
instances in four families: banded or sparse Hessian structure, crossed with
interior-dominant or moderately active bounds. The size split is
\(40/40/20\) over
\[
(m,n,\mathrm{density})\in
\{(300,1200,0.010),(520,2600,0.0045),(650,3600,0.0032)\}.
\]
All methods use tolerance \(10^{-5}\) and a \(60\) second per-instance time
limit. The augmented parameter for \(\texttt{RHR-FA-CP}\) is fixed at
\(\sigma=10^{-3}\). For \(\texttt{RHR-Lin-CP-AL}\), we use the fixed profile
\(\sigma_0=10^{-1}\), reflected coefficient \(0.7\), Pock--Chambolle exponent
\(1.5\), and termination-evaluation frequency \(400\). This profile is selected
once from a small representative screen and then applied to all \(1000\)
instances; it is not tuned instance by instance.

This experiment is not intended to show that explicit linearization is
uniformly preferable across QP geometries. Instead, it identifies the
computational regime targeted by the linearized theory: when the Hessian is
cheap to apply and the primal subproblem solve is no longer negligible, the
explicit update can have much lower per-iteration cost. In this regime, the fixed-profile
\texttt{RHR-Lin-CP-AL} row gives the smallest total elapsed time and
\(\mathrm{SGM10}\), and both RHR-linearized methods reduce the median elapsed
cost per \(10^4\) iterations relative to the subproblem-based methods.

\FloatBarrier

\section{Conclusion}
\label{sec7}

We have developed a unified augmented primal--dual framework for linearly
constrained composite convex problems. The exact augmented scheme admits a
degenerate preconditioned proximal-point representation, and its linearized
counterpart admits a preconditioned forward--backward representation. These two
operator representations provide a common basis for reflected Halpern
acceleration and yield convergence to KKT points together with nonergodic
\(O(1/k)\) bounds for the KKT residual and the objective gap of the shadow
iterates. The scalar example shows that this global residual rate is
worst-case tight.

The local theory identifies the shadow sequence as the correct object for
finite identification in reflected Halpern trajectories. Under an additional
affine-face hypothesis, the identified shadow dynamics admit an exact reduced
residual identity and a reduced Jacobian criterion for local sharpness. The
restart analysis then proves linear convergence of restart anchors under
fixed-point sharpness on the visited restart set. Consequently, the convergence
statement is global when fixed-point sharpness is global, and local or tail
when sharpness is obtained from local error bounds in the positive definite
metric setting.

Several questions remain open. It would be useful to weaken the affine-face
hypothesis while retaining an exact or approximate reduced residual
description, to develop broader polyhedral error-bound or metric-subregularity
arguments for $\ell_1$ and other structured nonsmooth terms, and to extend the
restart verification beyond the positive definite metric setting. Another
direction is to refine implementable inexact variants of the exact augmented
maps without losing the residual interpretation used in the theory.

\clearpage
\appendix

\section{Assumption Map for the Main Results}
\label{app:assumption-map}

Table~\ref{tab:assumption-map} summarizes where the main assumptions enter the
analysis. It is intended only as a navigation aid; the formal assumptions are
those stated in the corresponding sections and theorems.

\begingroup
\begin{table}[!htbp]
\centering
\caption{Assumption map for the main results.}
\label{tab:assumption-map}
\footnotesize
\setlength{\tabcolsep}{3pt}
\renewcommand{\arraystretch}{1.15}
\begin{tabular}{@{}>{\raggedright\arraybackslash}p{0.12\textwidth}
                  >{\raggedright\arraybackslash}p{0.38\textwidth}
                  >{\raggedright\arraybackslash}p{0.24\textwidth}
                  >{\raggedright\arraybackslash}p{0.20\textwidth}@{}}
\toprule
Result & Map and assumptions & Main conclusion & Scope \\
\midrule
Prop.~\ref{prop:dppm_representation}
& exact augmented primal--dual scheme with self-adjoint \(P,Q\); no metric
regularity is needed for the algebraic identity
& exact dPPM representation
& global algebraic identity \\
Thm.~\ref{thm:kkt_gap_rate}
& reflected Halpern dPPM under Assumption~\ref{ass:regularity}, with
\(\mathcal M\) possibly semidefinite and \(W^\star\neq\emptyset\);
\(\gamma\in(-1,1)\) for common state/shadow convergence, while the abstract
residual estimate also permits \(\gamma=1\)
& convergence and \(O(1/k)\) KKT residual and objective gap bounds for shadow
iterates
& global under KKT nonemptiness \\
\makecell[l]{Thm.~\ref{thm:sec32_relaxation};\\
Prop.~\ref{prop:halpern_pfbs_basic};\\
Thm.~\ref{thm:halpern_pfbs_rate}}
& linearized PFBS reflected Halpern under \(\mathcal M\succ0\),
\(\mu_x>L_h/2\), and \(W^\star\neq\emptyset\); the algorithm uses
\(\gamma\in(-1,1-L_h/(2\mu_x))\), with the endpoint reserved for map-level
nonexpansiveness
& nonexpansiveness, convergence, and \(O(1/k)\) KKT residual bounds
& global under KKT nonemptiness \\
Thm.~\ref{thm:finite-identification-shadow-sequence}
& reflected Halpern shadow sequence after selecting either the exact dPPM or
linearized PFBS map; Assumption~\ref{ass:shadow-identification} gives the local
identification conditions
& finite identification of the shadow sequence
& local around the limiting KKT point \\
\makecell[l]{Thm.~\ref{thm:reduced-residual-identity-affine-face};\\
Prop.~\ref{prop:reduced-jacobian-sufficient-condition}}
& post-identification reduced dynamics under the affine-face model in
Assumption~\ref{ass:affine-face-model}
& exact reduced residual identity and reduced sharpness criterion
& local and structural \\
\makecell[l]{Thms.~\ref{thm:fixed_restart_linear}\\
and~\ref{thm:adaptive_restart_linear}}
& fixed-frequency and adaptive restart for a nonexpansive reflected map;
positive definiteness is used where metric transfer to the original KKT
residual is invoked
& epoch contraction of restart anchors under fixed-point sharpness on the
visited set
& global under global fixed-point sharpness; otherwise local or tail \\
Prop.~\ref{prop:box_lp_qp_error_bound}
& box LP/QP verification with \(\mathcal M\succ0\) for transfer from KKT error
bounds to fixed-point sharpness
& local metric subregularity and Hoffman--Robinson error-bound verification
& local or tail after entry into the error-bound neighborhood \\
\bottomrule
\end{tabular}
\end{table}
\endgroup
\FloatBarrier

\section*{Statements and Declarations}

\textbf{Funding.} Zaiwen Wen was supported in part by the National Key
Research and Development Program of China (grant no.~2024YFA1012900) and the
National Natural Science Foundation of China (grant nos.~12331010 and
12288101).

\textbf{Competing interests.} The authors have no competing interests to
declare that are relevant to the content of this article.

\textbf{Data availability}.\ The data generated and analyzed in
the numerical experiments are available from the corresponding author upon
reasonable request.

\bibliographystyle{spmpsci}
\bibliography{z_manuscript} 
\end{document}